\documentclass{article}
\usepackage[utf8]{inputenc}
\usepackage[margin=1 in]{geometry}
\usepackage{amsmath}
\usepackage{amsfonts}
\usepackage{amssymb}
\usepackage{amsthm}
\usepackage{mathrsfs}
\usepackage{enumitem}
\usepackage{hyperref}
\usepackage[vcentermath,enableskew]{youngtab}
\usepackage{ytableau}
\usepackage{xcolor}
\usepackage{graphicx}

\newtheorem{theorem}{Theorem}
\numberwithin{theorem}{subsection}
\newtheorem{lemma}[theorem]{Lemma}
\newtheorem{prop}[theorem]{Proposition}
\newtheorem{corollary}[theorem]{Corollary}
\newtheorem{conjecture}[theorem]{Conjecture}

\newtheorem{question}[theorem]{Question}

\theoremstyle{definition}
\newtheorem{definition}[theorem]{Definition}
\newtheorem{remark}[theorem]{Remark}
\newtheorem{example}[theorem]{Example}

\newcommand{\mb}{\mathbb}
\newcommand{\tb}{\textbf}
\newcommand{\mc}{\mathcal}
\newcommand{\tn}{\textnormal}

\newcommand{\maj}{\tn{maj}}
\newcommand{\qbin}[2]{\begin{bmatrix}{#1}\\ {#2}\end{bmatrix}_q}

\newcommand{\lam}{\lambda}

\newcommand{\tcb}{\textcolor{blue}}
\newcommand{\pr}{\tn{pr}}
\newcommand{\lcm}{\tn{lcm}}
\newcommand{\T}{\tn{\textsf{T}}}
\newcommand{\SYT}{\tn{\textsf{SYT}}}

\usepackage[maxbibnames=10,style=alphabetic,maxalphanames=10]{biblatex}
\title{Orbit lengths for promotion on 2-row and near-hook tableaux}
\author{Laura Pierson \\ \href{mailto:lcpierson73@gmail.com}{lcpierson73@gmail.com}}
\begin{document}

\maketitle

\begin{abstract}
    Promotion has been well-studied for rectangular standard Young tableaux, in which case the orbit lengths divide the total number of boxes and are described by a cyclic sieving phenomenon (CSP), but little is known about the orbit lengths for tableaux of general shape. We approach this problem by building a stable sequence of tableaux where we fix the bottom portion and add extra boxes to the first row to get $n$ total boxes, with $n$ varying. We show that for 2-row tableaux with a fixed bottom row, the orbit lengths are divisors of certain monic polynomials in $n$, with degree generally equal to the number of distinct lengths of runs of consecutive numbers in the bottom row. For the subsets of 2-row tableaux where all runs have the same length, we show that the orbit lengths are characterized by a CSP polynomial that is a slightly modified version of the major index generating function, like in the rectangle case. We also show that for any stable sequence of tableaux, the orbit lengths are linear in $n$ as long as all non-first-row entries differ from each other by at least 2, which asymptotically happens for almost all tableaux in the limit as $n\to\infty.$ We also calculate the orbit lengths for near-hook tableaux, which are divisors of certain linear or quadratic polynomials in $n$.
\end{abstract}

\section{Introduction}

Promotion is an operation on standard Young tableaux introduced by Sch\"utzeberger in \cite{schutzenberger1972promotion} that is closely related to jeu-de-taquin and has been extensively studied for rectangular tableaux. In \cite{haiman1992dual}, Haiman shows that for rectangular and staircase tableaux, the order of promotion (i.e. the number of times it must be repeated to recover the starting tableau) always divides the total number of boxes. Rhoades \cite{rhoades2010cyclic} shows using representation theory that the set of orbit lengths for promotion of rectangular tableaux exhibits the cyclic sieving phenomenon (CSP) of Reiner, Stanton, and White \cite{reiner2004cyclic}. Alternative representation theoretic proofs for that CSP are given by Westbury \cite{westbury2016invariant, westbury2019interpolating}, and geometric proofs are given by Fontaine and Kamnitzer \cite{fontaine2014cyclic} and Purbhoo \cite{purbhoo2013wronskians}. Schank, Sekheri, and Djermane \cite{sekheri2014csp} prove a similar CSP theorem for staircase tableaux. A visual interpretation for the rectangle-case CSP in terms of webs is given by Petersen, Pylyavskyy, and Rhoades \cite{petersen2009promotion} for the 2-row and 3-row cases, and by Gaetz, Pechenik, Pfannerer, Striker, and Swanson \cite{gaetz2023rotation, gaetz2024promotion} for the 4-row case. Catania, Kendrick, Russell, and Tymoczko \cite{catania2025identifying} give an algorithm for determining the exact orbit length for promotion given a specific rectangular tableau.

However, the orbit lengths for promotion on tableaux of other shapes are more complicated and have generally been believed not to have a nice structure, or an associated CSP. For instance, Striker \cite{striker2017dynamical} gives as an example that the order of promotion for tableaux of the 2-row shape (8,6) is 7,554,844,752. However, we show here that the orbit lengths for 2-row tableaux and certain other non-rectangular tableaux do in fact have an understandable structure which we describe precisely, and we also show that there is in fact a CSP closely related to the CSP for rectangular tableaux if one restricts to certain subsets of 2-row tableaux. The large-looking promotion order 7,554,844,752 from that example should really be thought of as the least common multiple of a number of much smaller orbit lengths for various subsets of tableaux, which we will describe precisely in \S\ref{sec:2row}. 

\subsection{Main results and organization}

The remainder of this paper is organized as follows:
\begin{itemize}
    \item In \S\ref{sec:background}, we give background on Young tableaux and promotion (\S\ref{subsec:promotion_background}), and on cyclic sieving (\S\ref{subsec:csp_background}).
    
    \item In \S\ref{sec:generic}, we introduce our general approach here of building a stable sequence of tableaux $T[n]$ by fixing the bottom portion $T$ and adding boxes to the top row to get $n$ total boxes. In Theorem \ref{thm:generic} we show that as long as the differences between the non-first-row entries are all at least 2, the orbit length divides a certain constant multiple of $n-1$, and in Theorem \ref{thm:asymptotic}, we show that in the limit as $n\to\infty$, the fraction of tableaux falling under this ``generic" case approaches 1.
    
    \item In \S\ref{sec:2row_linear}, we prove a CSP theorem (Theorem \ref{thm:2row_linear}) on the subsets of 2-row tableaux of a given shape such that all runs of consecutive numbers in the bottom row are the same length (with some subtleties about runs being allowed to ``wrap around" from $n$ back to the beginning, and about the runs needing to be sufficiently far apart from each other). We show in particular that the orbit lengths divide certain monic linear polynomials in $n$. Then in Theorem \ref{thm:maj_index}, we show that the CSP polynomial almost matches the major index generating function for the tableaux in these subsets.
    
    \item In \S\ref{sec:2row}, we characterize in Theorem \ref{thm:2row} the general orbit lengths for 2-row tableaux, which we show are divisors of monic polynomials in $n$ of degree equal to the number of distinct lengths of runs of consecutive numbers in the bottom row (again with some subtleties about runs being allowed to wrap around, and about runs needing to either be far apart from each other or to be counted in a somewhat non-obvious manner). The main idea is to imagine the runs of each length moving together around a circular track that rotates independently from and at a slightly different speed from the tracks for runs of different lengths.
    
    \item In \S\ref{sec:hook_plus_box}, we show that for near-hooks (i.e. hooks plus one extra box), the orbit lengths are divisors of either linear or quadratic polynomials in $n$, and we characterize the linear polynomials in Proposition \ref{prop:hook_box_linear_1} and Theorem \ref{thm:hook_box_linear_2}, and the quadratic polynomials in Theorem \ref{thm:hook_box_quadratic}. In this case, the polynomials essentially depend on the number of runs of consecutive non-first-row numbers (but not on their exact lengths) and on the number of extra ``singleton" non-first-row numbers that are not part of the runs.
\end{itemize}

\subsection{Future directions}

There are many more questions to ask here. Based on experiments in Sage \cite{sage}, the orbit lengths for the tableaux $T[n]$ built from a fixed starting tableau $T$ are best described not as divisors of fixed polynomials in $n$ (although they are for the cases studied here), but rather as divisors of fixed \emph{\tb{\tcb{quasipolynomials}}} in $n$, meaning piecewise functions of $n$ where the cases depend only on the value of $n$ mod some fixed number, and where the function is a polynomial in each case. We pose this as a conjecture:

\begin{conjecture}
    For every fixed tableau $T$, there exists a quasipolynomial in $n$ such that for $n$ sufficiently large, the orbit length of $T[n]$ always divides that quasipolynomial (provided $T[n]$ is a standard Young tableau).
\end{conjecture}

One could of course still say in such a case that the orbit length is a divisor of a fixed polynomial obtained by multiplying all the cases of the quasipolynomial, but it is more informative to describe the orbit lengths in terms of the quasipolynomial. The overarching question would then be to describe how to compute these polynomials and/or quasipolynomials in general:

\begin{question}
    Given a tableau $T$ such that the orbit lengths $T[n]$ divide some polynomial or quasipolynomial in $n$, how can we compute that (quasi)polynomial?
\end{question}

A related question is whether our CSP results here can be generalized to other cases, and whether they have any sort of algebraic or geometric interpretation:

\begin{question}
    Is there a representation theoretic or geometric interpretation for the CSP results in \S\ref{subsec:CSP}? Relatedly, are there other tableaux $T$ where the orbit lengths are described by a CSP polynomial, and does that CSP polynomial have such an interpretation?
\end{question}

\section{Background}\label{sec:background}

\subsection{Young tableaux and promotion}\label{subsec:promotion_background}

A \emph{\tb{\tcb{partition}}} $\lam = \lam_1\dots\lam_\ell$ is a nonincreasing sequence of positive integers $\lam_1\ge \dots \ge \lam_\ell.$ We call $\ell = \ell(\lam)$ is its \emph{\tb{\tcb{length}}}, and $|\lam|:=\lam_1+\dots+\lam_\ell$ is its \emph{\tb{\tcb{size}}}. For $n> |\lambda|+\lambda_1$, we write $\lambda[n]$ for the partition obtained from $\lambda$ by appending a new largest part of size $n-|\lambda|$ to get a partition of $n$. 

A \emph{\tb{\tcb{Young diagram}}} of shape $\lam$ consists of $\ell$ left-justified rows of boxes of lengths $\lam_1,\dots,\lam_\ell$ from top to bottom. A \emph{\tb{\tcb{Young tableau}}} $T$ of shape $\lam$ is obtained by putting a positive integer in each box of the Young diagram for $\lam$, and we write $|T|=|\lam|$ for the number of boxes in $T$. A \emph{\tb{\tcb{semistandard Young tableau (SSYT)}}} is one in which the entries weakly increase from left to right across each row and from top to bottom down each column. A \emph{\tb{\tcb{standard Young tableau (SYT)}}} $T$ is an SSYT where the entries are $1,2,\dots,|T|$ in some order. We write $f^\lam$ for the number of SYT of shape $\lam$.

For $T$ an SSYT with all distinct entries $i_1 < i_2 < \dots < i_{|T|}$, we write $\SYT(T)$ for the corresponding SYT formed from $T$ by replacing $i_k$ with $k$ for each $k$. We also write $T[n]$ for the tableau formed from $T$ by appending above $T$ a top row of length $n-|T|$ containing all the entries among $\{1,2,\dots,n\}$ that are not entries of $T$. We will generally be interested in cases where $T[n]$ is an SYT. 

For an SSYT $T$ with all distinct entries $i_1<i_2<\dots<i_{|T|}$ (which may or may not be an SYT), we define the \emph{\tb{\tcb{promotion}}} action on $T$ as follows:
\begin{enumerate}
    \item Remove the smallest (top left) entry, leaving an empty box.
    \item At each subsequent step such that there is still a box directly to the right of or below the current empty box, slide the smaller of those two entries (or the one that exists, if there is only one) into the current empty box, leaving a new empty box. (These are \emph{\tb{\tcb{jeu-de-taquin}}} slides.)
    \item Once the empty box has no boxes below it or to its right, replace each entry $i_k$ with $i_{k-1}$ for $2\le k\le |T|$.
    \item Place $i_{|T|}$ in the new empty box.
\end{enumerate}
(Some sources define promotion to be the inverse of this action.) We write $\mc{P}(T)$ for the tableau obtained by doing promotion on $T$, and $\pr(T)$ for the period of the promotion action on $T$, i.e. the minimum number of promotion steps required to recover the original tableau $T$. Note that $\SYT(\mc{P}(T))=\mc{P}(\SYT(T))$, since we can equivalently do promotion on $T$ by first replacing $i_k$ with $k$ for every $k$ to get $\SYT(T)$, then doing a promotion step on $\SYT(T)$, and then turning the entry $k$ back into $i_k$. Thus $\pr(T) = \pr(\SYT(T)).$

\begin{example}
    For the starting tableau $T$ shown below, we would compute $\mc{P}(T)$ through the steps shown:
    \begin{center}
        \ytableausetup{boxsize=0.5cm}
        $T = $\begin{ytableau}
            2 & 5 & 9 \\
            6 & 7
        \end{ytableau}
        $\rightarrow$
        \begin{ytableau}
            {} & 5 & 9 \\
            6 & 7 \\
        \end{ytableau}
        $\rightarrow$
        \begin{ytableau}
            5 &  & 9 \\
            6 & 7 \\
        \end{ytableau}
        $\rightarrow$
        \begin{ytableau}
            5 & 7 & 9 \\
            6 & \\
        \end{ytableau}
        $\rightarrow$
        \begin{ytableau}
            2 & 6 & 7 \\
            5 &  \\
        \end{ytableau}
        $\rightarrow$
        \begin{ytableau}
            2 & 6 & 7 \\
            5 & 9 \\
        \end{ytableau} $= \mc{P}(T).$
        \end{center}
        The corresponding promotion action on $\SYT(T)$ would be computed as shown:
        \begin{center}
        $\SYT(T) = $\begin{ytableau}
            1 & 2 & 5 \\
            3 & 4
        \end{ytableau}
        $\rightarrow$
        \begin{ytableau}
            {} & 2 & 5 \\
            3 & 4 \\
        \end{ytableau}
        $\rightarrow$
        \begin{ytableau}
            2 &  & 5 \\
            3 & 4 \\
        \end{ytableau}
        $\rightarrow$
        \begin{ytableau}
            2 & 4 & 5 \\
            3 & \\
        \end{ytableau}
        $\rightarrow$
        \begin{ytableau}
            1 & 3 & 4 \\
            2 &  \\
        \end{ytableau}
        $\rightarrow$
        \begin{ytableau}
            1 & 3 & 4 \\
            2 & 5 \\
        \end{ytableau} $= \mc{P}(\SYT(T)$.
        \end{center}
        The full promotion orbit for $T$ is
        \begin{center}
            $T = $\begin{ytableau}
                2 & 5 & 9 \\ 
                6 & 7
            \end{ytableau}
            $\rightarrow \mc{P}(T) = $
            \begin{ytableau}
                2 & 6 & 7 \\ 
                5 & 9
            \end{ytableau}
            $\rightarrow \mc{P}^2(T) =$
            \begin{ytableau}
                2 & 5 & 6 \\ 
                7 & 9
            \end{ytableau}
            $\rightarrow \mc{P}^3(T) = $
            \begin{ytableau}
                2 & 5 & 9 \\ 
                6 & 7
            \end{ytableau}
            $ = T,$
        \end{center}
        and the corresponding orbit for $\SYT(T)$ is
        \begin{center}
            $\SYT(T) = $
            \begin{ytableau}
                1 & 2 & 5 \\ 
                3 & 4
            \end{ytableau}
            $\rightarrow$
            \begin{ytableau}
                1 & 3 & 4 \\ 
                2 & 5
            \end{ytableau}
            $\rightarrow$
            \begin{ytableau}
                1 & 2 & 3 \\ 
                4 & 5
            \end{ytableau}
            $\rightarrow$
            \begin{ytableau}
                1 & 2 & 5 \\ 
                3 & 4
            \end{ytableau} $ = \SYT(T)$.
        \end{center}
         Since $\mc{P}^3(T) = T$, we have $\pr(T) = \pr(\SYT(T)) = 3.$
\end{example}

\subsection{Cyclic sieving and \texorpdfstring{$q$}{q}-analogues}\label{subsec:csp_background}

Given an action of a cyclic group $\mb{Z}_n=\langle g\rangle$ on a set $X$, and a polynomial $f(q)$, we say that the triple $$(X, \ \mb{Z}_n=\langle g\rangle,\ f(q))$$ exhibits the \emph{\tb{\tcb{cyclic sieving phenomenon}}} or is a \emph{\tb{\tcb{CSP triple}}} if for each $k$, $$|\{x\in X \mid g^k(x) = x\}| =f(\zeta_n^k)$$ whenever $\zeta_n$ is a primitive $n^\text{th}$ root of unity. Equivalently, for each $d\mid n,$ we need $$|\{x\in X \mid g^{n/d}(x) = x\}| = f(\zeta_d)$$ whenever $\zeta_d$ is a primitive $d^\text{th}$ root of unity. We call $f(q)$ the \emph{\tb{\tcb{CSP polynomial}}} for the action of $\mb{Z}_n = \langle g\rangle$ on $X$. CSP polynomials are often related to \emph{\tb{\tcb{$\boldsymbol{q}$-analogues}}} $$[n]_q := 1+q+\dots+q^{n-1} = \frac{q^n-1}{q-1}.$$ Many common combinatorial polynomials also have $q$-analogues. We will use the \emph{\tb{\tcb{$\boldsymbol{q}$-binomial coefficients}}} $$\qbin{n}{k} := \frac{[n]_q[n-1]_q\dots[n-k+1]_q}{[k]_q[k-1]_q\dots[1]_q}.$$

\section{Generic case linear orbits for arbitrary Young diagrams}\label{sec:generic}

In this section, we consider any stable sequence of tableaux shapes $\lam[n]$ formed by fixing the bottom portion $\lam$ and adding boxes to the first row. We will show that asymptotically as $n\to\infty,$ $100\%$ of tableaux of shape $\lam[n]$ fall into ``generic case" orbits whose length is a divisor of $c\cdot (n-1)$, where $c$ is one of a fixed set of constants depending on $\lam$ but not on $n$.

\subsection{Generic case orbit lengths}\label{subsec:generic_length}

\begin{theorem}\label{thm:generic}
    Suppose $T[n]$ is a standard Young tableau such that all entries of $T$ differ by at least 2 from each other, and 2 and $n$ are not both entries of $T$. Then $\pr(T[n])$ is a divisor of $$\frac{\lcm(|T|,\pr(\SYT(T)))}{|T|}\cdot(n-1).$$
\end{theorem}

\begin{proof}
    We begin with the following lemma:

    \begin{lemma}\label{lem:generic}
        Throughout the promotion process:
        \begin{enumerate}[label=(\arabic*)]
            \item The sequence of mod $n-1$ gaps between consecutive entries of $T$ always stays the same up to cyclic shifts, where the smallest entry of $T$ minus the largest entry counts as one of those gaps.
            \item An entry of $T$ only gets promoted into the top row of $T[n]$ if it is a 2 in the top left box of $T$.
        \end{enumerate} 
    \end{lemma}

    \begin{proof}
        At steps where no entry of $T$ gets promoted into the top row, all entries of $T$ simply decrease by 1, so the sequence of gaps is preserved, and (1) holds. Now we can assume inductively that (1) has always held up to some particular tableau $T[n]$ in the promotion orbit, and it suffices to show that (1) and (2) both hold for the next promotion step.

        Our starting assumptions on $T$ imply that all gaps between entries of $T$, taken as mod $n-1$ values between 0 and $n-1$, are at least 2. This includes the mod $n-1$ value of the smallest minus largest entry. To see that, note that if the smallest entry is $j$ and the largest is $n-k$ for $j,k\ge 0$,  then their mod $n-1$ difference is $$j-(n-k)= j+k-n\equiv j+k-1\pmod{n-1}.$$ Now, $j+k-1\ne 0$ since we must have $j\ge 2$, as the 1 must be in the top row of $T[n]$ and hence not in $T$. It also cannot equal 1 because our assumption that 2 and $n$ are not both in $T$ implies that if $j=2$ then $k\ge 1$, hence $j+k-1\ge 2$. Now since (1) has held at every step so far by assumption, the mod $n-1$ gaps between entries of $T$ are the same as at the beginning of the orbit, so they are also all at least 2. 
        
        To prove (2), suppose some entry of $T$ gets promoted, and assume for contradiction it is the $k^\text{th}$ entry of the top row of $T$ for some $k\ge 2$. Since the top left entry of $T$ does not get promoted, it must be at least 3, and that top left entry is the smallest entry of $T$, so since consecutive entries of $T$ differ by at least 2, the $k^\text{th}$ entry in the top row of $T$ (shown in yellow below) is at least $3 + 2(k-1) = 2k + 1$. 
        
        Also, of the $2k$ numbers $1,2,\dots,2k$, at most $k-1$ of them can be in $T$, since the gaps of at least 2 between entries of $T$ mean we can at most have $3,5,7,\dots,2k-1$ in the starting tableau. Thus, at least $k+1$ of the numbers $1,2,\dots,2k$ are in the top row of $T[n]$, so the $(k+1)^{\text{st}}$ entry of the top row of $T[n]$ is at most $2k$. But then the $(k+1)^{\text{st}}$ entry of the top row of $T[n]$ is less than the $k^\text{th}$ entry of the top row of $T$:
        \begin{center}
        \ytableausetup{boxsize=1cm}
        \begin{ytableau}
            1 & 2 & & \cdots & & \le 2k & \cdots \\
            \ge 3 & \ge 5 & \ge 7 & \cdots & *(yellow) \scriptstyle{\ge 2k+1} & \cdots \\
            \vdots & \vdots
        \end{ytableau}
        \end{center}
        Thus, that $k^\text{th}$ entry in the second row will not actually get promoted. This is a contradiction, so in fact only a 2 can get promoted from $T$ into the top row, proving (2).

        It remains to show that (1) still holds in steps where an entry of $T$ gets promoted into the top row. From (2), that entry must be a 2, so in the next step $T$ loses a 2 and gains an $n$. Since $n\equiv 1\pmod{n-1}$, the 2 decreases by 1 (mod $n-1$). But all other entries of $T$ also decrease by 1 (mod $n-1$), so the mod $n-1$ gaps between them stay the same, proving (2).
    \end{proof}

    Now to prove Theorem \ref{thm:generic}, assume without loss of generality that the top left entry of the starting tableau $T_0[n]$ is a 2 (which we may do since if not it will become a 2 after some number of promotion steps). Then $\pr(T_0[n])$ is the minimum number of promotion steps until we reach a tableau $T[n]$ with the following properties:
    \begin{enumerate}[label=(\arabic*)]
        \item The top left entry of $T$ is a 2.
        \item The gaps between consecutive elements of $T$, starting with the gap from the 2 to the next smallest entry, are the same as for $T_0$.
        \item $\SYT(T) = \SYT(T_0)$.
    \end{enumerate}
    To see how often each of (1), (2), and (3) happens, we can imagine labeling the entries $i_1<i_2<\dots<i_{|T|}$ of $T_0$ with the numbers $1,2,\dots,|T|$, respectively. Then we can correspondingly label the entries of each subsequent tableau $T$ in the orbit, such that if an entry $i\ge 3$ of $T$ had label $k$, then $i-1$ has label $k$ in $\mc{P}(T[n])$, and if the 2 in $T$ has label $k$, then the $n$ has label $k$ in $\mc{P}(T[n])$. Note that the labeled entries in $T[n]$ will always be precisely the entries of $T$, since if $i\ge 3$ is in $T$, then $i-1$ will be in $T$ in the next step since only a 2 can get promoted out of $T$, and if $i=2$ is in $T$, it gets replaced by an $n$ in $T$.

    At steps where no entry of $T$ gets promoted into the first row, $\SYT(T)$ does not change since all entries of $T$ simply decrease by 1. At steps where the top left entry of $T$ is a 2 that gets promoted, a promotion step is performed on $\SYT(T)$. Thus, among steps where (1) holds, (3) holds once every $\pr(T)=\pr(T_0)$ of those steps, since that is when $\SYT(T)$ will have gone through a full promotion orbit and will be back to its starting configuration.

    Let $g_1,g_2,\dots,g_{|T|}$ be the sequence of mod $n-1$ gaps between the consecutive entries of $T_0$ when listed in increasing order, starting with $g_1$ as the gap between the 2 and the second smallest entry in $T_0$, and ending with $g_{|T|}$ as the smallest entry of $T_0$ minus the largest entry (mod $n-1$). The second time we reach a state where (1) holds, $g_2$ will now be the gap from the 2 to the next smallest entry, so the sequence of gaps will have cyclically shifted one position and become $g_2,g_3,\dots,g_{|T|},g_1$. Similarly, each time we return to a state where (1) holds, the sequence of gaps cyclically shifts by one position. Thus, after $|T|$ steps where (1) holds, we will be back to the starting sequence of gaps, and (2) will also hold.

    Note next that there are $|T|$ steps where (1) holds for every $n-1$ total promotion steps, because each of the $|T|$ labeled entries will decrease down to 2 before we get back to the entry labeled 1 being the 2. After $$\frac{\lcm(|T|,\pr(T))}{|T|}\cdot(n-1)$$ total promotion steps, we will have completed $\lcm(|T|,\pr(T))/|T|$ of these rounds that each include $|T|$ steps satisfying (1), for a total of $\lcm(|T|,\pr(T))$ steps satisfying (1). Since the total number of promotion steps will be a multiple of $n-1$, (1) will be satisfied. (2) will also be satisfied because the number of steps so far satisfying (1) will be a multiple of $|T|$, and (3) will hold because the number of steps satisfying (1) will be a multiple of $\pr(T)$. Since all three conditions hold, we will be back at the starting tableau $T[n]$ after $\lcm(|T|,\pr(T))/|T|\cdot(n-1)$ steps, so that number of steps must be a multiple of $\pr(T[n])$.
\end{proof}

    \begin{example}
        \ytableausetup{boxsize=0.4cm}
        Let $T=$ \begin{ytableau}
            4 & 6 \\
            8
        \end{ytableau} and $n=8,$ so $\SYT(T) = $ \begin{ytableau}
            1 & 2 \\
            3
        \end{ytableau} and $$T[8] = \begin{ytableau}
            1 & 2 & 3 & 5 & 7 \\
            4 & 6 \\
            8
        \end{ytableau}.$$ Then $\mc{P}(\SYT(T)) =$ \begin{ytableau}
            1 & 3 \\ 2
        \end{ytableau}
        and $\mc{P}(\mc{P}(\SYT(T))$ = \begin{ytableau}
            1 & 2 \\ 3
        \end{ytableau} $= \SYT(T)$, so $\pr(\SYT(T)) = 2$. So we should have $$\pr(T[8]) \mid \frac{\lcm(3,2)}{3}\cdot(8-1) = 2\cdot 7 = 14.$$
        If we track the entry that starts as the 8 (shown in orange), it will go through 2 full cycles during the orbit, going from 8 down to 2. During each of those cycles, there are 3 steps (shown in yellow) where an entry of $T$ becomes 2 and $\SYT(T)$ changes in the next step. Thus, $\SYT(T)$ goes through 6 promotion steps over the orbit, so it ends up back at its starting state since $\pr(\SYT(T))=2$.
        \begin{center}
        \ytableausetup{boxsize=0.4cm}
        \begin{tabular}{ccccccccc}
            &\begin{ytableau}
                1 & 2 & 3 & 5 & 7 \\
                4 & 6 \\
                *(orange) 8
            \end{ytableau} &
            \begin{ytableau}
                1 & 2 \\
                3
            \end{ytableau}
            & $\longrightarrow$ &
            \begin{ytableau}
                1 & 2 & 4 & 6 & 8 \\
                3 & 5 \\
                *(orange) 7
            \end{ytableau} &
            \begin{ytableau}
                1 & 2 \\
                3
            \end{ytableau}
            & $\longrightarrow$ &
            \begin{ytableau}
                1 & 3 & 5 & 7 & 8 \\
                *(yellow) 2 & *(yellow) 4 \\
                *(orange) 6
            \end{ytableau} &
            \begin{ytableau}
                *(yellow) 1 & *(yellow) 2 \\
                *(yellow) 3
            \end{ytableau} \\ \\

            $\longrightarrow$ &
            \begin{ytableau}
                1 & 2 & 4 & 6 & 7 \\
                3 & 8 \\
                *(orange) 5
            \end{ytableau} &
            \begin{ytableau}
                1 & 3 \\
                2
            \end{ytableau} &
            $\longrightarrow$ &
            \begin{ytableau}
                1 & 3 & 5 & 6 & 8 \\
                *(yellow) 2 & *(yellow) 7 \\
                *(orange) 4
            \end{ytableau} &
            \begin{ytableau}
                *(yellow) 1 & *(yellow) 3 \\
                *(yellow) 2
            \end{ytableau} &
            $\longrightarrow$ &
            \begin{ytableau}
                1 & 2 & 4 & 5 & 7 \\
                *(orange) 3 & 6 \\
                8
            \end{ytableau} &
            \begin{ytableau}
                1 & 2 \\
                3
            \end{ytableau} \\ \\

            $\longrightarrow$ &
            \begin{ytableau}
                1 & 3 & 4 & 6 & 8 \\
                *(orange) 2 & *(yellow) 5 \\
                *(yellow) 7
            \end{ytableau} &
            \begin{ytableau}
                *(yellow) 1 & *(yellow) 2 \\
                *(yellow) 3
            \end{ytableau} &&&& \\ \\
            $\longrightarrow$ &
            \begin{ytableau}
                1 & 2 & 3 & 5 & 7 \\
                4 & *(orange)8 \\
                6
            \end{ytableau} &
            \begin{ytableau}
                1 & 3 \\
                2
            \end{ytableau} &
            $\longrightarrow$ &
            \begin{ytableau}
                1 & 2 & 4 & 6 & 8 \\
                3 & *(orange)7 \\
                5
            \end{ytableau} &
            \begin{ytableau}
                1 & 3 \\
                2
            \end{ytableau}
            & $\longrightarrow$ &
            \begin{ytableau}
                1 & 3 & 5 & 7 & 8 \\
                *(yellow) 2 & *(orange) 6 \\
                *(yellow) 4
            \end{ytableau} &
            \begin{ytableau}
                *(yellow) 1 & *(yellow) 3 \\
                *(yellow) 2
            \end{ytableau} \\ \\
    
            $\longrightarrow$ &
            \begin{ytableau}
                1 & 2 & 4 & 6 & 7 \\
                3 & *(orange) 5 \\
                8
            \end{ytableau} &
            \begin{ytableau}
                1 & 2 \\
                3
            \end{ytableau} &
            $\longrightarrow$ &
            \begin{ytableau}
                1 & 3 & 5 & 7 & 8 \\
                *(yellow) 2 & *(orange) 4 \\
                *(yellow) 7
            \end{ytableau} &
            \begin{ytableau}
                *(yellow) 1 & *(yellow) 2 \\
                *(yellow) 3
            \end{ytableau}
            & $\longrightarrow$ &
            \begin{ytableau}
                1 & 2 & 4 & 6 & 7 \\
                *(orange) 3 & 8 \\
                6
            \end{ytableau} &
            \begin{ytableau}
                1 & 3 \\
                2
            \end{ytableau} \\ \\
            
            $\longrightarrow$ &
            \begin{ytableau}
                1 & 3 & 5 & 6 & 8 \\
                *(orange) 2 & *(yellow) 7 \\
                *(yellow) 5
            \end{ytableau} &
            \begin{ytableau}
                *(yellow) 1 & *(yellow) 3 \\
                *(yellow) 2
            \end{ytableau} &&&&&
        \end{tabular}
    \end{center}
    \end{example}

The reason why $\pr(T[n])$ could be a proper divisor of $\lcm(|T|,\pr(T))/|T|\cdot(n-1)$ instead of being equal to it is that the sequence of gaps $g_1,g_2,\dots,g_{|T|}$ could have $d$-fold symmetry for some $d$ and thus return to its starting configuration after some smaller number $|T|/d$ of cyclic shifts. In that case, for (2) to hold we only need the total number of steps where (1) holds to be a multiple of $|T|/d$ instead of a multiple of $|T|$. Since $g_1+g_2+\dots+g_{|T|}=n-1$, the $d$-fold symmetry implies that $g_1+\dots+g_{|T|/d} = (n-1)/d$. Thus, (2) will hold every $(n-1)/d$ steps, and each of those sequences of $(n-1)/d$ steps will include $|T|/d$ steps where (1) holds. For (3), we still need the total number of steps where (1) holds to be a multiple of $\pr(T)$. Thus, the number of steps where (1) holds must be at least $\lcm(|T|/d,\pr(T))$. That sequence of steps can then be split into a whole number of rounds each containing $|T|/d$ steps where (1) holds and $(n-1)/d$ steps total, so we get $(n-1)/d$ total steps per $|T|/d$ steps where (1) holds, and thus $(n-1)/|T|$ times as many total steps as steps where (1) holds. Thus, the promotion period in that case will instead be $$\frac{\lcm(|T|/d,\pr(T))}{|T|}\cdot(n-1),$$ which is smaller than the number of steps in Theorem \ref{thm:generic} if and only if $|T|/d$ and $\pr(T)$ have a smaller common multiple than $|T|$ and $\pr(T)$ do.

\subsection{Asymptotics}\label{subsec:asymptotics}

The promotion orbits described in Theorem \ref{thm:generic} are ``generic" in the following sense:

\begin{theorem}\label{thm:asymptotic}
    If we fix the shape $\lambda$ of $T$ and take a limit as $n\to\infty$, the fraction of tableaux of shape $\lambda[n]$ falling under Theorem \ref{thm:generic} approaches 1, so asymptotically almost all tableaux fall under this ``generic" case.
\end{theorem}

\begin{proof}
    From the hook length formula, the total number of tableaux of shape $\lambda[n]$ is $$f^{\lambda[n]} = \frac{n!}{(\tn{hook lengths within }\lambda)\cdot(\tn{hook lengths directly above }\lambda)\cdot(\tn{hook lengths in rest of first row})}.$$ By the hook length formula on $\lambda,$ the product of the hook lengths within $\lambda$ is $|\lambda|!/f^\lambda$. For $1\le i\le \lambda_i$, the length of the $i^\text{th}$ hook starting directly above $\lambda$ is $n-|\lambda|-i+1+\lambda_i^t$ where $\lambda^t$ is the transpose of $\lambda$. This is because the first row has total length $n-|\lambda|$, hence the $i^\text{th}$ hook includes $n-|\lambda|-i+1$ boxes from the first row plus $\lambda_i^t$ more boxes in its column. For the remaining $n-|\lambda|-\lambda_1$ boxes in the first row, the product of the hook lengths is just $(n-|\lambda|-\lambda_1)!$. Putting this together gives $$f^{\lambda[n]} = \frac{n(n-1)\dots(n-|\lambda|-\lambda_1+1)}{|\lambda|!/f^\lambda\cdot \prod_{i=1}^{\lambda_1}(n-|\lambda|-i+1+\lambda_i^t)} = \frac{f^\lambda}{|\lambda|!}\cdot n^{|\lambda|}+O(n^{|\lambda|-1}).$$ The last equality follows because the linear factors in the denominator cancel out $\lambda_1$ of the $|\lambda|+\lambda_1$ linear factors in the numerator, meaning there are $|\lambda|$ linear factors left, which multiply to give a polynomial of degree $|\lambda|$.

    Now, the tableaux satisfying the conditions in Theorem \ref{thm:generic} are the ones where all entries in $T$ differ from each other by at least 2, and also 2 and $n$ are not both in $T$. We can build such a tableau by first choosing which entries to put in $T$, and then choosing how to arrange them within $T$. 
    
    For choosing the entries, we can enforce the conditions of the smallest entry being at least 2 and each subsequent entry being at least 2 more than the previous entry by adding a buffer of size 1 before each entry, so we are arranging $|\lambda|$ blocks of size 2 and $n-2|\lambda|$ singletons, which can be done in $\binom{n-|\lambda|}{|\lambda|}$ ways. 
    
    However, we then need to subtract cases where both 2 and $n$ are in $T$. In that case, $|\lambda|-2$ more entries need to be chosen from among $4,5,\dots,n-2$ (so from among $n-5$ numbers) such that consecutive ones differ by at least 2. We can again enforce the differences by making each entry except the first into a block of 2 by adding a buffer of size 1 before it. We are thus arranging $|\lambda|-2$ blocks of size 2 and $n-5-2(|\lambda|-2)=n-2|\lambda|-1$ singletons, which can be done in $\binom{n-|\lambda|-3}{|\lambda|-2}$ ways. Putting this together, the number of ways to choose the entries is $$\binom{n-|\lambda|}{|\lambda|}-\binom{n-|\lambda|-3}{|\lambda|-2}.$$
    
    Now, the number of ways to arrange these entries within $T$ is just $f^\lambda$. To see that, it suffices to check that we will never have an entry in the top row of $T$ which is larger than the entry above it. Because of our condition on the entries in $T$, the $k^\text{th}$ entry in the top row of $T$ is at least $2k$ for every $k$. For the entry above it, i.e. the $k^\text{th}$ entry in the top row of $T[n]$, note that among the $2k-1$ numbers $1,2,\dots,2k-1$, we can at most put $2,4,\dots,2k-2$ in $T$, so at most $k-1$ of them can be in $T$. Thus, at least $k$ of them are in the top row of $T[n]$, so the $k^\text{th}$ entry of the top row of $T[n]$ is always at most $2k-1$ and hence is automatically less than the entry below it. Thus, every way of arranging the chosen entries within $T$ such that $T$ is a semistandard Young tableau will result in $T[n]$ being a standard Young tableau.

    It follows that the number of tableaux satisfying the conditions of Theorem \ref{thm:generic} is $$f^\lambda\cdot\left(\binom{n-|\lambda|}{|\lambda|}-\binom{n-|\lambda|-3}{|\lambda|-2}\right) = \frac{f^\lambda}{|\lambda|!}\cdot n^{|\lambda|}+O(n^{|\lambda|-1}).$$ Thus, the total number of standard Young tableaux of shape $\lambda[n]$ and the number of tableaux satisfying Theorem \ref{thm:generic} are both polynomials in $n$ with the same degree $|\lambda|$ and the same leading coefficient $f^{\lambda}/|\lambda|!$, so in the limit as $n\to\infty$, their ratio is 1, as claimed.
\end{proof}

\section{Cyclic sieving on 2-row Young tableaux with linear orbit lengths}\label{sec:2row_linear}

For 2-row Young tableaux, we will show that for a fixed $T$ and varying $n$, $\pr(T[n])$ is always a divisor of a certain monic polynomial in $n$, and we will describe how to compute those polynomials in general in \S \ref{sec:2row}. In this section, we will describe the cases when that polynomial is linear. Essentially, it is linear when the runs of consecutive numbers in $T$ all have the same length, as long as those runs are sufficiently far apart. We will also show that for the subsets of tableaux corresponding to a fixed run length $\ell$ and number of runs $r$, the orbit lengths are described by a CSP polynomial which is a slightly modified version of the major index generating function.

\subsection{Cyclic sieving theorem}\label{subsec:CSP}

\begin{definition}\label{def:one_length}
    For $n\ge (2r+1)\ell$, let $\T(n,\ell,r)$ be the set of standard Young tableaux $T[n]$ such that $T$ consists of a single row, and the entries of $T$ can be partitioned into $r$ \emph{\tb{\tcb{runs}}} of exactly $\ell$ consecutive numbers each, including possibly a \emph{\tb{\tcb{wraparound run}}} consisting of the numbers $k+1,k+2,\dots,2k$ together with $n-\ell+k+1,\dots,n$ for some $1\le k\le \ell-1$. We also require that if we subtract the starting number of each run from the starting number of the next run and take the remainder when that difference is divided by $n-\ell$, that remainder must be at least $2\ell$.
\end{definition}

If there is no wraparound run, the smallest run is counted as the ``next" run after the largest run, and those two runs are also required to satisfy the condition on the mod $n-\ell$ differences. If there is a wraparound run, the $n-\ell+k+1$ counts as its smallest element, and the wraparound run is still required to satisfy the mod $n-\ell$ difference requirement, with the largest non-wraparound run counting as the run before it, and the smallest non-wraparound run counting as the run after it.

\begin{theorem}\label{thm:2row_linear}
     For every $T[n]\in \T(n,\ell,r)$, $\tn{pr}(T[n])$ is a divisor of $n-\ell$. Furthermore, $\T(n,\ell,r)$ is closed under promotion, and $$\left(\T(n,\ell,r), \ \ \mb{Z}_{n-\ell}=\langle \mc{P}\rangle, \ \ \frac{[n-\ell]_q}{[r]_q}\qbin{n-\ell-2r\ell+r-1}{r-1}\right)$$ is a CSP triple with respect to the action of promotion on $\T(n,\ell,r)$.
\end{theorem}

\begin{proof}
    Like we saw in Theorem \ref{thm:generic}, for many tableaux $T[n]$ in the orbit, all the entries in $T$ will simply decrease by 1 in the next promotion step, so the run lengths and gaps between the runs are preserved. The tricky part is when entries from $T$ become small enough that they get promoted out of $T$ into the top row. We can characterize exactly when that happens:
    
    \begin{lemma}\label{lem:promoted_entry}
        For any 2-row standard Young tableau $T[n]$, whether or not it is in $\T(n,\ell,r)$, the  $k^\text{th}$ entry of $T$ gets promoted into the top row if and only if it is equal to $2k$ and there is no $i<k$ such that the $i^\text{th}$ entry of $T$ is $2i$.
    \end{lemma}

    \begin{proof}
        First note at most one entry from $T$ can get promoted into the top row at a given step, since if one entry gets promoted into the top row, the remaining entries of $T$ just slide one position to the left within $T$. Thus, the $k^\text{th}$ entry of $T$ gets promoted into the top row if and only if $k$ is the smallest number such that the $k^\text{th}$ entry of $T$ is less than the $(k+1)^{\text{st}}$ entry of the top row. The $k^\text{th}$ entry of $T$ must be larger than the first $k-1$ entries of $T$ and the first $k$ entries in the top row of $T[n]$, so it must be larger than $2k-1$ other entries and thus must be at least $2k.$ The $(k+1)^{\text{st}}$ entry of the top row of $T[n]$ can at most be larger than the first $k$ entries of the top row and the first $k$ entries of $T$, so it must be at most $2k+1$. Thus, the only way the $k^\text{th}$ entry of $T$ can be smaller is if the $k^\text{th}$ entry of $T$ is exactly $2k$. In that case the $(k+1)^{\text{st}}$ entry of the top row must be the next smallest entry and so must be $2k+1$, so the $2k$ will get promoted out of $T$ unless some entry to its left is getting promoted out of $T$ instead.
    \end{proof}

    Now for $T[n]\in \T(n,\ell,r)$, entries from a run of length $\ell$ will start getting promoted once the entries in the run are $\ell+1,\ell+2,\dots,2\ell$, since then the $\ell^\text{th}$ entry of $T$ is $2\ell$ and hence gets promoted, while if the first entry of the run is more than $\ell+2$, then for every $1\le k\le \ell$ the $k^\text{th}$ entry of the run will be at least $\ell+1+k>2k$ and so will not get promoted. When the entries of the smallest run become $\ell+1,\ell+2,\dots,2\ell,$ the first $\ell+1$ entries of the top row will be $1,2,\dots,\ell,2\ell+1$, and the $2\ell$ shown in yellow is the entry that first gets promoted:
    \begin{center}
        \ytableausetup{boxsize=1cm}
        \begin{ytableau}
            1 & 2 & \cdots & \ell & 2\ell+1 & \cdots \\
            \ell+1 & \ell+2 & \cdots & *(yellow) 2\ell & \cdots
        \end{ytableau}
    \end{center}
    Then in the next step we gain a $2\ell-1$ in the top row (since the $2\ell$ moves up and then decreases by 1) and an $n$ in $T$, and the remaining entries of the run decrease by 1 to become $\ell,\ell+1,\dots,2\ell-2$:
    \begin{center}
        \ytableausetup{boxsize=1cm}
        \begin{ytableau}
            1 & 2 & \cdots & \ell-1 & 2\ell-1 & \cdots \\
            \ell & \ell+1 & \cdots & *(yellow) 2\ell-2 & \cdots
        \end{ytableau}
    \end{center}
    The $2\ell-2$ gets promoted next because it is now the $(\ell-1)^{\text{st}}$ entry of $T$ and is equal to $2(\ell-1)$, and in the next step the entries of the run in $T$ will be $\ell-1,\ell,\dots,2\ell-4$ together with $n-1$ and $n.$
    
    Continuing this process, the set of entries from the run goes through being $k+1,k+2,\dots,2k$ together with $n-\ell+k+1,\dots,n-1,n$ for each $k$ from $\ell$ down to 1, since at the $k^\text{th}$ step, the $2k$ (show in yellow below) gets promoted into the top row:
    \begin{center}
    \ytableausetup{boxsize=1cm}
        \begin{ytableau}
            1 & 2 & \cdots & k & 2k+1 & \cdots \\
            k+1 & k+2 & \cdots & *(yellow) 2k & \cdots
        \end{ytableau}
    \end{center}
    After $\ell$ steps, the whole run of length $\ell$ has moved from being $\ell+1,\ell+2,\dots,2\ell$ at the front of $T$ to being $n-\ell+1,\dots,n-1,n$ at the end of $T$. If we follow the progression of where the ``first" entry of the run is, where for a wraparound run we count the ``first" entry as being the first one in the latter part of the run, then the first entry decreases down to $\ell+1$, but then it skips $1,2,\dots,\ell$ and instead becomes $n$, then $n-1,n-2,\dots$ until it decreases down to $\ell+1$ again. Thus, the first entry of the run ranges through the $n-\ell$ values $\ell+1,\ell+2,\dots,n-1,n$, so after $n-\ell$ steps, the first run is back to where it started. The same holds for each other run, which is essentially why all of $T$ will be back to its starting configuration after $n-\ell$ steps. 
    
    However, we need the assumptions about the gaps between the runs in order to ensure that it is actually always the entries from the smallest run that get promoted and not the entries of some larger run:

    \begin{lemma}
        Given our assumptions on the gaps between the runs, an entry of $T$ can only get promoted from the smallest run and not from a larger run.
    \end{lemma}

    \begin{proof}
        As seen above, after the entries from the smallest run decrease to become $\ell+1,\ell+2,\dots,2\ell$, all $\ell$ of them will get promoted into the first row over the next $\ell$ steps. So, if an entry from a larger run were to get promoted before that, it would have to get promoted while the first entry from the smallest run is at least $\ell+2,$ and hence its $\ell^\text{th}$ and largest entry is at least $2\ell+1$. For each subsequent run, its last entry is at least $2\ell$ more than the last entry of the previous run by our definition of $\T(n,\ell,r)$, so also its first entry is at least $\ell+1$ more than the last entry of the previous run, since its last entry is $\ell-1$ more than its first entry. Thus, its $j^\text{th}$ entry is at least $\ell+j$ more than the last entry of the previous run. Thus, after $i$ full runs and $j$ steps into the next run, we will be at the $(i\ell+j)^\text{th}$ entry of $T$, and its value will be at least $$2\ell+1 + 2\ell(i-1)+\ell+j = 2i\ell+\ell+j+1.$$ We have $j\le \ell$ since each run has length $\ell$, so $$2i\ell+\ell+j+1 \ge 2i\ell + 2j+1 > 2(i\ell+j).$$ Thus, the entry is more than twice the index of its position in $T$, so by Lemma \ref{lem:promoted_entry}, it will not get promoted out of $T$.
    \end{proof}

    Next we check that the set $\T(n,\ell,r)$ is closed under promotion:

    \begin{lemma}
        If $T[n]\in \T(n,\ell,r)$, then also $\mc{P}(T[n]) \in \T(n,\ell,r).$
    \end{lemma}

    \begin{proof}
        During steps where no entry gets promoted into the top row, all entries of $T$ simply decrease by 1, so the mod $n-\ell$ differences between the starting entry of each run and the starting entry of the next run do not change, nor do the run lengths, so the tableau remains in $\T(n,\ell,r).$ However, as seen above, it is also the case that during steps where entries of one of the runs start getting promoted into the top row, that run becomes a wraparound run exactly as described in Definition \ref{def:one_length}, and also its first entry continues to decrease by 1 mod $n-\ell$ at every step, since it generally decreases by 1 except when it skips from $\ell+1$ to $n$, and in that step it still decreases by 1 mod $n-\ell$. Since the start of each other run continues to decrease by 1 at each step while the entries of the run of interest get promoted and wrap around, the mod $n-\ell$ differences between the start of that run and the starts of the runs before and after it always stay the same, so they always remain at least $2\ell$. Thus, the tableau stays in $\T(n,\ell,r)$ at every step, as claimed.
    \end{proof}

    Since we know that $\T(n,\ell,r)$ is closed under promotion and that $\pr(T[n])\mid (n-\ell)$ for every $T[n]\in\T(n,\ell,r)$, promotion gives an action of the cyclic group $\mb{Z}_{n-\ell}=\langle \mc{P}\rangle$ on $\T(n,\ell,r)$, so it remains to count the number of fixed points of $\mc{P}^{(n-\ell)/d}$ for each $d\mid (n-\ell)$, and to check that it matches the value we get from plugging in a primitive $d^\text{th}$ root of unity to the claimed CSP polynomial.

    Since the sequence of differences between the starting entries of the runs is preserved at every step, and each run takes $n-\ell$ steps for its entries to return to their starting values at the front of $T$, the way we can get a period of $(n-\ell)/d$ instead of $n-\ell$ is if $T$ looks the same once a different run has moved to the front, meaning the sequence of gaps between starts of runs has $d$-fold symmetry under cyclic shifts. There are $r$ such gaps $g_1,g_2,\dots,g_r$ total, since we have $r$ runs, and altogether they need to add up to $n-\ell$, since the differences are taken mod $n-\ell$. Also, Definition \ref{def:one_length} requires $g_i\ge 2\ell$ for all $i$, and that is the only other requirement. Choosing such a sequence of gaps with $d$-fold symmetry is thus equivalent to choosing $g_1,g_2,\dots,g_{r/d}$ such that $$g_1+g_2+\dots+g_{r/d} = \frac{n-\ell}{d},\hspace{1cm}g_1,g_2,\dots,g_{r/d}\ge 2\ell.$$ Note that this can only be possible if $d\mid r$ in addition to $d\mid (n-\ell).$ If we subtract $2\ell$ from each gap so that the only requirement is that the new gaps be nonnegative, this becomes $$(g_1-2\ell)+\dots+(g_{r/d}-2\ell) = \frac{n-\ell}{d} - \frac rd\cdot 2\ell = \frac{n-\ell-2r\ell}{d}.$$ By stars and bars, the number of ways to have $r/d$ nonnegative differences with that sum is thus $$\binom{(n-\ell-2r\ell+r)/d-1}{r/d-1}.$$ Note that this sequence of gaps could also have more symmetry; we are just requiring that it have at least $d$-fold symmetry. This sequence of $r/d$ gaps then determines the full sequence of $r$ gaps because we can get the full gap sequence by repeating it $d$ times, so the above combination also counts the number of full gap sequences with at least $d$-fold symmetry.

    Now given such a gap sequence $g_1,g_2,\dots,g_r$, let $d'$ be maximal such that the sequence has $d'$-fold symmetry, so it is equivalent to the sequence $g_1,g_2,\dots,g_{r/d'}$ repeated $d'$ times, where the sequence $g_1,g_2,\dots,g_{r/d'}$ has no additional symmetry under cyclic shifts. Then to get a tableau corresponding to that gap sequence, we have $(n-\ell)/d'$ choices that actually give different tableaux for where the run before the size $g_1$ gap can start, because each time we add $(n-\ell)/d'$ to the start of that run, we get the same tableau by the symmetry in the overall gap sequence $g_1,g_2,\dots,g_r$. But actually, when we sum this over all gap sequences, every tableau gets counted $r/d'$ times, because it gets counted once for each of the $r/d'$ cyclic shifts of the gap sequence $g_1,g_2,\dots,g_{r/d'}$. Thus, the number of different tableaux is $\frac{(n-\ell)/d'}{r/d'} = (n-\ell)/r$ times the number of gap sequences. So, if we include only the gap sequences with at least $d$-fold symmetry that we counted above, the number of tableaux corresponding to gap sequences of that type (which is also the number of tableaux in $\T(n,\ell,r)$ fixed under $\mc{P}^{(n-\ell)/d}$), is $$\frac{n-\ell}{r}\binom{(n-\ell-2r\ell+r)/d-1}{r/d-1}.$$ We can already see that this matches the claimed CSP polynomial when we set $d=1$ and plug in $q=1$. To see that it matches for all $d$, note that our claimed CSP polynomial is equivalent to $$\frac{[n-\ell]_q}{[r]_q}\qbin{n-\ell-2r\ell+r-1}{r-1} = \frac{q^{n-\ell}-1}{q^r-1}\cdot\frac{(q^{n-\ell-2r\ell+r-1}-1)(q^{n-\ell-2r\ell+r-2}-1)\dots(q^{n-\ell-2r\ell+1}-1)}{(q^{r-1}-1)(q^{r-2}-1)\dots(q-1)}.$$ Now we are interested in plugging in a primitive $d^\text{th}$ root of unity $\zeta_d$ for $q$. In the second fraction, we can pair up each $q^i-1$ term in the denominator with the $q^{n-\ell-2r\ell+i}-1$ term in the numerator. Since $d\mid (n-\ell)$ and $d\mid r$, we also have $d\mid (n-\ell-2r\ell)$, so $\zeta_d^{n-\ell-2r\ell}=1$ and hence $\zeta_d^{n-\ell-2r\ell+i}-1 = \zeta_d^i-1.$ For $d\nmid i,$ we can thus cancel out the $\zeta_d^{n-\ell-2r\ell+i}-1$ factor in the numerator with the $\zeta_d^i-1$ factor in the denominator, since both factors equal the same nonzero value. What remains is thus $$\frac{q^{n-\ell}-1}{q^r-1}\cdot\frac{(q^{n-\ell-2r\ell+r-d}-1)(q^{n-\ell-2r\ell+r-2d}-1)\dots(q^{n-\ell-2r\ell+d}-1)}{(q^{r-d}-1)(q^{r-2d}-1)\dots(q^d-1)}.$$ All the powers of $q$ are divisible by $d$, so we can divide every factor by $q^d-1$ without changing the polynomial, since in both fractions there are the same number of factors in the numerator as the denominator. When we divide $q^{kd}-1$ by $q^d-1$ for any $k$, we get $$\frac{q^{kd}-1}{q^d-1} = q^{(k-1)d}+q^{(k-2)d}+\dots+q^{2d}+q^d+1.$$ When we set $q = \zeta_d,$ all terms on the right side become 1, so the right side becomes $k$. Thus, for each $k$, plugging in $q=\zeta_d$ is equivalent to replacing the $q^{kd}-1$ factor in our expression with $k$. So our expression becomes $$\frac{(n-\ell)/d}{r/d}\cdot\frac{\dfrac{n-\ell-2r\ell+r-d}{d}\cdot\dfrac{n-\ell-2r\ell+r-2d}{d}\dots\dfrac{n-\ell-2r\ell+d}d}{\dfrac{r-d}d\cdot\dfrac{r-2d}{d}\dots\dfrac dd},$$ which simplifies to $$\frac{n-\ell}r\cdot\frac{\left(\dfrac{n-\ell-2r\ell+r}d-1\right)\left(\dfrac{n-\ell-2r\ell+r}d-2\right)\dots\left(\dfrac{n-\ell-2r\ell}d+1\right)}{\left(\dfrac rd-1\right)\left(\dfrac rd - 2\right)\dots 1},$$ or equivalently, $$\frac{n-\ell}r\binom{(n-\ell-2r\ell+r)/d-1}{r/d-1},$$ which is what we wanted.
\end{proof}

    \begin{example}\label{ex:2row}
        Below is an orbit we get for $\ell=3$ and $n=16$, with one run shown in orange and the other in green. The orbit length is $n-\ell=16-3=13.$ For each run, the ``first" number cycles through the 13 values $16,15,\dots,4.$
        \begin{center}
        \begin{tabular}{cccc}
        \ytableausetup{boxsize=0.5cm}
        & \begin{ytableau}
            1 & 2 & 3 & 7 & 8 & 9 & 13 & 14 & 15 & 16 \\
            *(orange) 4 & *(orange) 5 & *(orange) 6 & *(green) 10 & *(green) 11 & *(green) 12 \\
        \end{ytableau}
        & $\longrightarrow$ & 
        \begin{ytableau}
            1 & 2 & 5 & 6 & 7 & 8 & 12 & 13 & 14 & 15 \\
            *(orange) 3 & *(orange) 4 & *(green) 9 & *(green) 10 & *(green) 11 & *(orange) 16 \\
        \end{ytableau} \\ \\ 
        
        $\longrightarrow$ & 
        \begin{ytableau}
            1 & 3 & 4 & 5 & 6 & 7 & 11 & 12 & 13 & 14 \\
            *(orange) 2 & *(green) 8 & *(green) 9 & *(green) 10 & *(orange) 15 & *(orange) 16 \\
        \end{ytableau} & 
        $\longrightarrow$ & 
        \begin{ytableau}
            1 & 2 & 3 & 4 & 5 & 6 & 10 & 11 & 12 & 13 \\
            *(green) 7 & *(green) 8 & *(green) 9 & *(orange) 14 & *(orange) 15 & *(orange) 16 \\
        \end{ytableau} \\ \\
        $\longrightarrow$ & 
        \begin{ytableau}
            1 & 2 & 3 & 4 & 5 & 9 & 10 & 11 & 12 & 16 \\
            *(green) 6 & *(green) 7 & *(green) 8 & *(orange) 13 & *(orange) 14 & *(orange) 15 \\
        \end{ytableau} & 
        $\longrightarrow$ & 
        \begin{ytableau}
            1 & 2 & 3 & 4 & 8 & 9 & 10 & 11 & 15 & 16 \\
            *(green) 5 & *(green) 6 & *(green) 7 & *(orange) 12 & *(orange) 13 & *(orange) 14 \\
        \end{ytableau} \\ \\
        $\longrightarrow$ & 
        \begin{ytableau}
            1 & 2 & 3 & 7 & 8 & 9 & 10 & 14 & 15 & 16 \\
            *(green) 4 & *(green) 5 & *(green) 6 & *(orange) 11 & *(orange) 12 & *(orange) 13 \\
        \end{ytableau} &
        $\longrightarrow$ &
        \begin{ytableau}
            1 & 2 & 5 & 6 & 7 & 8 & 9 & 13 & 14 & 15 \\
            *(green) 3 & *(green) 4 & *(orange) 10 & *(orange) 11 & *(orange) 12 & *(green) 16 \\
        \end{ytableau} \\ \\
        $\longrightarrow$ &
        \begin{ytableau}
            1 & 3 & 4 & 5 & 6 & 7 & 8 & 12 & 13 & 14 \\
            *(green) 2 & *(orange) 9 & *(orange) 10 & *(orange) 11 & *(green) 15 & *(green) 16 \\
        \end{ytableau} &
        $\longrightarrow$ &
        \begin{ytableau}
            1 & 2 & 3 & 4 & 5 & 6 & 7 & 11 & 12 & 13 \\
            *(orange) 8 & *(orange) 9 & *(orange) 10 & *(green) 14 & *(green) 15 & *(green) 16 \\
        \end{ytableau} \\ \\
        $\longrightarrow$ &
        \begin{ytableau}
            1 & 2 & 3 & 4 & 5 & 6 & 10 & 11 & 12 & 16 \\
            *(orange) 7 & *(orange) 8 & *(orange) 9 & *(green) 13 & *(green) 14 & *(green) 15 \\
        \end{ytableau} &
        $\longrightarrow$ &
        \begin{ytableau}
            1 & 2 & 3 & 4 & 5 & 9 & 10 & 11 & 15 & 16 \\
            *(orange) 6 & *(orange) 7 & *(orange) 8 & *(green) 12 & *(green) 13 & *(green) 14 \\
        \end{ytableau} \\ \\
        $\longrightarrow$ &
        \begin{ytableau}
            1 & 2 & 3 & 4 & 8 & 9 & 10 & 14 & 15 & 16 \\
            *(orange) 5 & *(orange) 6 & *(orange) 7 & *(green) 11 & *(green) 12 & *(green) 13 \\
        \end{ytableau} && \\
    \end{tabular}
    \end{center}
\end{example}


\subsection{Major index interpretation of the cyclic sieving polynomial}\label{subsec:maj_index}

For $\ell=1$, we will show that the above CSP polynomial matches the major index generating function for the tableaux in $\T(n,\ell,r)$, up to a power of $q$. For $\ell>1$, it does not quite match the major index generating function, but it does match the generating function for a modified version of the major index:

\begin{definition}
    The \emph{\tb{\tcb{major index}}} $\maj(T[n])$ of a tableau $T[n]$ is the sum of all numbers $k$ such that $k+1$ is in a strictly lower row than $k$ (called \emph{\tb{\tcb{descents}}}). Define the \emph{\tb{\tcb{$\boldsymbol{\ell}$-major index}}} $\maj_\ell(T[n])$ to be the same as the major index, except that descents $k$ with $k<\ell$ do not get added.
\end{definition}

\begin{example}
    Let $T[15]$ be the second tableau from Example \ref{ex:2row}:
    \begin{center}
        \begin{ytableau}
            1 & 2 & 5 & 6 & 7 & 8 & 12 & 13 & 14 & 15 \\
            *(orange) 3 & *(orange) 4 & *(green) 9 & *(green) 10 & *(green) 11 & *(orange) 16 \\
        \end{ytableau}
    \end{center}
    To compute $\maj(T[15])$, we count 2, 8, and 15 as all being descents, but to compute $\maj_3(T[15])$, we only count 8 and 15 as descents, since $2<3$. Thus, $\maj(T[15])=2+8+15=25$ while $\maj_3(T[15])=8+15=23.$
\end{example}

Thus, for a 2-row tableau $T[n]$, $\maj_\ell(T[n])$ is the sum of all entries $k\in T[n]$ such that $k$ is in the top row of $T[n]$, $k+1$ is in the second row $T$ of $T[n]$, and also $k\ge \ell$. Note that $\maj_1(T[n]) = \maj(T[n])$ is the normal major index.

\begin{theorem}\label{thm:maj_index}
    The CSP polynomial from Theorem \ref{thm:2row_linear} matches the generating function for the $\ell$-major index $\tn{maj}_\ell(T[n])$ over the tableaux in $\T(n,\ell,r)$, up to a power of $q$, namely:
    $$q^{r^2 \ell}\cdot \frac{[n-\ell]_q}{[r]_q}\qbin{n-\ell-2r\ell+r-1}{r-1} = \sum_{T[n]\in \T(n,\ell,r)} q^{\maj_\ell(T[n])}.$$
\end{theorem}

\begin{proof}
    Starting from the left side, a version of the $q$-binomial theorem implies that $$\qbin{n-\ell-2r\ell+r-1}{r-1} = \sum_{0=i_1\le i_2\le \dots \le i_{r} \le n-\ell-2r\ell}q^{i_1+\dots+i_{r}}.$$ Now we make the change of variables $$i_1'=i_1+\ell, \hspace{0.5cm} i_2' = i_2 + 3\ell,\hspace{0.5cm}\dots, \hspace{0.5cm} i_k' = i_k + (2k-1)\ell, \hspace{0.5cm}\dots, \hspace{0.5cm} i_r' = i_r + (2r-1)\ell.$$ Then the new sum is $$i_1' + \dots + i_r' = i_1 + \dots + i_r + (1+3+5+\dots+(2r-1))\ell = i_1 + \dots + i_r + r^2\ell,$$ and the new upper bound is $$i_r' = i_r + (2r-1)\ell \le n-2r\ell-\ell + 2r\ell-\ell = n-2\ell,$$ so the possible sequences of integers $i_1',\dots,i_r'$ that correspond to a sequence $i_1,\dots,i_r$ are the sequences $i_1',\dots,i_r'$ satisfying $$\ell = i_1'<i_2'<\dots < i_r' \le n-2\ell,\hspace{1cm}i_k'+2\ell \le i_{k+1}'\hspace{0.5cm}\tn{ for }1\le k\le r-1.$$ So we have $$q^{r^2\ell} \qbin{n-\ell-2r\ell+r-1}{r-1} = \sum_{\substack{\ell = i_1'<i_2'<\dots <i_r'\le n-2\ell, \\ i_k'+2\ell \le i_{k+1}'}}q^{i_1'+\dots+i_r'}.$$ 
    
    On the right side, note that we can write $$\maj_\ell(T[n])=s_1+\dots + s_r,$$ where $s_1+1,\dots,s_r+1$ are the starts of the runs in $T$. Our definition of $\T(n,\ell,r)$ is equivalent to the requirements $$\ell \le s_1<\dots < s_r \le n-1,\hspace{1cm}s_k+2\ell\le s_{k+1},\hspace{1cm}s_r \le s_1 + n-3\ell,$$ where the last requirement comes the fact that we require the mod $n-\ell$ residue of $s_1-s_r$, which is equal to $s_1-s_r + n-\ell$, to be at least $2\ell$. Thus, we can rewrite the right side of Theorem \ref{thm:2row_linear} as 
    \begin{equation}\label{eqn:maj_index_sum}
        \sum_{\substack{\ell \le s_1<\dots <s_r\le n-1, \\
        s_k+2\ell\le s_{k+1},\\
        s_r\le s_1 + n-3\ell}}q^{s_1+\dots+s_r}.
    \end{equation}
    Our sequences $i_1',\dots,i_r'$ correspond precisely to the subset of the $s_1,\dots,s_r$ sequences that satisfy $s_1=\ell.$

    Next, note that we can write 
    \begin{equation}\label{eqn:q_fraction}
        \frac{[n-\ell]_q}{[r]_q} = \frac{q^{n-\ell}-1}{q^r-1} = \frac{\left(\dfrac{q^{r(n-\ell)}-1}{q^r-1}\right)}{\left(\dfrac{q^{r(n-\ell)}-1}{q^{n-\ell}-1}\right)} = \frac{1+q^r+q^{2r}+\dots+q^{(n-\ell-1)r}}{1+q^{n-\ell}+q^{2(n-\ell)}+\dots+q^{(r-1)(n-\ell)}}.
    \end{equation} Now, in the product \begin{equation}\label{eqn:q_product}
        (1+q^r+q^{2r}+\dots+q^{(n-\ell-1)r})\sum_{\substack{\ell = i_1'<i_2'<\dots <i_r'\le n-2\ell, \\ i_k'+2\ell \le i_{k+1}'}}q^{i_1'+\dots+i_r'},
    \end{equation} 
    we can think of each term $$q^{jr}\cdot q^{i_1'+\dots+i_r'} = q^{(i_1'+j)+(i_2'+j)+\dots+(i_r'+j)}$$ as corresponding to $r$ numbers $i_1'',\dots,i_r''$ that are equal to $i_1',\dots,i_r'$ but shifted by $j$, so $i_k'' := i_k'+j.$ The smallest number $i_1'+j$ is now allowed to be $\ell,\ell+1,\dots,\ell+(n-\ell-1)=n-1.$ The conditions on the gaps between the numbers remain the same as before, so (\ref{eqn:q_product}) is equal to \begin{equation}\label{eqn:i_k_sum}
        \sum_{\substack{\ell\le i_1''\le n-1, \\ i_k''+2\ell \le i_{k+1}'', \\ i_r'' \le i_1'' + n-3\ell}}q^{i_1''+\dots+i_r''}. 
    \end{equation}
    Our possible sequences $i_1'',\dots,i_r''$ in (\ref{eqn:i_k_sum}) include all the desired sequences $s_1,\dots,s_r$ in (\ref{eqn:maj_index_sum}), but also additional sequences, because we require $s_r\le n-1$ we are not enforcing the requirement $i_r''\le n-1.$ 
    
    Each of our desired sequences $s_1,\dots,s_r$ actually corresponds to exactly $r$ sequences $i_1'',\dots,i_r''$. The first is the actual sequence $i_1''=s_1,\dots,i_r''=s_r.$ The second is the sequence $$i_1''=s_2,\hspace{0.5cm}i_2''=s_3,\hspace{0.5cm}\dots,\hspace{0.5cm}i_{r-1}''=s_r,\hspace{0.5cm}i_r'' = s_1+n-\ell.$$ In terms of the condition that the mod $n-\ell$ gaps between consecutive terms all be be at least $2\ell$, this is a valid $i_1'',\dots,i_r''$ sequence if and only if we start with a valid $s_1,\dots,s_r$ sequence, because adding $n-\ell$ to $s_1$ and ``wrapping it around" to become the last term does not change the mod $n-\ell$ gaps between consecutive terms, so the condition that the mod $n-\ell$ gaps all be at least $2\ell$ is satisfied by the sequence $i_1'',\dots,i_r''$ if and only if it is satisfied by the sequence $s_1,\dots,s_r$. The sum of this sequence $i_1'',\dots,i_r''$ is then $$i_1''+\dots+i_r''=s_1+\dots+s_r + n-\ell.$$ Similarly, $s_1,\dots,s_r$ also corresponds to the sequence $$i_1''=s_3,\hspace{0.5cm}i_2''=s_4,\hspace{0.5cm}\dots,\hspace{0.5cm}i_{r-2}''=s_r,\hspace{0.5cm}i_{r-1}''=s_1+n-\ell,\hspace{0.5cm}i_r''=s_2+n-\ell,$$ where we ``wrap around" both $s_1$ and $s_2$ by adding $n-\ell$ to both of them, and for this sequence the sum is $$i_1''+\dots+i_r'' = s_1+\dots+s_r + 2(n-\ell).$$ Continuing in this manner, for each $0\le k \le r-1,$ the $k^\text{th}$ sequence is $$i_1'' = s_{k+1},\hspace{0.5cm}\dots,\hspace{0.5cm}i_{r-k}''=s_r,\hspace{0.5cm}i_{r-k+1}=s_1+n-\ell,\hspace{0.5cm}\dots,\hspace{0.5cm}i_r''=s_k+n-\ell,$$ where the first $k$ terms $s_1,\dots,s_k$ get increased by $n-\ell$, so the sum is $$i_1''+\dots+i_r''=s_1+\dots+s_r+k(n-\ell).$$ Since the amounts that get added to the $s_1+\dots+s_r$ sums to get the corresponding $i_1''+\dots+i_r''$ sums are $0,n-\ell,2(n-\ell),\dots(r-1)(n-\ell),$ we can get from (\ref{eqn:i_k_sum}) to (\ref{eqn:maj_index_sum}) by dividing by $1+q^{n-\ell}+\dots+q^{(r-1)(n-\ell)}.$ But that is precisely the denominator in (\ref{eqn:q_fraction}), which completes the proof.
\end{proof}

\section{General orbit lengths for 2-row Young tableaux}\label{sec:2row}

The main idea for this section is that for a 2-row Young tableau $T[n]$ with $n$ sufficiently large, the orbit length $\pr(T[n])$ essentially depends only on how many runs of consecutive numbers there are in $T$, and what the lengths of the runs are. Based on what we saw in \S \ref{sec:2row_linear}, we can roughly think of all the runs of each length $\ell$ as moving together at a speed of $n-\ell$ steps per round, with the gaps between the runs being preserved. If there are $d$ different run lengths, then we essentially have $d$ different groups of runs, with each group moving independently from the other groups at a slightly different speed. 

If we fix a representative ``first run" in each group, then each tableau in the orbit is essentially determined by where the $d$ representative runs currently start, since the starting position for each other run is determined by the starting position of its group representative, together with the fixed gaps between runs within its group. There are $n-O(1)$ possible starting positions for each of the $d$ group representatives (assuming no extra symmetry), and those $d$ starting positions can be chosen almost independently of each other, so the number of tableaux within the orbit will be a divisor of some polynomial of the form $n^d-O(n^{d-1})$.

However, the exact polynomials will be somewhat complicated, because we need to describe exactly what happens when the runs overlap with each other. The remainder of this section is organized as follows:
\begin{itemize}
    \item When the runs overlap, they will temporarily look like runs of different lengths that are ``too close together," so in \S\ref{subsec:run_lengths} we describe how to read off the ``actual" run lengths from a given tableau.
    \item In \S\ref{subsec:2row_setup}, we introduce the main result of this section, Theorem \ref{thm:2row}, which gives a combinatorial interpretation for what exactly the orbit lengths are counting, which will essentially be the possible sets of $d$ different positions for the $d$ tracks of runs, assuming that within each track, the gaps between the runs stay fixed.
    \item In \S\ref{subsec:tableau_to_tracks} and \S\ref{subsec:tracks_to_tableau}, we describe how to translate between a tableau and its associated sequence of track positions and sequences of gaps between runs.
    \item In \S\ref{subsec:2row_proof}, we give the proof of Theorem \ref{thm:2row}.
    \item In \S\ref{subsec:2row_recursive}, we give a recursive formula to compute the orbit lengths more explicitly, showing that they are monic polynomials in $n$ with degree equal to the number of distinct run lengths.
\end{itemize}   


\subsection{Determining the run lengths from a tableau}\label{subsec:run_lengths}
    
    For each 2-row Young tableau $T[n]$, we will partition the entries of $T$ into \emph{\tb{\tcb{runs}}} using a somewhat subtle procedure. Most of the time, the runs will just be the sequences of consecutive numbers, as long as they are sufficiently far apart. However, for runs that are ``too close together" (which essentially means the number of missing numbers between the end of one run and the start of the next is less than both run lengths), we will actually want to consider them to be overlapping runs of different lengths than the obvious ones. This is because during most of the promotion orbit of $T[n],$ those runs will actually be replaced with runs of different ``stable" lengths that will persist for the majority of the orbit, and the orbit length will be characterized by the lengths of its ``stable" set of runs.
    
    To determine these ``stable" run lengths for a tableau $T[n]$, we first visualize a number line containing the numbers $1,2,\dots,n,$ and we mark the entries of $T$ with \emph{\tb{\tcb{dots}}}. Then we build a set of noncrossing arcs connecting each dot to an unmarked number (which we will call a \emph{\tb{\tcb{space}}}) using the following procedure:
    \begin{enumerate}
        \item Whenever there is an unpaired dot and the next unpaired number after it is a space, pair the dot with the space by drawing an arc between them. Continue until no more such pairings can be made.
    
        \item If all dots have been paired, we are done. If not, let the smallest unpaired dot be at position $k,$ so all unpaired numbers among $k,k+1,\dots,n$ are dots (or else we could make more pairings). Then find the smallest $j$ such that among the numbers $k,k+1,\dots,n$ together with $1,2,\dots,j,$ there are the same number of dots as spaces. Since in total there are more spaces than dots (as the first row of $T[n]$ is longer than the second row $T$), such a $j<k$ must exist. Remove any arcs involving the numbers $1,2,\dots,j$, and move those numbers to the end of the line after $n$, keeping them in the order $1,2,\dots,j.$
        
        
        \item Among the numbers $1,2,\dots,j,$ whenever we have an unpaired dot and the closest unpaired number \emph{before} it is a space, pair the dot with the space. Repeat until all dots among $1,2,\dots,j$ have been paired with a space. 
        
        If this fails, then there must be some $i$ such that among the numbers $1,2,\dots,2i,$ there are more dots than spaces, i.e. more than $i$ dots and fewer than $i$ spaces. But then in the tableau $T[n]$, the $i^\text{th}$ entry in the bottom row would be at most $2i$ (since it is a dot) while the $i^\text{th}$ entry in the top row would be more than $2i$ (since it is a space), so $T[n]$ would not be an SYT. Thus, this step cannot fail, so we will be able to pair every dot among $1,2,\dots,j$ with a space before it.

        \item Finally, pair the last unpaired dot among $k,k+1,\dots,n$ with the first unpaired space among $1,2,\dots,j$, and repeat until all the remaining unpaired numbers among $k,k+1,\dots,n$ (which were dots) have been paired with all the remaining unpaired numbers among $1,2,\dots,j$ (which were spaces). There will be no unpaired numbers left over among $k,k+1,\dots,n,1,2,\dots,j,$ since that set of numbers has the same number of dots as spaces.
    \end{enumerate}
We will later see that this arc diagram is telling us something about certain steps of the promotion process, but for now we will just describe how to use it to read off the run lengths. Once the pairing is complete, we can read off the run lengths as follows:
\begin{enumerate}
    \item Call each maximal sequence of $2\ell$ consecutive numbers such that the first $\ell$ numbers are paired with the last $\ell$ numbers a \emph{\tb{\tcb{rainbow}}}. For each sequence of adjacent rainbows with no extra numbers between them, remove all but the largest rainbow (or all but one of the largest rainbows, if there is a tie). Whenever a rainbow of $2\ell$ numbers is removed, count its $\ell$ dots as a run of length $\ell$.
    \item Once all remaining dots are in rainbows, count the dots in each rainbow of size $2\ell$ as a run of length $\ell.$
\end{enumerate}
For each set of distinct positive integers $\ell_1>\dots>\ell_d\ge 1,$ and run lengths $r_1,\dots,r_d\ge 1$, write $\T(n,\vec{\ell},\vec{r})$ for the set of tableaux with $r_i$ runs of length $\ell_i$ for each $i$, and no other runs, where $\vec{\ell}=(\ell_1,\dots,\ell_d)$ and $\vec{r}=(r_1,\dots,r_d)$. By the above procedure, every tableau $T[n]$ gets assigned to exactly one of these sets, and the run lengths will satisfy $r_1\ell_1+\dots+r_d\ell_d=|T|,$ since every entry in $T$ (i.e. every dot) gets assigned to exactly one run.

\begin{example}
    For the tableau $T[22]$ below, the gaps between runs are large enough that the runs are just the sequences of consecutive numbers in the bottom row:
    \begin{center}
        \ytableausetup{boxsize=0.6cm}
    \begin{ytableau}
        1 & 2 & 3 & 7 & 8 & 9 & 11 & 12 & 15 & 16 & 19 & 20 & 21 & 22 \\
        *(orange) 4 & *(orange) 5 & *(orange) 6 & *(green) 10 & *(red) 13 & *(red) 14 & *(cyan) 17 & *(cyan) 18
    \end{ytableau}
    \end{center}
    The longest run is $(4,5,6)$ of length 3, so $\ell_1 = 3$ and $r_1 = 1$, the next longest runs are $(13,14)$ and $(17,18)$, each of length 2, so $\ell_2 = 2$ and $r_2 = 2$, and the final run is $(10)$ of length 1, so $\ell_3 = r_3 = 1$. Thus, $\vec{\ell} = (3,2,1)$ and $\vec{r} = (1,2,1).$
\end{example}

\begin{example}\label{ex:finding_runs}
    Consider the tableau $T[20]$ shown below:
    \ytableausetup{boxsize=0.6cm}
    \begin{center}
        \begin{ytableau}
            1 & 3 & 4 & 5 & 6 & 7 & 8 & 11 & 14 & 15 & 16 & 19 \\
            *(green) 2 & *(red) 9 & *(cyan) 10 & *(red) 12 & *(red) 13 & *(green) 17 & *(orange) 18 & *(green) 20
        \end{ytableau}
    \end{center}
    At first glance, it looks like the runs are $(2),$ $(9,10)$, $(12,13)$, $(17,18)$, and $(20)$, but those turn out not to be the runs we want, because some of those runs are ``too close together." Instead, we need to use the pairing procedure described above. We are able to pair each dot with a space after it except for 17 and 20, so we need to move 1, 2, 3, and 4 to the end, since then among 17, 18, 19, 20, 1, 2, 3, 4, we have 4 dots and 4 spaces. With our notation above, that means $k=17$ and $j=4$. We then pair the dots as shown below:
    \begin{center}
        \includegraphics[width=15cm]{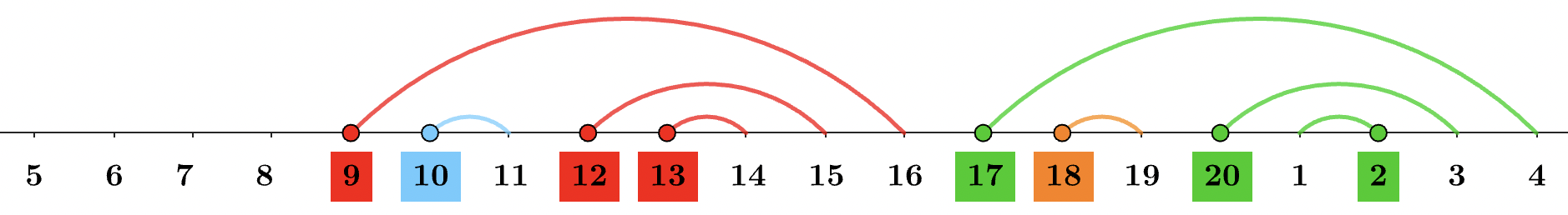}
    \end{center}
    Note that 2 gets paired with a space \emph{before} it since it got moved to the end, while every other dot gets paired with a space \emph{after} it (where we are counting 3 as being after 20, and 4 as being after 17, since 3 and 4 got moved to the end).
    
    Then to read off the run lengths, we remove the smallest rainbows $(10,11)$ and $(18,19)$, so $(10)$ and $(18)$ get counted as runs of length 1. With those rainbows removed, $(9,12,13,14,15,16)$ and $(17,20,1,2,3,4)$ are now rainbows, so $(9,12,13)$ and $(17,20,2)$ get counted as runs of length 3. We thus have run lengths $\ell_1=3$ and $\ell_2=1$, and numbers of runs $r_1=r_2=2,$ so the tableau belongs to $\T(20,(3,1),(2,2)).$ The colors in the diagram indicate run membership, with the colored numbers being the ones in $T$.
\end{example}

\subsection{Characterization of the 2-row orbits (Theorem \ref{thm:2row})}\label{subsec:2row_setup}

Now we will show that for $n$ sufficiently large, $\T(n,\vec{\ell},\vec{r})$ is closed under promotion, and we will describe what the orbits within it look like. The idea is that for each $1\le i\le d,$ the runs of length $\ell_i$ are all moving together around a circular track of some length $n_i<n$, with the gaps between them always staying the same. Generally, the $d$ different groups of runs all move simultaneously around their tracks of slightly different lengths, except that whenever a run of length $\ell_i$ is crossing over the boundary of its track, the runs of each length $\ell_j\ne \ell_i$ all pause on their track for $\min(\ell_i,\ell_j)$ of the crossing steps and then $\min(\ell_i,\ell_j)$ additional steps. Thus, we cannot have two runs from different groups actively crossing at the same time (although sometimes one run will start crossing, pause while another run crosses, and then finish crossing). 

The track length for the track containing the length $\ell_i$ runs will be $$n_i := n - \ell_i - 2\sum_{j\ne i} r_j\cdot \min(\ell_i,\ell_j),$$ and we will refer to the track with the length $\ell_i$ runs as the \emph{\tb{\tcb{$\boldsymbol{i}^{\text{th}}$ track}}}. We will number the $i^{\text{th}}$ track with the numbers $1,2,\dots,n_i$ in order, and we will call position $n_i$ the \emph{\tb{\tcb{boundary}}} of the track. We will describe every tableau $T[n]\in \T(n,\vec{\ell},\vec{r})$ using the following two pieces of data:
\begin{enumerate}
    \item A list of $d$ \emph{\tb{\tcb{positions}}} $(p_1,\dots,p_d)$ such that $p_i\in\{1,2,\dots,n_i\}$ denotes where one particular representative run on the $i^\text{th}$ track starts.
    \item A list of $d$ \emph{\tb{\tcb{gap sequences}}} $(\vec{g}(1),\dots,\vec{g}(d))$, such that $$\vec{g}(i)=(g_1(i),\dots,g_{r_i}(i)),\hspace{1cm}g_1(i)+\dots+g_{r_i}(i) = n_i,\hspace{1cm}g_1(i),\dots,g_{r_i}(i)\ge 2\ell_i,$$ representing the gaps between the starts of consecutive runs on the $i^{\text{th}}$ track.
\end{enumerate}
We will then say that the $j^\text{th}$ run on the $i^\text{th}$ track starts at position $p_i+g_1(i)+\dots+g_{j-1}(i)$ and ends at position $p_i+g_1(i)+\dots+g_{j-1}(i)+\ell_i-1.$ The positions and gap sequences will be uniquely determined from the tableau, up to changing which run within each group is chosen as the representative and cyclically shifting the corresponding gap sequences. 

The positions and gap sequences will also satisfy the additional restriction that if $\ell_i>\ell_j,$ there can never be both a length $\ell_i$ run currently ending at one of the $2\ell_j$ positions positions $n_i-\ell_j+1,\dots,n_i,1,2,\dots,\ell_j$ closest to the boundary of its track and also a length $\ell_j$ run currently ending at one of the $2\ell_j$ positions $n_j-\ell_j+1,\dots,n_j,1,2,\dots,\ell_j$ closest to the boundary of its track. We can visualize this by imagining that the tracks all share a common boundary, and we can think of the space close to this boundary as a ``bottleneck" through which at most one run can be passing at a time. In terms of promotion, a run will be passing through the bottleneck when some of its entries are being promoted into the top row, and then immediately afterwards.

We can now state the main result of \S\ref{sec:2row}:

\begin{theorem}\label{thm:2row}
     As long as $n_i\ge 2r_i\ell_i$ for every $i$, the set $\T(n,\vec{\ell},\vec{r})$ is closed under promotion, and within each orbit, we can choose the gap sequences to always stay the same while the positions rotate. Moreover, every orbit length divides the number $P_d(n,\vec{\ell},\vec{r})$ of valid position sequences given a list of gap sequences, which is independent of the choice of gap sequences.
\end{theorem}

Before giving the proof of Theorem \ref{thm:2row} in \S\ref{subsec:2row_proof}, we will need to explain how exactly to determine the positions and gaps from a given tableau and vice versa, which we will do in \S\ref{subsec:tableau_to_tracks} and \S\ref{subsec:tracks_to_tableau}

\subsection{Mapping a tableau to a set of tracks} \label{subsec:tableau_to_tracks}

Given a tableau $T[n]$, we will first draw the arc diagram described in \S\ref{subsec:run_lengths}. Then for each $1\le i \le d,$ we can find its $i^\text{th}$ position number and $i^\text{th}$ gap sequence as follows:
\begin{enumerate}
    \item Assign the track numbers $1,2,\dots,n_i$ to the positions on the number line in the order $1,2,\dots,n$, except that for each run of a different length $\ell_j\ne\ell_j,$ skip over the middle $2\min(\ell_i,\ell_j)$ numbers in the rainbow corresponding to that run. Additionally, skip over the first $\ell_i$ numbers that are not already being skipped and are not dots.
    
    To see that the number of assigned track positions equals the track length $n_i$, note that for each $j\ne i,$ we skip $2r_j\min(\ell_i,\ell_j)$ total numbers, plus we skip an additional $\ell_i$ numbers at the start, so the number of track numbers assigned to the positions on the number line is $n-\ell_i-2\sum_{j\ne i}r_j\min(\ell_i,\ell_j)=n_i$. Thus, we will not run out of track numbers or have any track numbers leftover. The condition $n_i\ge2r_i\ell_i$ ensures that there is enough space on the track to fully fit all $r_i$ rainbows of size $2\ell_i$ corresponding to the $r_i$ runs of length $\ell_i.$
    \item We can now take the position $p_i$ to be the starting track position for any one of the $r_i$ runs of length $\ell_i$ (which we designate as the first run). Then we number the length $\ell_i$ runs as the first, second, $\dots,r_i^{\text{th}}$ runs in order, starting from the second run, wrapping around to the start of the track once we reach the end. Then $g_j(i)$ is the mod $n_i$ residue of the starting position of the $(i+1)^{\text{st}}$ run minus the starting position of the $i^{\text{th}}$ run.   

    To see that this gives a valid gap sequence, note that we will automatically have $g_1(i)+\dots+g_{r_i}(i)=n_i,$ since after adding all $r_i$ gaps, we will be back to the starting track position. To see that $g_j(i)\ge 2\ell_i$ for every $j=1,2,\dots,r_i,$ note that the rainbows corresponding to different runs can never overlap, since the arcs are noncrossing, and they also cannot be nested inside each other, since if a size $2\ell_i$ rainbow were nested inside of another size $2\ell_i$ rainbow, then the initial center part of the outer rainbow would have to be a rainbow of size less than $2\ell_i,$ so that smaller rainbow would have been removed and counted as its own run instead of the inner size $2\ell_i$ rainbow being removed and counted as a separate run. Since each rainbow takes up $2\ell_i$ numbers along the track, and each rainbow starts with a dot, the gaps between the starts of adjacent rainbows must be at least $2\ell_i$.
\end{enumerate}
Note also that by the above argument, we can only ever have rainbows corresponding to smaller runs nested inside of rainbows corresponding to larger runs, and never larger runs nested inside of smaller runs. On the $i^{\text{th}}$ track, for every run of shorter length $\ell_j<\ell_i$, the entire rainbow corresponding to that run gets removed, so there will not be anything nested inside the rainbows corresponding to length $\ell_i$ runs. Note also that when a run of length $\ell_i$ crosses the boundary, for any of its arcs with both endpoints after the boundary, the left endpoints of those arcs are among the extra $\ell_i$ numbers that get skipped at the beginning, so only the right endpoints (the dots) gets included. Those dots thus end up immediately after the dots of the run that come to the left of the boundary. Thus, each run of length $\ell_i$ actually corresponds to $\ell_i$ \emph{consecutive} numbers on the $i^{\text{th}}$ track.


We also need to check that the resulting position and gap sequences satisfy the bottleneck condition:

\begin{prop}\label{prop:bottleneck_issues}
    The resulting position and gap sequence will never have two runs of lengths $\ell_j<\ell_i$ such that both the last dot of the length $\ell_j$ run on its track is at one of $n_j-\ell_j+1,\dots,n_j,1,2,\dots,\ell_j$, and also the last dot of the length $\ell_i$ run on its track is at one of $n_i-\ell_j+1,\dots,n_i,1,2,\dots,\ell_j.$
\end{prop}

\begin{proof}
    Assume for contradiction that two such runs exist. For the shorter run, its rainbow on the original track either crosses the boundary or starts immediately to the right of the boundary. If its last dot on its track is at one of $n_j-\ell_j+1,\dots,n_j,$ then its rainbow crosses the boundary, because there are $\ell_j$ numbers in the rainbow after the last dot but fewer than $\ell_j$ spaces between the dot and the boundary. Similarly, if its last dot is at one of $1,2,\dots,\ell_j$, then there are $\ell_j-1$ dots before it in the rainbow, so either the rainbow starts at position 1 on the track, or it starts to the left of that and crosses the boundary.
    
    For the longer run, the rainbow necessarily crosses the boundary. If the last dot in the run is at one of $n_i-\ell_j+1,\dots,n_i,$ there are still $\ell_i>\ell_j$ spaces of the rainbow to its right, so the rightmost part of the rainbow is on the other side of the boundary. Similarly, if the last dot is at one of $1,2,\dots,\ell_j,$ there are $\ell_i-1\ge \ell_j$ dots of the rainbow to its left, so again the rainbow crosses the boundary. 

    Since they both cross the boundary (or the smaller one starts immediately after the boundary), the smaller rainbow must be nested inside the larger one. That means that before the smaller rainbow is removed, a subset of the larger rainbow must initially form a rainbow that sits adjacent to the smaller one and is at least as large as it, so at least size $2\ell_j$. But if that rainbow sits to the left of the size $2\ell_j$ rainbow, its last dot (which is also the last dot of the larger size $2\ell_i$ rainbow) will be at position $n_i-\ell_j$ or earlier, since there are at least $\ell_j$ spaces of the rainbow after it and before the size $2\ell_j$ rainbow starts. Then the last dot would not be within $\ell_j$ positions of the boundary, a contradiction. Similarly, if the rainbow sits to the right of the size $2\ell_j$ rainbow, its last dot will be at least at position $\ell_j+1$, since the size $2\ell_j$ rainbow ends to the right of the boundary, at position 1 or later. Then again the last dot would not be within $\ell_j$ steps of the boundary, so we have the desired contradiction, and no bottleneck issue can occur.
\end{proof}

Now we give an example to illustrate this mapping:

\begin{example}\label{ex:finding_tracks}
    Returning to Example \ref{ex:finding_runs}, the red and green runs $(9,12,13)$ and $(17,20,2)$ of length 3 will both belong to the first track, so we will now color them both green, and similarly, the blue and orange runs $(10)$ and $(18)$ of length 1 will both belong to the second track, so we will color them both orange. Since $n=20,$ $\ell_1=3,\ell_2=1,$ and $r_1=r_2=2,$ the first track will have length $$n_1 = n-\ell_1-2r_2\min(\ell_1,\ell_2) = 20-3-2\cdot2\cdot1 = 13,$$ while the second track will have length $$n_2 = n-\ell_2 - 2r_1\min(\ell_1,\ell_2) = 20-1-2\cdot2\cdot1=15.$$ In the figure below, the top row gives the numbering from the original number line, the middle row gives the numbering for the first track, and the bottom row gives the numbering for the second track:
    \begin{center}
        \includegraphics[width=15cm]{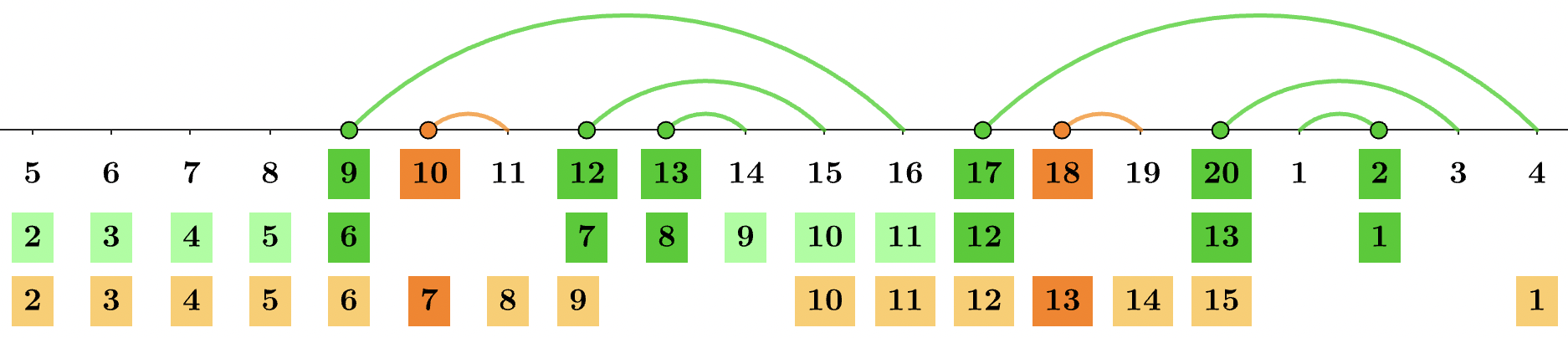}
    \end{center}
    For the first track, we skip over all rainbows for the second track, and we also skip the first 3 other numbers that are not dots, which are 1, 3, and 4. For the second track, we skip over the center 2 numbers of each green rainbow, $(13,14)$ and $(1,2)$, and we also skip the first extra number, 3. 
    
    We can now visualize the two tracks separately as shown below:
    \begin{center}
        \includegraphics[width=9.75cm]{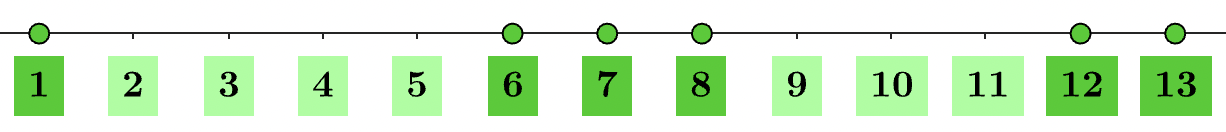}
        \includegraphics[width=11.25cm]{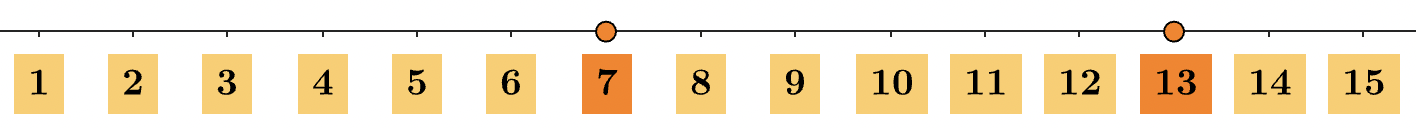}
    \end{center}
    Note that the runs of dots are consecutive on each track, provided we imagine the first track as being circular with the 13 wrapping around to come before the 1. 
    
    For the first track, we could choose the position as either $p_1=6$ or $p_1=12.$ If we choose $p_1=6,$ we get the gap sequence $\vec{g}(1) = (6,7),$ since $12-6=6$ and $6-12=-6\equiv 7\pmod{13}.$ If we choose $p_1=12,$ we get the cyclically shifted gap sequence $\vec{g}(1)=(7,6)$, since the 7 becomes the first gap and the 6 the second gap.

    For the second track, we could choose either $p_2=7$ or $p_2=13.$ If we choose $p_2=7,$ then we get the gap sequence $\vec{g}(2) = (6,8),$ since $13-7=6,$ and $7-13=-6\equiv 8\pmod{14}$. Thus, if we choose $p_2=13,$ we get the cyclically shifted gap sequence $\vec{g}(2)=(8,6).$
\end{example}

\subsection{Recovering a tableau from its set of tracks}\label{subsec:tracks_to_tableau}

To show that each list of positions and gap sequences corresponds to a unique tableau, we will describe how to reverse the above process and recover a tableau $T[n]$ given its positions and gap sequences:
\begin{enumerate}
    \item First use the positions and gap sequences to build the $d$ individual tracks.
    \item Add an extra $\ell_i$ spaces to the beginning of the $i^{\text{th}}$ track, except that if there are $k\le \ell_i$ dots at the very beginning of the track, add the first $k$ spaces before those dots, and the rest after the $k$ dots.
    \item We then add an extra $2\sum_{j\ne i}r_j\min(\ell_i,\ell_j)$ spaces to the $i^{\text{th}}$ track for every $i$ so that all tracks have length $n$. To do this, we go through the numbers $1,2,\dots,n$ in order. Whenever we see that a number $k$ is currently the end of a run on the $i^{\text{th}}$ track for some $i$, we insert $2\min(\ell_i,\ell_j)$ extra spaces into the $j^{\text{th}}$ track for each $j\ne i$, immediately after position $k-\min(\ell_i,\ell_j)$, i.e. we insert the numbers $k-\min(\ell_i,\ell_j)+1,\dots,k+\min(\ell_i,\ell_j)$ into the track, and we add $2\min(\ell_i,\ell_j)$ to all existing numbers on the track that are $k-\min(\ell_i,\ell_j)+1$ or larger, maintaining whether they are dots or spaces. The positions of the inserted spaces thus correspond to the positions on the $i^{\text{th}}$ track of the centermost part of the rainbow involving $k$. If there are multiple runs ending at $k$, we can choose any one of them to use first.
    
    The exception is that if the rainbow involving $k$ crosses the boundary and $k$ is before the boundary within the rainbow, and if $\min(\ell_i,\ell_j)>n-k,$ for any $j\ne i,$ the final $\min(\ell_i,\ell_j)-(n-k)$ extra spaces get added to the beginning of the $j^{\text{th}}$ track instead of to the end. In that case, we will add those extra spaces to the starts of the tracks at the \emph{beginning} of the process instead of at the end, starting with whichever of the runs in rainbows crossing the track boundaries would end first if we imagined lining up all their boundaries, and breaking ties arbitrarily. Similarly, if the rainbow involving $k$ crosses the boundary and $k$ is after the boundary, then for each $j\ne i$ with $\min(\ell_i,\ell_j)<k,$ the first $k-\min(\ell_i,\ell_j)$ extra spaces get added to the end of the $j^{\text{th}}$ track instead of the beginning.

    \item Finally, we read off the entries in $T$ by merging all the tracks together and taking $k$ to be an entry of $T$ if and only if it is a dot on any track. 
\end{enumerate}

\begin{example}\label{ex:merging_tracks}
    Starting with the two tracks from Example \ref{ex:finding_runs}, we can draw both individual tracks and then follow the above process to recover the tableau $T[20]$. The colored numbers throughout will be the ones that were on the tracks originally. First we insert an extra $\ell_1=3$ spaces into the first track, with one of them before the first dot, and the other two after, and we insert an extra $\ell_2=1$ space into the second track at the beginning:
    \begin{center}
        \includegraphics[width=12cm]{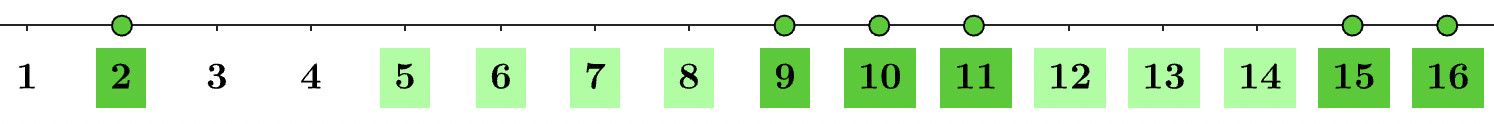}
        \includegraphics[width=12cm]{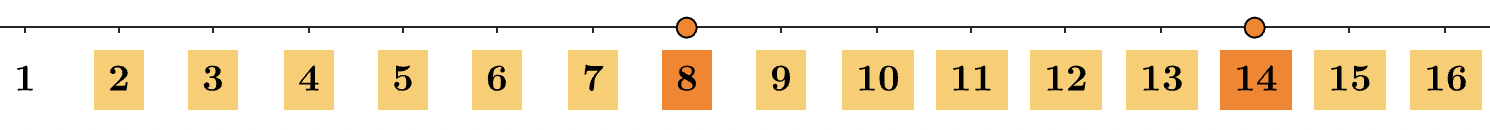}
    \end{center}
    The first end of a run we encounter is the 2 on the first track, so we insert $2\min(\ell_1,\ell_2)=2$ extra spaces in the second track at the positions $(2,3)$:
    \begin{center}
        \includegraphics[width=13.5cm]{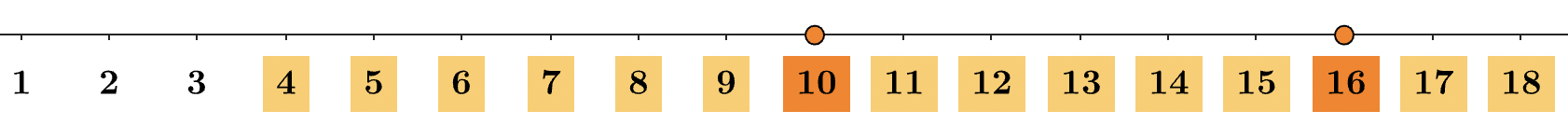}
    \end{center}
    The next end of a run we encounter is the 10 on the second track, so we insert two spaces into the first track at positions $(10,11)$:
    \begin{center}
        \includegraphics[width=13.5cm]{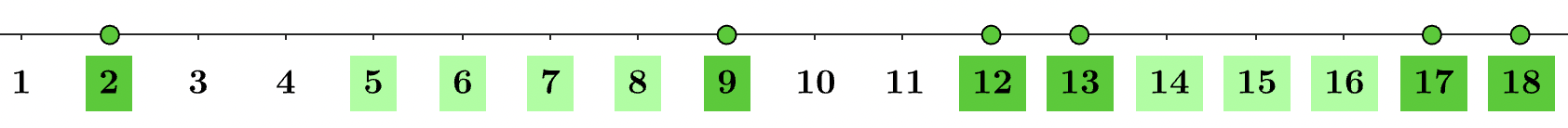}
    \end{center}
    The next end of a run is the 13 on the first track, so we insert two spaces into the second track at positions $(13,14)$:
    \begin{center}
        \includegraphics[width=15cm]{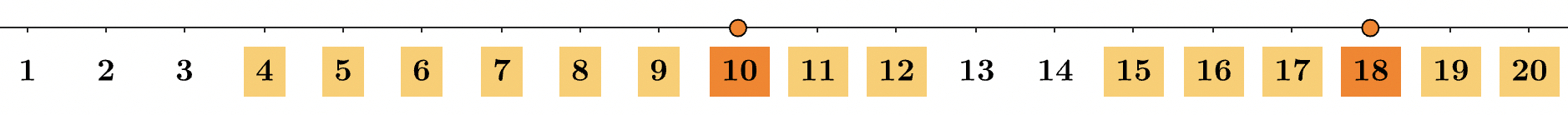}
    \end{center}
    Finally, the last end of a run is the 18 on the second track, so we insert two spaces into the first track at positions $(18,19)$:
    \begin{center}
        \includegraphics[width=15cm]{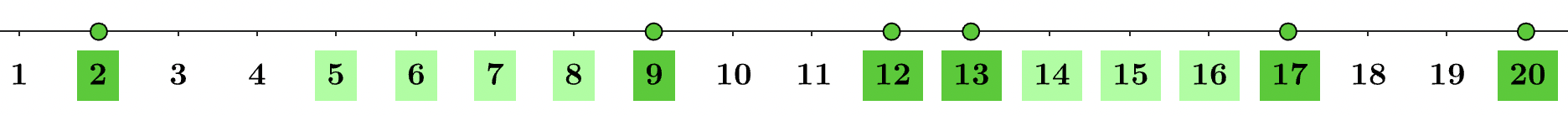}
    \end{center}
    We can now see that merging these tracks gives us back our original number line from Examples \ref{ex:finding_runs} and \ref{ex:finding_tracks}, and that the dots are now the entries of $T$.
\end{example}

\begin{prop}\label{prop:inverses}
    The process described above always results in a valid SYT, and is the inverse of our process for turning tableaux into positions and gap sequences.
\end{prop}

\begin{proof}
    We will break this into several lemmas for things we need to check.

    \begin{lemma}
        Every end of a run on any track gets seen exactly once in the merging process.
    \end{lemma}
     
    This will mean that the resulting tracks all actually have length $n$, since we will add $\ell_i$ spaces to the $i^{\text{th}}$ track initially, and an additional $2\min(\ell_i,\ell_j)$ spaces for each of the $r_j$ tracks of each length $i\ne j$. 
    
    \begin{proof}
        If an end of a run has not yet been seen and we encounter some other end of a run first, we will insert spaces ahead of the not-yet-seen run end, so it will only move further to the right on its track and will still be seen in the future. Thus, every run end gets seen at least once. 
    
    For some end of a run of length $\ell_i$ to get seen twice, we would need to first visit it once, and then have it get moved far enough ahead when inserting spaces into its track that we also visit it a second time. When it gets visited the first time, any run end to its right on the $j^{\text{th}}$ track gets moved ahead by $2\min(\ell_i,\ell_j)$ spaces, so it will end up at least $2\min(\ell_i,\ell_j)$ spaces to the right of the length $\ell_i$ run end. Then when the length $\ell_j$ run end gets visited, earlier numbers on the $i^{\text{th}}$ track can only get shifted to the right if they are within $\min(\ell_i,\ell_j)$ positions of the length $\ell_j$ run end, which the length $\ell_i$ run end is not. In fact, an already visited end of a run will never again get shifted further to the right, so it will never get visited a second time.
    \end{proof}

    \begin{lemma}
        After inserting the extra spaces, there cannot be a dot at the same position on two tracks.
    \end{lemma}

    \begin{proof}
     Suppose for contradiction that a dot belongs to runs on both the $i^{\text{th}}$ and $j^{\text{th}}$ tracks, with $\ell_i > \ell_j.$ But then if the end of the length $\ell_j$ run got visited before the end of the length $\ell_i$ run, all dots of the length $\ell_i$ run that overlap with a dot of the length $\ell_j$ run would have been shifted $2\ell_j$ spaces to the right, so they would no longer overlap. Similarly, if the end of the length $\ell_i$ run got visited first, then the final $\ell_j$ dots of that run (which must include the overlap dot, or else the length $\ell_j$ run would have ended first), will all get shifted $2\ell_j$ spaces to the right, so the dots will no longer overlap. The only way the dots in either run can be impacted later in the process would be if the dots that were moved to the right could get moved even further to the right if another run on some other track ends before their run, but this would not make them overlap with the run of dots further to the left. 
    
    Another potential issue is if shifting the dots at the very beginning of the track to be at the other endpoints of their arcs results in those dots ending up on top of other dots. However, this cannot happen, because our process results not only in non-overlapping dots but also non-overlapping arcs, because whenever we have a smaller rainbow of size $2\ell_j$ and a larger rainbow of size $2\ell_i,$ either the smaller rainbow will be entirely before or after the larger rainbow when the extra spaces corresponding to it are inserted into the larger track, or else we will insert spaces to the larger track corresponding to the entirety of the smaller rainbow, so that it becomes fully nested within the larger rainbow and none of the arcs overlap. This still holds near the boundary of the track since the insertion process there is essentially the same. Thus, since moving a dot to the other endpoint of its arc moves it to a new position that is still within its own rainbow, the dot cannot end up on top of another dot. Thus, all the dots on the final tracks are at different positions from each other, so they will all correspond to different entries of $T$.
    \end{proof}

    \begin{lemma}
        The resulting tableau $T[n]$ is a valid standard Young tableau.
    \end{lemma}

    \begin{proof}
        To check that $T[n]$ is a valid SYT, we need the $k^{\text{th}}$ entry of $T$ to be larger than the $k^{\text{th}}$ entry of the top row for every $k$, which is equivalent to the $k^{\text{th}}$ dot always coming after the $k^{\text{th}}$ space. To show that, we can imagine starting with the first track with just its extra $\ell_1$ spaces at the beginning added, and then inserting the runs from the second track, and then the runs from the third track, and so on. 

        We first claim that the condition holds when we just have the first track together with its extra $\ell_1$ spaces added at the beginning, with no other spaces added yet. If there are $0\le k\le \ell_1$ dots at the very start of the track, the statement holds for those dots, because we insert $k$ spaces ahead of them. If there are no dots at the very beginning of the track, the statement also holds for the leftmost run of $\ell_1$ dots, because all $\ell_1$ extra spaces get inserted ahead of those dots. Then we can inductively show that the statement holds for each subsequent run of $\ell_1$ dots on the track because of the assumption that the gaps between the starts of the runs are at least $2\ell_1$. That assumption implies there will always be at least $\ell_1$ spaces between the end of the previous run and the start of the new run, which means we can pair each of the $\ell_1$ new dots with a different one of those $\ell_1$ spaces after the previous run. Since each dot is paired with a space before it and no two dots are paired with the same space, there are at least as many spaces as dots in every initial segment of the track.

        Next, we show by induction on $j$ that the statement will still hold when we insert each run of dots from the $j^{\text{th}}$ track into the track formed by merging the first $j-1$ tracks, assuming we insert them from leftmost length $\ell_j$ run to rightmost $\ell_j$ run at whatever position our tracks-to-tableau process would end up inserting them relative to the existing longer runs. Every time we insert a run of $\ell_j$ dots into the merged track, the entire run gets inserted together with $i$ extra spaces after it (i.e. a full rainbow gets inserted), since the new run we are inserting is shorter than all runs currently on the track. Now, if the condition previously held for all dots on the merged track, it will still hold for those dots, because for the dots before where the new run is inserted, nothing to their left changes, while for the dots to the right of where the new rainbow is inserted, they gain the same number of dots as spaces ahead of them ($\ell_j$ dots and $\ell_j$ spaces), so the number of spaces ahead of each dot is still always larger than the number of dots ahead of it, as long as it was before the new size $2\ell_j$ rainbow got inserted.
        
        To see that the condition also holds for the newly inserted dots, note first that it holds if the new dots are inserted before any of the existing dots, whether they belong to a rainbow crossing the boundary or a rainbow after the boundary. This is because we can use the exact same argument we used above for showing that the condition holds for the first track by itself, assuming that for the $i^{\text{th}}$ track, we have already inserted the extra $\ell_j$ spaces at the beginning (or after the rainbow crossing the boundary), and also that we have shifted any dots to the right of the boundary within a rainbow crossing the boundary (if there is one) to be at the right endpoints of their arcs instead of the left endpoints.

        Then for each subsequent run of $\ell_j$ dots that gets inserted, we can assume the condition holds when all length $\ell_j$ runs further to the left of it have already been inserted. If the new run is then inserted fully to the right of some existing rainbow (and not within any rainbow), then the condition will hold, because it holds for the last dot within the rainbow by the inductive hypothesis, and then following that dot there are $\ell_i$ spaces for some $\ell_i\ge \ell_j$ (where $2\ell_i$ is the size of the rainbow), meaning that if we look at the track up to the end of the new run, we have $\ell_j$ additional dots and $\ell_i\ge \ell_j$ additional spaces, thus still at least as many total spaces as dots.

        Finally, suppose the new run of $\ell_j$ dots gets inserted partway through some existing larger rainbow, so that it ends up nested inside the rainbow. If it gets inserted before some of the dots belonging to the rainbow, then there must be at least $\ell_1$ more dots of the rainbow after it, or else the rainbow's run would have ended first and so it would have been inserted before all of those dots. Then the new dots will hold the same positions that $\ell_1$ consecutive dots within the rainbow's run previously held. Thus, since the condition held for those $\ell_1$ dots within the run before the new dots were inserted, it also holds for the new dots, as they occupy the same positions and so have the exact same set of dots and spaces to their left that the previous dots had.

        The other possibility is that the new dots get inserted within an existing rainbow (i.e. before the right endpoints of some of the rainbow's arcs), but after all dots of the rainbow. In that case, our process for editing the $j^{\text{th}}$ track inserts the final $\ell_j$ dots of the rainbow's run together with $\ell_j$ spaces after them into the $j^{\text{th}}$ track. That means there are at least $\ell_j$ spaces after the rightmost dot of the existing longer run but before the start of the new run. Within the left portion of the track up to the last dot of the longer run, there are at least as many spaces as dots by the inductive hypothesis. Then since there are $\ell_j$ more spaces before the new run of $\ell_j$ dots, by the end of the new run there are still at least as many spaces as dots, as needed. 
        
        This covers all cases for inserting a new run, so by induction, the statement will still hold after inserting the dots from all $d$ tracks into the first track, i.e. it will hold for the final merged track.
    \end{proof}

    \begin{lemma}
        Our tableau-to-tracks process is the inverse of our tracks-to-tableau process.
    \end{lemma}

    \begin{proof}
        It suffices to check if we start with the tracks, build the tableau, and then rebuild the tracks, we will end up with the same tracks that we started with. 
        
        One main thing to check is that when our merging process nests a smaller size $2\ell_j$ rainbow inside of a larger size $2\ell_i$ rainbow, the smaller rainbow is actually the one that gets removed, i.e. it is smaller than the other inner rainbow formed by a portion of the larger run. 
        
        If the smaller size $2\ell_j$ rainbow get inserted before some of the dots of the larger rainbow, that means that when merging the tracks, the length $\ell_j$ run ended first. Then the final portion of the longer length $\ell_i$ run that gets shifted $2\ell_j$ positions to the right when inserting the size $2\ell_j$ rainbow starts $\ell_j$ positions before the last dot of the newly inserted run, so it contains at least $\ell_j$ dots. Thus, the nested rainbow formed by that portion of the run has size at least $2\ell_j.$

        If the smaller rainbow gets inserted after all dots of the larger rainbow but before some arcs of the larger rainbow, we need to show that the number of right endpoints of arcs of the larger rainbow that come before the newly inserted run is at least $\ell_j,$ since the number of such arcs determines the size of the portion of the larger rainbow that gets nested inside the outer portion. In that case, the longer run must have ended before the shorter run, so the entire shorter run would have been shifted $2\ell_j$ spots to the right. Since it initially ended after the end of the longer run, after the shifting it ends at least $2\ell_j$ positions after the end of the longer run, which means there are at least $\ell_j$ spaces after the end of the longer run and before the shorter run. Those $\ell_j$ or more spaces then correspond to the right endpoints of the arcs of the portion of the larger rainbow that gets nested inside, so that inner rainbow has size at least $2\ell_j$, as needed.

        Thus, when we do the merging process and then separate the merged track back into individual tracks, we will always recover the original runs of the correct lengths. Also, the numbering of the tracks that we read off from the merged track by skipping over the centers of other rainbows will exactly match the numbering of the original tracks, since in both cases, the original track numbering for the $i^{\text{th}}$ skips over the central $2\min(\ell_i,\ell_j)$ spaces of each other rainbow, and also in both cases an extra $\ell_i$ spaces at the beginning or after the end of the rainbow crossing the boundary get skipped.

        However, the other thing we should be careful about is runs that cross over the boundary. Consider a run of length $\ell_i$ with $k$ dots after the boundary. We may assume that all smaller runs nested inside of it have already been removed. Our tracks-to-tableau process will insert the first $k$ extra spaces before those $k$ dots. Then the arc diagram will connect those $k$ dots to the $k$ spaces preceding them in a rainbow, and it will connect the remaining $\ell_i-k$ dots of the run immediately before the boundary to the next $\ell_i-k$ spaces after the $k$ dots, which are the remaining extra spaces that get inserted. These dots will thus form a single rainbow and so will be counted as a single run of length $\ell_i$, and the dots will all end up being consecutive on the track (since all the spaces within the rainbow will get removed), with exactly $k$ of them to the right of the track boundary. Thus, in this case as well we recover the original run of the correct length and in its original position on the track.
        \end{proof}

    Putting these lemmas together completes the proof of Proposition \ref{prop:inverses}, because we already know that every SYT satisfying the required conditions maps to a unique set of tracks, and we have now shown that every set of tracks also maps to an SYT, and that the SYT we get from merging a set of tracks is always the same as the SYT that maps to those tracks. We thus have a bijection between SYT and tracks, or equivalently, between SYT and position and gap sequences (up to cyclically shifting the gap sequences and adjusting the positions accordingly).
\end{proof}

\subsection{Proof of Theorem \ref{thm:2row}: 2-row orbit lengths}\label{subsec:2row_proof}

We will split the proof into several lemmas. The main things to check are that $\T(n,\vec{\ell},\vec{r}$) is closed under promotion (Lemma \ref{lem:closed_under_promotion}), that the gap sequences stay the same throughout each orbit (Lemma \ref{lem:gaps_constant}), that the tracks rotate in a certain way relative to each other (Lemma \ref{lem:tracks_rotating}) and that the orbit lengths divide $P_d(n,\vec{\ell},\vec{r})$ (Lemma \ref{lem:divisor_of_num_tableaux}).

\begin{lemma}\label{lem:closed_under_promotion}
    $\T(n,\vec{\ell},\vec{r})$ is closed under promotion.
\end{lemma}

\begin{proof}
    First, we claim that the entire arc diagram associated to the tableau just shifts one position to the left at every step, except at steps where the leftmost rainbow hits the boundary (meaning the leftmost dot of the leftmost rainbow becomes $n$), which corresponds to the first step where a number corresponding to a dot from the rainbow gets promoted into the top row of the tableau.
    
    It is clear that the dots and arc diagrams will all just shift one position to the left at promotion steps where no rainbow is crossing over the boundary, which is equivalent to no dot being promoted into the top row. We claim that at all steps where a rainbow is crossing the boundary after the first step, the whole arc diagram will still just move one step to the left along the circular track at each promotion step, except that when the left endpoint of one of its arcs hits $n$, the dot at the right endpoint of that arc will jump over to the left endpoint at $n.$ To see that, note that it follows from Lemma \ref{lem:promoted_entry} that a dot gets promoted into the top row if and only if it is the first dot such that for some $k\ge 1$, it is the $k^{\text{th}}$ dot and it is at position $2k$ on the track. That is equivalent to saying that up to and including the dot, there are the same number of dots as spaces on the track.
    
    Since the arc diagram connects every dot to the right of the boundary but within the rainbow crossing the boundary with a space to its left, a dot will be promoted at the next step if and only if it is currently connected to a space at 1. This is because in that case, the dot in question has the same number of dots as spaces weakly to its left because the arc diagram gives a 1-to-1 pairing between those dots and those spaces. On the other hand, every dot to its left but past the boundary has strictly more spaces than dots that are weakly to its left, because if we focus on the region up to and including such a dot, every dot is paired with a space to its left within the region, while at least one space (namely, 1) is not paired to any dot within the region. Thus, none of those dots can get promoted yet, so the dot connected to 1 will get promoted. 
    
    Then when the dot connected to 1 gets promoted, that dot gets replaced with a dot at $n$, which is the new position of the left endpoint of its arc when we shift the entire arc diagram to the left. The rest of the arc diagram stays the same because all other numbers within the arc diagram will either stay as dots or stay as spaces. Thus, the only change to the arc diagram is that the dot connected to 1 jumps to the other endpoint of its arc and is now at $n$.

    We know that the run lengths can be read off from the arc diagram based only on the sizes of the rainbows and not the positions of the dots within the rainbows, so the multiset of run lengths stays the same as long as the arc diagram stays the same, and thus the tableau remains in $\T(n,\vec{\ell},\vec{r})$ at those steps. 
    
    It remains to check that the multiset of run lengths also stays the same at the steps where the arc diagram changes. If we have a single run of length $\ell_i$ with nothing nested inside it, then immediately before it crosses the boundary, there are $\ell_i$ spaces to the right of the boundary before the $\ell_i$ dots, which are precisely the $\ell_i$ extra spaces that get skipped when numbering the $i^{\text{th}}$ track, so that the run currently starts at position 1 on the $i^{\text{th}}$ track. The arc diagram connects those $\ell_i$ dots to the $\ell_i$ spaces to their right. Then it is the rightmost dot of the run that will first get promoted, so that dot jumps to $n$, and will now be connected to the space at position $2\ell_i-1$, while the remaining $\ell_i-1$ dots of the run will now be at positions $\ell_i-1,\dots,2\ell_i-2$, and they will be connected to the $\ell_i-1$ spaces at positions $1,2,\dots,\ell_i-1$. Thus, the rainbow will still have the same size $2\ell_i$, and the dots will still form a run of length $\ell_i$. This looks exactly the same as when one of our runs from Theorem \ref{thm:2row_linear} gets promoted and crosses the boundary.

    Now suppose we have one or more smaller runs nested inside the outer larger run. We will show by induction on the number of runs inside the original version of the portion of the arc diagram that changes that the multiset of run lengths does not change when we change the arc diagram. Look at the shortest run nested inside, which we can say has length $\ell_j < \ell_i$. Before we switch the arc diagram, that run looks like $\ell_j$ dots connected to $\ell_j$ spaces immediately to their right. Following those spaces, there could be more dots within the large rainbow, or more spaces. 
    
    If there are more dots following the $\ell_j$ spaces, there must be at least $\ell_j$ more dots, since the size $2\ell_j$ rainbow must be the smallest one nested inside the large outer rainbow. Then when we switch the arc diagram, those $\ell_j$ spaces will instead be connected to the $\ell_j$ dots immediately to their right. But then in both cases, we have a rainbow of size $2\ell_j$ nested at essentially the same place within the larger outer rainbow, with the only change being that the dots of that rainbow moved from being to the left of their spaces to being to the right of their spaces (so some of the dots changed which run they belong to). However, we still have a run of length $\ell_j$ in both cases, and the remaining dots after removing that rainbow look exactly the same in both cases.

    If the $\ell_j$ spaces are followed by more spaces, then the $\ell_j$ dots must also be preceded by more spaces, because if they were preceded by more dots, those dots would be connected to the spaces immediately after the size $2\ell_j$ rainbow, so the rainbow would actually have size larger than $2\ell_j.$ The dots must then be preceded by at least $\ell_j$ spaces, because if there were fewer than $\ell_j$ spaces to their left, then connecting the leftmost of those spaces to as many as possible of the dots immediately to their left would create a rainbow of size smaller than $2\ell_j$, contradicting the minimality of the size $2\ell_j$ rainbow. Thus, there are at least $\ell_j$ spaces to the left of the $\ell_j$ dots, so in the new arc diagram, the dots will get connected to the spaces immediately to their left instead of to their right. But then removing those dots and their associated spaces will again leave the same sequence of remaining dots and spaces for the original arc diagram as for the new arc diagram.

    Thus, in either case, the smallest run within the portion whose arc diagram changes stays the same length in both cases, and also the remaining set of dots and spaces once the rainbow corresponding to those dots is removed looks exactly the same in the new arc diagram as in the old one. Then it follows by induction on the number of rainbows within the changed portion of the arc diagram that the entire list of run lengths stays the same in the new arc diagram as in the old one (since we already did the base case of a single rainbow). Thus, the tableau stays in $\T(n,\vec{\ell},\vec{r})$ at every promotion step.
\end{proof}

\begin{lemma}\label{lem:gaps_constant}
    At every promotion step, the gap sequences on the tracks do not change.
\end{lemma}

\begin{proof}
    Again, this is clear as long as no rainbow is crossing the boundary, so it suffices to show that it also holds at the steps where the arc diagram changes, and at the steps where a rainbow is not changing but is crossing the boundary.

    While a rainbow is crossing the boundary, the arc diagram does not change except that each time a dot gets promoted into the top row, it hops from the right endpoint of its arc to the left endpoint. If the dot getting promoted is not the first in its run, then the gaps do not change because the gaps can all be determined from where the first dot in each run is on its track (since all other dots in the runs will just end up immediately behind the first dot of their runs on their tracks), and the first dot in the run of the dot being promoted just shifts one position to the left during that step, like all other dots on the track.
    
    If the dot being promoted is the first in a run of some length $\ell_i$, we note that before it hops, the left endpoint of its arc must have been at position 1, as we saw above. Also, all other numbers between position 1 and the dot about to hop are either dots from a shorter run nested inside the run of the dot being promoted (in which case they will be skipped in the numbering for the $i^{\text{th}}$ track), or they are other dots from the same run, or the left endpoints of the arcs connected to those dots (which are precisely the $\ell_i$ extra spaces that get skipped in the numbering for the $i^{\text{th}}$ track). Thus, at the step where the dot gets promoted and the start of its run moves to the last position on its track, the start of its run must have previously been at position 1 on its track, since in the numbering for the $i^{\text{th}}$ track, all numbers before the first dot of the run get skipped either because they belong to smaller rainbows or because they are among the extra $\ell_i$ skipped spaces. Thus, the start of the run still just shifts one position to the left at that step, as do the starts of all later runs of the same length, so the gaps along the $i^{\text{th}}$ track stay the same.

    It remains to show that at steps where the arc diagram changes, the runs within the changed portion still all just shift one position to their left along their respective tracks, thus maintaining the same gaps between them and the other runs on their tracks.

    Again, we will use induction on the number of runs nested within the starting arc diagram. The argument above gives us the base case of a single run. Now consider what happens when we remove a run of minimal length $\ell_j$ within the portion of the arc diagram that changes. As we saw earlier, once that run and its associated spaces are removed, the remaining set of dots and spaces looks exactly the same before and after the change in the arc diagram, so it follows by the inductive hypothesis that the gaps will stay the same on all other tracks when the length $\ell_j$ run is added back, because those $2\ell_j$ numbers get skipped in the numbering of every other track.

    Now for the $j^{\text{th}}$ track itself, to show that the gaps between the run in question and other runs on its track stay the same, it suffices to consider the gaps between consecutive length $\ell_j$ runs that are both within the rainbow that changes, and the gap between the last length $\ell_j$ run in the part that changes and the first length $\ell_j$ run in the part that does not change. 
    
    The position of each run on the $j^{\text{th}}$ track is determined by the position of its dots on the main merged track together with the number of dots and spaces that get removed prior to the run. Within the portion that changes, we can identify the length $\ell_j$ runs because they correspond to exactly $\ell_j$ dots with at least $\ell_j$ spaces on either side, or to exactly $\ell_j$ spaces with at least $\ell_j$ dots on either side. This does not change when we change the arc diagram. 
    
    The numbers that get skipped when we number the $j^{\text{th}}$ track are the middle $2\ell_j$ numbers of each other rainbow. In the starting arc diagram, these are the last $\ell_j$ dots of each other set of consecutive dots on the main track and the $\ell_j$ spaces following them, and in the ending arc diagram, they are the first $\ell_j$ dots of each set of consecutive dots and the $\ell_j$ spaces preceding them. Thus, for two length $\ell_j$ runs that are both within the portion that changes, the number of skipped numbers between the runs does not change because in both cases it is just equal to $2\ell_j$ times the number of other sets of consecutive dots in between the two sets of dots of interest. The same thing holds for the last length $\ell_j$ run within the changed portion and the first length $\ell_j$ run within the portion that does not change, so the number of skipped numbers between those two runs also stays the same. Thus, the gaps between the length $\ell_j$ runs stay the same, completing the proof.
\end{proof}

\begin{lemma}\label{lem:tracks_rotating}
    At every promotion step, all dots rotate one position to the left on their tracks, except that whenever a run of length $\ell_i$ is crossing the boundary, for every $j\ne i$ the $j^{\text{th}}$ track pauses while the last $\min(\ell_i,\ell_j)$ dots of the run are crossing the boundary, and also during the following $\min(\ell_i,\ell_j)$ steps.
\end{lemma}

\begin{proof}
    At each step, the dots all rotate one position to the left along the main track. Generally the number of skipped numbers in the $j^{\text{th}}$ track numbering will also stay the same between each pair of consecutive length $\ell_j$ runs, and between the first length $\ell_j$ run and start of the main track, meaning the dots will also just rotate one position to the left along the $j^{\text{th}}$ track. We have also already checked that along the $j^{\text{th}}$ track, the dots will always stay in runs of length $\ell_j$ with constant gaps between them, so it suffices to consider whether or not the first run on the $j^{\text{th}}$ track moves closer to the boundary at a given step.

    Since its position along the main track moves one step closer to the boundary at every promotion step, the first run on the $j^{\text{th}}$ track stays still if and only if the number of skipped numbers between its first dot and the boundary decreases by 1, as opposed to staying the same. The number of such skipped numbers will decrease by 1 precisely when one of the skipped numbers moves to the other side of the boundary. Aside from the extra $\ell_j$ skipped numbers at the beginning of the track (which are not relevant because they are always after the start of the boundary and before the run in question unless that run itself is crossing over the boundary, in which case it definitely is not staying still), the skipped numbers are the middle $2\min(\ell_i,\ell_j)$ numbers of each rainbow for a run of a different length $\ell_i\ne \ell_j.$ 
    
    The steps where that central portion of the rainbow is crossing the boundary are precisely the steps where each of its final $\min(\ell_i,\ell_j)$ dots crosses the boundary (which look in terms of arc diagrams like a dot hopping from the right endpoint to the left endpoint of its arcs while the arc diagram itself just shifts one position to the left), followed by the next $\min(\ell_i,\ell_j)$ steps where the right half of that central portion of the rainbow (consisting of the right endpoints of the arcs whose left endpoints are now dots on the other side of the boundary) crosses over the boundary. Thus, the tracks stay still precisely during the steps described in Lemma \ref{lem:tracks_rotating}, as claimed.
\end{proof}

\begin{lemma}\label{lem:divisor_of_num_tableaux}
    If an orbit of tableaux within $\T(n,\vec{\ell},\vec{r})$ does not include the full number $P_d(n,\vec{\ell},\vec{r})$ of possible tableaux given its gap sequences, the number of tableaux within the orbit is a divisor of $P_d(n,\vec{\ell},\vec{r})$.
\end{lemma}

\begin{proof}    
    Compare some orbit $\mc{O}$ within $\T(n,\vec{\ell},\vec{r})$ to another orbit $\mc{O}'$ where the $i^{\text{th}}$ track starts shifted by some amount $s_i\le \ell_d$ in $\mc{O}'$ compared to where it starts in $\mc{O}$, while all other tracks start in the same positions in $\mc{O}$ and $\mc{O'}$. We claim that $\mc{O}$ and $\mc{O}'$ have the same length. To see this, we will show that throughout $\mc{O}$ and $\mc{O}'$, the $i^{\text{th}}$ track will always stay offset by the same amount in $\mc{O}'$ relative to $\mc{O}$, while  
    
    The only times when that may not be the case are when a run on the $i^{\text{th}}$ track is crossing the boundary or has just finished crossing the boundary in one or both $\mc{O}$ and $\mc{O'}$. If we say the $i^{\text{th}}$ track is one step ahead on $\mc{O}$ compared to on $\mc{O'}$, then for each $j\ne i,$ the $j^{\text{th}}$ track will pause one step earlier in $\mc{O}$ than in $\mc{O}'$ to let the final portion of that run cross the boundary. Then the position of the $j^{\text{th}}$ track in $\mc{O}'$ will be offset by one position from where it is in $\mc{O}$ for the next $2\min(\ell_i,\ell_j)-1$ steps, since the position at which it stops moving is offset by one in $\mc{O}$ as compared to $\mc{O}$, and for the next $2\min(\ell_i,\ell_j)-1$ steps, the $j^{\text{th}}$ track will be paused in both orbits. Then the $j^{\text{th}}$ track will restart moving a step earlier in the orbit $\mc{O}$ in which it paused first, so after that step, its position in the two orbits will again match, since from then on it will be moving in both orbits, until the next time when a run on the $i^{\text{th}}$ track is crossing the boundary in $\mc{O}$. These temporary offsets that soon get reset will not change the orbit length, since the orbit length is determined by when all the runs simultaneously return to their starting positions. Thus, the orbits $\mc{O}$ and $\mc{O}'$ have the same length.

    Now, given two orbits within the size $P_d(n,\vec{\ell},\vec{r}$) subset of $\T(n,\vec{\ell},\vec{r})$ that corresponds to a particular set of gap sequences, we can show that those orbits always have the same length by turning one orbit into the other through a series of these steps of shifting one track at a time by one position while keeping the other tracks fixed. If we run into a bottleneck issue when trying to shift one of the tracks by one, then we can just wait some number of promotion steps until the rainbow that would cause the bottleneck issue has finished crossing the boundary, and then shift the track by one at that position in the orbit instead to show that the orbits have the same length. It follows that $\T(n,\vec{\ell},\vec{r})$ can be partitioned into subsets of size $P_d(n,\vec{\ell},\vec{r})$, each of which is closed under promotion and can be partitioned into promotion orbits that all have the same length. Thus, this common orbit length must be a divisor of $P_d(n,\vec{\ell},\vec{r}).$
    \end{proof}

Putting these lemmas together completes the proof of Theorem \ref{thm:2row}. \qed

\bigskip

To understand why the orbit length might be a divisor of $P_d(n,\vec{\ell},\vec{r})$ instead of equaling $P_d(n,\vec{\ell},\vec{r})$, note that recovering the starting tableau after some number of promotion steps is equivalent to having all the tracks return to their starting positions. The number of promotion steps for the $i^{\text{th}}$ track to rotate all the way around is always $n-\ell_i$. Intuitively, this is because if we look at any fixed run on that track, the smallest entry of that run essentially decreases by 1 at each step, except that when that run is crossing the boundary, it ``skips ahead" by $\ell_i$ steps because its smallest entry changes directly from $\ell_i+1$ to $n$. More precisely, we can see this because the $i^{\text{th}}$ track length is $n_i = n-\ell_i-2\sum_{j\ne i}r_j\min(\ell_i,\ell_j)$, so there are $n_i$ steps per round where the track is rotating by 1 position at each step, plus an additional $2\sum_{j\ne i}r_j\min(\ell_i,\ell_j)$ steps per round where it is staying still, since there are $2\min(\ell_i,\ell_j)$ such steps for each of the $r_j$ runs of each length $\ell_j\ne \ell_i.$ Thus, in total it takes $n-\ell_i$ steps per round for the track to return to its starting position.

One reason for the orbit length to be shorter than $P_d(n,\vec{\ell},\vec{r})$ is if there is at least one $i$ for which the $i^{\text{th}}$ track has $d_i$-fold rotational symmetry for some $d_i>1,$ meaning the track length $n_i$ is a multiple of $d_i$, the number $r_i$ of runs on the track is a multiple of $d_i$, and the gap sequence has $d_i$-fold rotational symmetry. In that case, the orbit length would just get divided by $d_i$ for each such $i$, since there are only $1/d_i$ times as many distinct rotations for the $i^{\text{th}}$ track.

Now to understand how the tracks move relative to each other, note that for the $i^{\text{th}}$ and $j^{\text{th}}$ tracks for each $i < j,$ after each round the offset between the tracks changes by $\ell_i-\ell_j$, because it takes $n-\ell_i$ steps for the $i^{\text{th}}$ track to return to its starting position, but $n-\ell_j$ steps for the $j^{\text{th}}$ track to return to its starting position, meaning the $i^{\text{th}}$ track will have rotated an additional $\ell_i-\ell_j$ positions by the time the $j^{\text{th}}$ track returns to its starting position. Thus, as long as $\ell_i-\ell_j$ is relatively prime to the lengths $n_i$ and $n_j$ of the $i^{\text{th}}$ and $j^{\text{th}}$ tracks, the $i^{\text{th}}$ and $j^{\text{th}}$ tracks will cycle through all possible offsets over the course of the orbit. However, if $d_{ij}=\gcd(\ell_i-\ell_j,n_i,n_j)>1,$ then the offset between the $i^{\text{th}}$ and $j^{\text{th}}$ tracks will always stay the same mod $d_{ij}$ (except that it will briefly change when one of the tracks is staying still but the other is moving), so the number of possible pairs of positions for those two tracks, and in fact the entire orbit length, will just get divided by $d_{ij}$ compared to what it would have been. More generally, if we have some larger set $S$ of tracks that all return to their starting positions after $1/d_S$ times the number of expected steps (when the tracks not in $S$ are ignored), the full orbit length will just get divided by $d_S.$

\bigskip

As a consequence of Theorem \ref{thm:2row}, we can count the total number of tableaux in $\T(n,\vec{\ell},\vec{r})$:

\begin{corollary}\label{cor:num_tableaux}
    The total number of tableaux in $\T(n,\vec{\ell},\vec{r})$ is $$P_d(n,\vec{\ell},\vec{r})\cdot \prod_{i=1}^d \frac1{r_i}\binom{n_i-2r_i\ell_i+r_i-1}{r_i-1}.$$
\end{corollary}

\begin{remark}
    Note that the $\T(n,\ell,r)$ case from \S\ref{sec:2row_linear} is the special case $d=1$. In that case, $n_1=n-\ell$ and $P_1(n,\ell,r)=n-\ell,$ since there are $n-\ell$ rotations for the single track of length $n_1=n-\ell$. The formula in Corollary \ref{cor:num_tableaux} thus becomes $$(n-\ell)\cdot\frac1r\binom{n-\ell-2r\ell+r-1}{r-1},$$ matching what we get from plugging in $q=1$ to the CSP polynomial in Theorems \ref{thm:2row_linear} and \ref{thm:maj_index}.
\end{remark}

\begin{proof}
    The counting argument is essentially the same as in the $\T(n,\ell,r)$ case. $P_d(n,\vec{\ell},\vec{r})$ counts the number of valid rotations of the $d$ tracks given a fixed sequence of gaps for each track, if we designate a particular ``first" run on each track. Thus, multiplying $P_d(n,\vec{\ell},\vec{r})$ by the number of ways to assign gap sequences to the $d$ tracks gives the total number of valid ways to place the runs of dots on all the tracks, and also to choose a ``first" run for every track. Each tableau in $\T(n,\vec{\ell},\vec{r})$ corresponds to a unique arrangement of dots on tracks, but then for each such dot arrangement, there are $r_1r_2\dots r_d$ ways to choose the ``first" runs for all the tracks. Thus, $\T(n,\vec{\ell},\vec{r})$ is equal to $P_d(n,\vec{\ell},\vec{r})$, times the number of valid gap sequences, divided by $r_1r_2\dots r_d.$

    The gap sequences for different tracks can be chosen independently, so we just need to count the number of possible gap sequences for the $i^{\text{th}}$ track for each $i$. We need $r_i$ gaps $g_1,\dots,g_{r_i}$ that sum to $n_i$ and are each at least $2\ell_i.$ Subtracting $2\ell_i$ from every gap, this is equivalent to choosing $r_i$ nonnegative numbers with sum $n_i-r_i\cdot 2\ell_i$, which by stars and bars is $$\binom{n_i-2r_i\ell_i+r_i-1}{r_i-1}.$$ Putting this together proves Corollary \ref{cor:num_tableaux}.
\end{proof}

We end this section with an example to illustrate what happens at the tableau, arc diagram, and track level while a run with another run nested inside it crosses the boundary:

\begin{example}
    Returning again to the tableau from examples \ref{ex:finding_runs}, \ref{ex:finding_tracks}, and \ref{ex:merging_tracks}, we show below the promotion steps during which one of the large rainbows with the smaller rainbow nested inside it crosses the boundary. We show what happens for the relevant portions of the tableau, the arc diagram, and the two tracks. 
    
    Below is the rainbow immediately before it starts crossing the boundary, and the corresponding tableau. The colored numbers in the top row are the ones in the bottom row $T$ of the tableau. Note that for the first track, we skip over the numbers $(5,6)$ corresponding to a run on the other track and the space after it, and on the second track, we skip over the numbers $(8,9)$ corresponding to the end of a run on the first track and the space after it. We also skip the additional 3 numbers $1,2,3$ for the first track, and additional $1$ for the second track. In each tableau, we highlight in green the entries in $T$ belonging to the green run, and in orange the one belonging to the orange run.
    \begin{center}
        \raisebox{-.5\height}{\includegraphics[width=8.25cm]{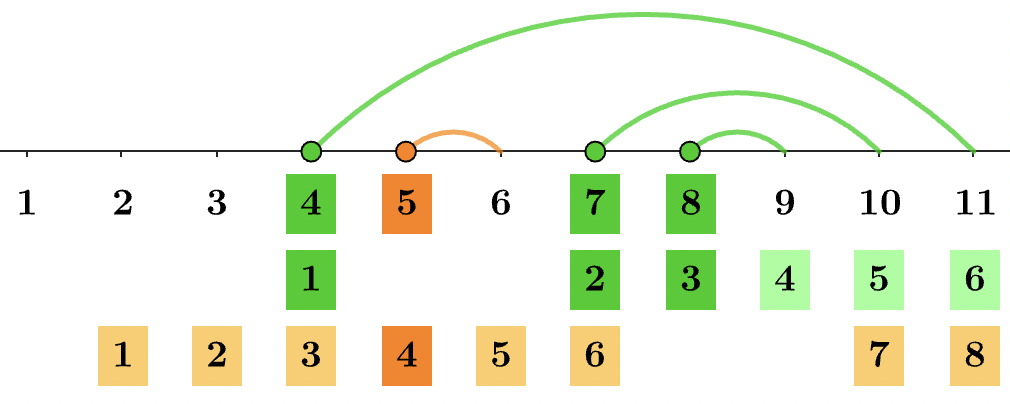}}
        \hspace{1cm}
        \begin{ytableau}
            1 & 2 & 3 & 6 & 9 & \cdots \\
            *(green) 4 & *(orange) 5 & *(green) 7 & *(green) 8 & \cdots
        \end{ytableau}
    \end{center}
    In the next step, the 8 is promoted into the top row and wraps around to become a 20. The entire rainbow now gets redrawn so that the new dot at 20 can be connected to a space after it, with our other dots now connected to spaces to their left instead of to their right. The 4 and 6 switch run membership, so the 4 becomes a member of the longer run and the 6 of the shorter run. Along both individual tracks, however, the dots all just shift one position to the left, as usual. The 4 will be promoted next.
    \begin{center}
        \raisebox{-.5\height}{\includegraphics[width=8.25cm]{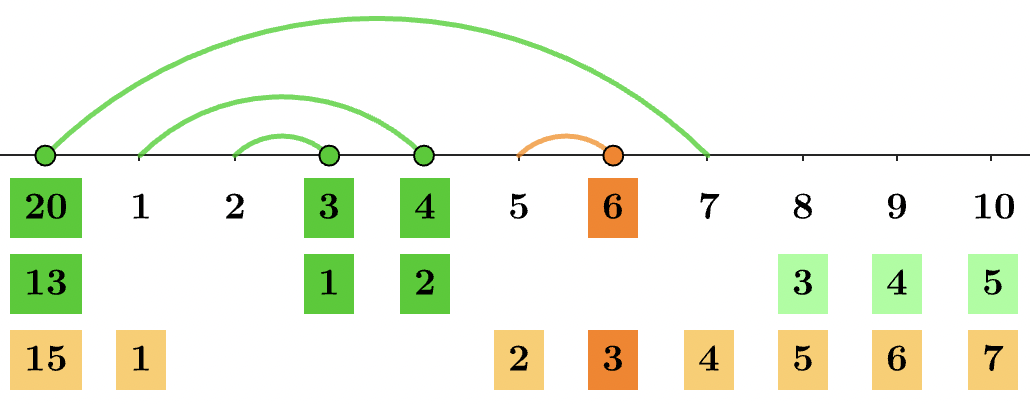}}
        \hspace{1cm}
        \begin{ytableau}
            1 & 2 & 5 & 7 & \cdots & \cdots \\
            *(green) 3 & *(green) 4 & *(orange) 6 & \cdots & *(green) 20
        \end{ytableau}
    \end{center}
    Next, the 4 gets promoted into the top row of the tableau, corresponding to its dot hopping over to the left endpoint of its arc. The dots still shift to the left on both individual tracks.
    \begin{center}
        \raisebox{-.5\height}{\includegraphics[width=8.25cm]{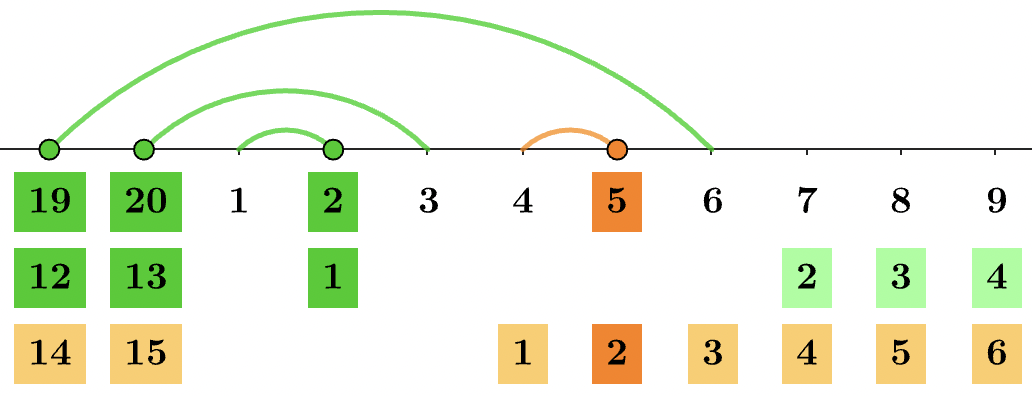}}
        \hspace{1cm}
        \begin{ytableau}
            1 & 3 & 4 & \cdots & \cdots & \cdots \\
            *(green) 2 & *(orange) 5 & \cdots & *(green) 19 & *(green) 20
        \end{ytableau}
    \end{center}
    Next, the 2 gets promoted, hopping over to the left endpoint of its arc. At this step and the next step (during which nothing gets promoted into the top row), the dots along the first track continue shifting to the left, while the dots on the second track hold still.
    \begin{center}
        \raisebox{-.5\height}{\includegraphics[width=8.25cm]{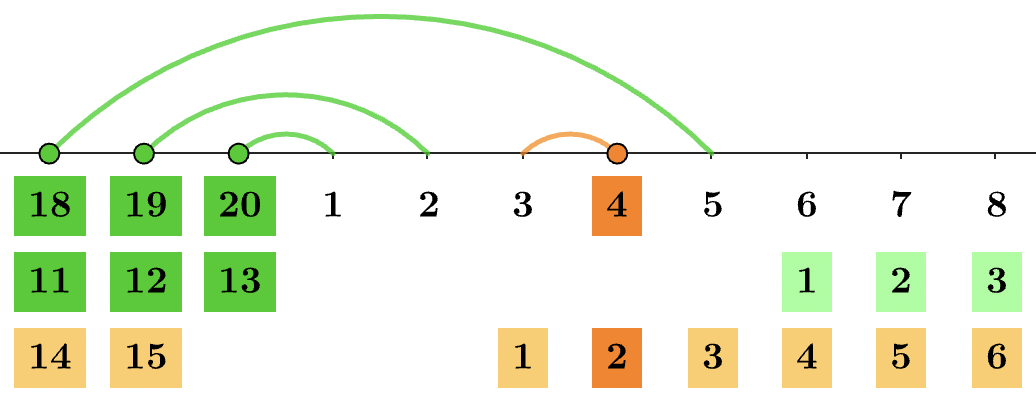}}
        \hspace{1cm}
        \begin{ytableau}
            1 & 2 & \cdots & \cdots & \cdots & \cdots \\
            *(orange) 4 & \cdots & *(green) 18 & *(green) 19 & *(green) 20
        \end{ytableau} \\
        \raisebox{-.5\height}{\includegraphics[width=8.25cm]{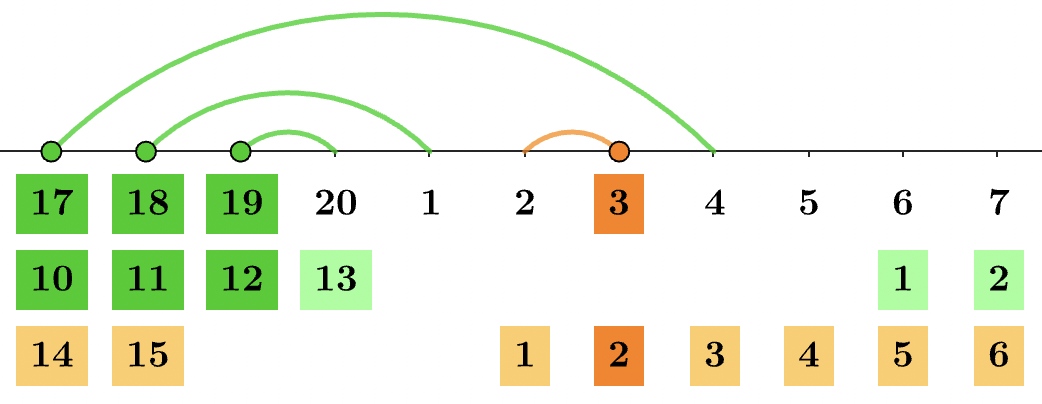}}
        \hspace{1cm}
        \begin{ytableau}
            1 & 2 & \cdots & \cdots & \cdots & \cdots \\
            *(orange) 3 & \cdots & *(green) 17 & *(green) 18 & *(green) 19
        \end{ytableau}
    \end{center}
    In the following step, still nothing gets promoted to the top row, but now the first track continues moving and the second track also resumes moving.
    \begin{center}
        \raisebox{-.5\height}{\includegraphics[width=8.25cm]{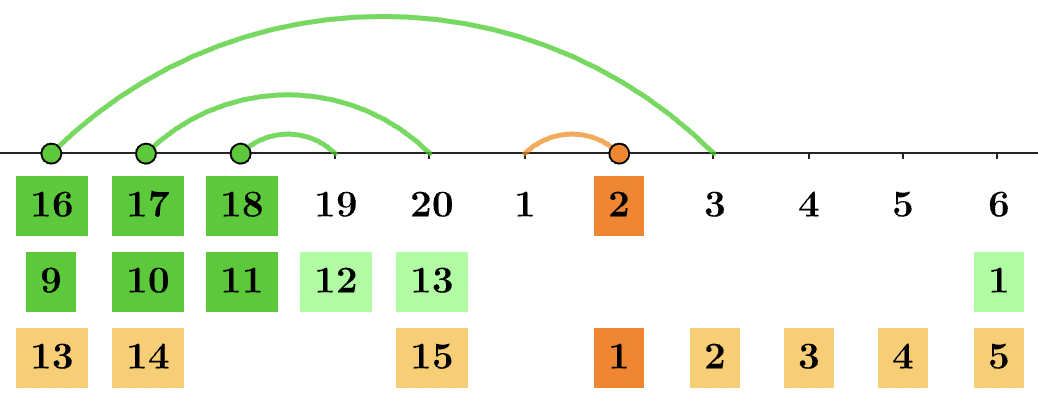}}
        \hspace{1cm}
        \begin{ytableau}
            1 & 3 & \cdots & \cdots & \cdots & \cdots \\
            *(orange) 2 & \cdots & *(green) 16 & *(green) 17 & *(green) 18
        \end{ytableau}
    \end{center}
    Then the 2 gets promoted, so during that step and the following step (where nothing gets promoted into the top row), the second track continues moving while the first track holds still:
    \begin{center}
        \raisebox{-.5\height}{\includegraphics[width=8.25cm]{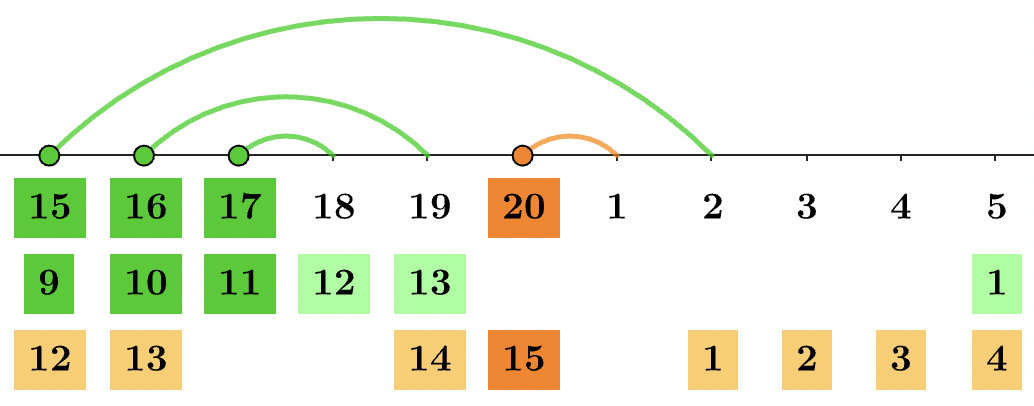}}
        \hspace{1cm}
        \begin{ytableau}
            1 & \cdots & \cdots & \cdots & \cdots & \cdots \\
            \dots & *(green) 15 & *(green) 16 & *(green) 17 & *(orange) 20
        \end{ytableau} \\
        \raisebox{-.5\height}{\includegraphics[width=8.25cm]{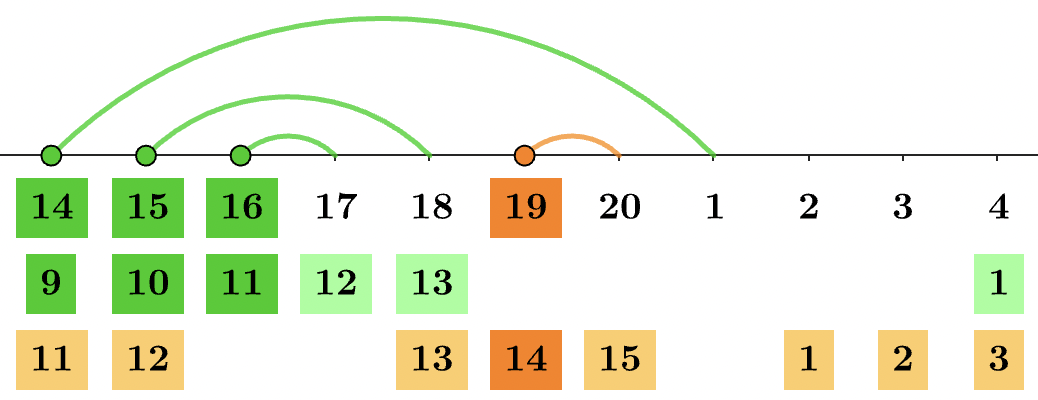}}
        \hspace{1cm}
        \begin{ytableau}
            1 & \cdots & \cdots & \cdots & \cdots & \cdots \\
            \dots & *(green) 14 & *(green) 15 & *(green) 16 & *(orange) 19
        \end{ytableau}
    \end{center}
    After that, both tracks resume moving:
    \begin{center}
        \raisebox{-.5\height}{\includegraphics[width=8.25cm]{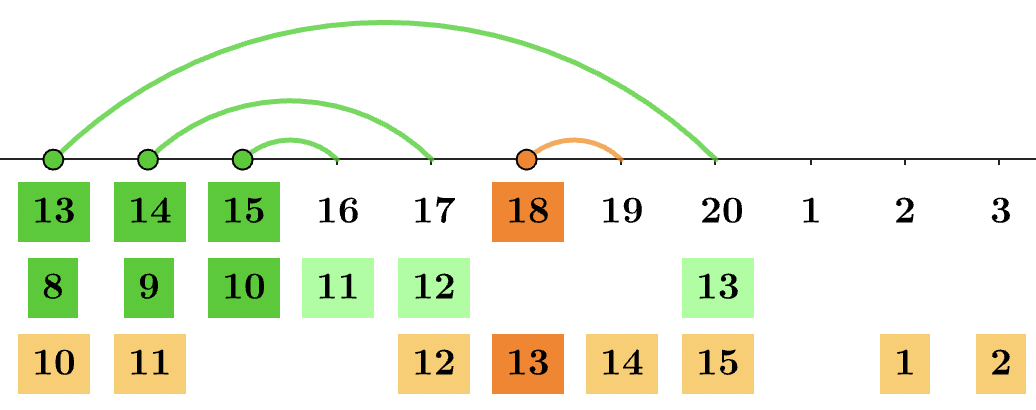}}
        \hspace{1cm}
        \begin{ytableau}
            1 & \cdots & \cdots & \cdots & \cdots & \cdots \\
            \dots & *(green) 13 & *(green) 14 & *(green) 15 & *(orange) 18
        \end{ytableau}
    \end{center}
    We can see that at the start, the runs were on top each other, but now the longer and faster-moving run has pulled ahead so that they no longer overlap.
\end{example}

\subsection{Recursive formula for the 2-row orbit length polynomials}\label{subsec:2row_recursive}

\begin{prop}\label{prop:recursion}
    $P_1(n,\ell,r) = n-\ell$, and for $d>1,$ $P_d(n,\vec{\ell},\vec{r})$ is a monic polynomial in $n$ of degree $d$. In particular, if $\vec{n}=(n_1,\dots,n_d)$, then $P_d(n,\vec{\ell},\vec{r})$ is actually a special case of a more general polynomial $P_d(\vec{n},\vec{\ell},\vec{r})$, which can can be recursively computed by $P_1(n_1,\ell,r) = n_1$ and
    \begin{align*}
        P_d(\vec{n},\vec{\ell},\vec{r}) = \prod_{i=1}^d n_i - 4\cdot\sum_{1\le i<j\le d}r_ir_j \ell_j^2 \cdot P_{i-1}(\vec{n}(i),\vec{\ell}(i),\vec{r}(i))\cdot \prod_{i<k<j}(n_k-2r_k\ell_k)\cdot\prod_{k>j}n_k,
    \end{align*}
    where for $1\le k\le i-1,$ $$\vec{n}(i) := (n_1 - 2r_1\ell_i,\dots,n_{i-1}-2r_{i-1}\ell_i),\hspace{0.5cm}\vec{\ell}(i) := (\ell_1-\ell_i,\dots,\ell_{i-1}-\ell_i),\hspace{0.5cm}\vec{r}(i) := (r_1,\dots,r_{i-1}).$$
\end{prop}

\begin{proof}
    The initial $\prod_{i=1}^d n_i$ term represents all ways to independently choose a rotation for each of the $d$ tracks, since there are $n_i$ possible rotations for the $i^{\text{th}}$ track. Then we need to subtract all cases where two or more tracks both have one of their runs ending too close to the boundary at the same time. We use casework on the first pair tracks $(i,j)$ for which this happens, with $i<j$. By ``first two tracks," we mean that $i$ is chosen to be minimal, and then $j$ is chosen to be minimal given $i$. Now we break the calculation for the number of illegal rotations in that case into four parts:
    \begin{itemize}
        \item There are $r_i$ ways to choose the run that is too close to the boundary on the $i^{\text{th}}$ track, and $r_j$ ways to choose the run too close to the boundary on the $j^{\text{th}}$ track. Then since $i<j,$ we have $\ell_i>\ell_j$, so there are $2\min(\ell_i,\ell_j)=2\ell_j$ ways to choose the position of the relevant run on each track. This gives us $r_ir_j\cdot(2\ell_j)^2 = 4r_ir_j\ell_j^2$ ways to choose the illegal rotations for the $i^{\text{th}}$ and $j^{\text{th}}$ tracks themselves.
        
        \item For each $k<i$, the $k^{\text{th}}$ track cannot have any run ending within $\ell_i$ steps of the boundary on either side, or else $(i,j)$ would not be the earliest pair of tracks with a bottleneck issue, because $(k,i)$ would be an earlier pair. Since there are $r_k$ runs on the $k^{\text{th}}$ track, this means $r_k\cdot 2\ell_i$ of the $n_k$ rotations for the $k^{\text{th}}$ track are not allowed, so we can essentially think of the $k^{\text{th}}$ track as instead having length $n_k-2r_k\ell_i.$ We can visualize this more concretely by imagining a rotation as being determined by the dot or space at position 1 on the track. Then we should remove the last $\ell_i$ dots of each run and the next $\ell_i$ spaces after them on each track, since those are the dots and spaces which cannot be in position 1.
        
        We also need to ensure that no pair $(k,m)$ with $k<m<i$ creates a bottleneck issue, which means we must ensure that we do not have one of the $r_k$ runs on the $k^{\text{th}}$ track and one of the $r_m$ runs on the $m^{\text{th}}$ track both ending within $\min(\ell_k,\ell_m)=\ell_m$ steps of the boundary. For our shortened tracks, this is equivalent to saying that we cannot have both a run on the $k^{\text{th}}$ track and a run on the $m^{\text{th}}$ track both ending within $\ell_m-\ell_i$ steps of the boundary, since now the final $\ell_m-\ell_i$ remaining dots of each shortened run (with its last $\ell_i$ dots removed) together with the next $\ell_m-\ell_i$ dots after the $\ell_i$ removed spaces following the run are not allowed to be at position 1 on the track. 
        
        But this is now exactly equivalent to the track rotations counted by the polynomial $P_{i-1},$ except that the new shortened track lengths are $\vec{n}(i)=(n_1-2r_1\ell_i,\dots,n_{i-1}-2r_{i-1}\ell_i)$, and the new shortened run lengths are $\vec{\ell}(i)=(\ell_1-\ell_i,\dots,\ell_{i-1}-\ell_i)$ since we have removed $\ell_i$ dots from each run. The list of run counts $\vec{r}(i)=(r_1,\dots,r_{i-1})$ is the same as $\vec{r}$ except that we cut it off after the first $i-1$ tracks. Putting this together, we see that the number of possible rotations for the first $i-1$ tracks in this case is $P_{i-1}(\vec{n}(i),\vec{\ell}(i),\vec{r}(i)).$
        
        \item For each $i<k<j$, the $k^{\text{th}}$ track cannot have any of its $r_k$ runs ending within $\ell_k$ steps of its track boundary, or else $(i,k)$ would be a smaller pair with a bottleneck issue than $(i,j),$ since then the $i^{\text{th}}$ and $k^{\text{th}}$ tracks would both have runs ending within $\ell_k$ steps of the boundary, since the $i^{\text{th}}$ track has a run ending within $\ell_j<\ell_k$ steps of the boundary. Thus, there are only $n_k-2r_k\ell_k$ legal rotations for the $k^{\text{th}}$ track. This is the only restriction needed, because as long as the $k^{\text{th}}$ track has no runs ending within $\ell_k$ steps of the boundary, it cannot be involved in any bottleneck issues. Thus, there are $\prod_{i<k<j}(n_k-2r_k\ell_k)$ ways to choose the rotations for all the tracks between the $i^{\text{th}}$ and $j^{\text{th}}$ tracks.
        
        \item For each $k>j,$ any rotation of the $k^{\text{th}}$ track is allowed, since even if multiple of those tracks are too close to their boundaries, it would not contradict the minimality of the pair $(i,j)$. Thus, all $\prod_{k>j}n_k$ possible rotations for these tracks are possible.
    \end{itemize}
    Putting all this together proves the recursive formula in Proposition \ref{prop:recursion}. Then since each $n_i$ is a monic linear polynomial in $n$, the product $\prod_{i=1}^d n_i$ is a monic degree $d$ polynomial in $n$. Each subtracted term has degree $d-2$ in $n$ (since the $(i,j)$ term essentially has a linear factor for each $k\ne i,j$), so those terms do not impact the overall leading coefficient. Thus, if we treat $\vec{\ell}$ and $\vec{r}$ as constants and $n$ as the variable, $P_d(n,\vec{\ell},\vec{r})$ is a monic polynomial in $n$ of degree $d$, as claimed.
    \end{proof}

To end this section, we show how to use Proposition \ref{prop:recursion} to explicitly compute the first few polynomials:

\begin{example}
    For $n = 1,$ we have $P_1(n,\ell,r) = n_1 = n-\ell,$ which we can think of as saying that all $n-\ell = n_1$ rotations of the single track are valid. For $n=2,$ we get 
    \begin{align*}
        P_2(n,\vec{\ell},\vec{r}) &= n_1n_2 - 4r_1 r_2\ell_2^2 \\
        &= (n - \ell_1 - 2r_2\ell_2)(n-\ell_2-2r_1\ell_2)-4r_1 r_2 \ell_2^2 \\
        &= n^2 - (\ell_1+\ell_2 + 2r_2\ell_2 + 2r_1\ell_2)n + \ell_1\ell_2 + 2r_1 \ell_1 \ell_2 + 2r_2 \ell_2^2.
    \end{align*} 
    To understand where the $n_1n_2 - 4r_2 r_2\ell_2^2$ comes from, we can imagine starting with all $n_1 n_2$ combinations of one of the $n_1$ rotations of the first track with one of the $n_2$ rotations of the second track. Then for each of the $r_1r_2$ pairs consisting of a run on the first track and a run on the second track, we subtract the $2\ell_2\cdot 2\ell_2 = 4\ell_2^2$ rotations where both runs end within $\ell_2$ positions of the boundary. For $n=3,$ Proposition \ref{prop:recursion} gives
    \begin{align*}
        P_3(n,\vec{\ell},\vec{r}) &= n_1n_2n_3 - 4r_1r_2\ell_2^2 \cdot n_3 - 4r_1r_3\ell_3^2\cdot (n_2 - 2r_2\ell_2) - 4r_2r_3\ell_3^2\cdot P_1(n_1 - 2r_1\ell_2,\ell_2-\ell_2,r_1) \\
        &= n_1n_2n_3 - 4r_1r_2\ell_2^2 \cdot n_3 - 4r_1r_3\ell_3^2\cdot (n_2 - 2r_2\ell_2) - 4r_2r_3\ell_3^2\cdot (n_1 - 2r_1\ell_2) \\
        &= (n-\ell_1 - 2r_2\ell_2 - 2r_3\ell_2)(n-\ell_2 - 2r_1\ell_2 - 2r_3\ell_3)(n-\ell_3 - 2r_1\ell_3 - 2r_2\ell_3) \\
        & \hspace{0.5cm}- 4r_1r_2\ell_2^2\cdot (n-\ell_3 - 2r_1\ell_3 - 2r_2\ell_3) \\
        &\hspace{0.5cm}- 4r_1r_3\ell_3^2\cdot (n-\ell_2 - 2r_1\ell_2 - 2r_3\ell_3 - 2r_2\ell_2) \\
        &\hspace{0.5cm}- 4r_2r_3 \ell_3^2 \cdot (n-\ell_1 - 2r_2\ell_2 - 2r_3\ell_3 - 2r_1\ell_2).
    \end{align*}
    Similarly, for $n=4,$ we get
    \begin{align*}
        P_4(n,\vec{\ell},\vec{r}) &= n_1 n_2 n_3 n_4 - 4r_1r_2\ell_2^2 \cdot n_3 n_4 - 4r_1r_3\ell_3^2\cdot (n_2-2r_2\ell_2)n_4 - 4r_1r_4\cdot (n_2 - 2r_2\ell_2)(n_3 - 2r_3\ell_3) \\
        &\hspace{2.2cm}-4r_2r_3\ell_3^2\cdot(n_1-2r_1\ell_2)n_4 - 4r_2r_4\cdot(n_1-2r_1\ell_2)(n_3-2r_3\ell_3) \\
        &\hspace{2.2cm}-4r_3r_4\cdot P_2((n_1-2r_1\ell_3,n_2-2r_2\ell_3),(\ell_1-\ell_3,\ell_2-\ell_3),(r_1,r_2)) \\
        &= n_1 n_2 n_3 n_4 - 4r_1r_2\ell_2^2 \cdot n_3 n_4 - 4r_1r_3\ell_3^2\cdot (n_2-2r_2\ell_2)n_4 - 4r_1r_4\cdot (n_2 - 2r_2\ell_2)(n_3 - 2r_3\ell_3) \\
        &\hspace{2.2cm}-4r_2r_3\ell_3^2\cdot(n_1-2r_1\ell_2)n_4 - 4r_2r_4\ell_4^2\cdot(n_1-2r_1\ell_2)(n_3-2r_3\ell_3)\\
        &\hspace{2.2cm}-4r_3r_4\ell_4^2\cdot((n_1-2r_1\ell_3)(n_2-2r_3\ell_3)-4r_1r_2(\ell_2-\ell_3)^2).
    \end{align*}
\end{example}


\section{Orbit lengths for near-hooks}\label{sec:hook_plus_box}

In this section, we characterize the orbits when $T[n]$ consists of a hook plus a single extra box in the second row, which we will call a \emph{\tb{\tcb{near-hook}}}. In \S\ref{subsec:hook_box_linear}, we describe the orbits whose lengths divide linear polynomials in $n$, which are either the orbits from Theorem \ref{thm:generic} where all entries of $T$ differ by at least 2, or orbits where all entries of $T$ belong to runs of 2 or more consecutive numbers (regardless of the exact run lengths). Then in \S\ref{subsec:hook_box_quadratic}, we show that the remaining orbit lengths divide certain quadratic polynomials in $n$, and we give an explicitly formula for those polynomials in terms of the number $r$ of runs of 2 or more consecutive numbers in $T$, and the number $s$ of singleton entries of $T$ not in a run.


\subsection{Linear orbit lengths for near-hooks}\label{subsec:hook_box_linear}

The first type of linear orbits we get for near-hooks are the ones from the generic case in Theorem \ref{thm:generic}:

\begin{prop}\label{prop:hook_box_linear_1}
    If $T[n]$ is a near-hook and all entries of $T$ differ by at least 2, and 2 and $n$ are not both entries of $T$, then $\pr(T[n])$ divides $(|T|-1)\cdot(n-1)$.
\end{prop}

\begin{proof}
    From Theorem \ref{thm:generic}, $\pr(T[n])$ divides $$\frac{\lcm(\pr(T),|T|)}{|T|}\cdot(n-1).$$ We will show that $\pr(T)=|T|-1$, and hence $\lcm(\pr(T),|T|)/|T|=|T|-1.$ Recall that $\pr(T)=\pr(\SYT(T))$, so it suffices to check that $\pr(\SYT(T))=1$. Since $T[n]$ is a near-hook, $T$ has shape $21^{|T|-1}$, and there are $f^{21^{|T|-1}}=|T|-1$ SYT of that shape, since there are $|T|-1$ choices for the top right entry. These $|T|-1$ tableaux are all in the same promotion orbit, because the top right entry decreases by 1 at each promotion step, except that once it reaches 2, it skips to $|T|$ in the next step. Thus, the orbit has length $\pr(T)=\pr(\SYT(T))=|T|-1$, so the orbit length $\pr(T[n])$ is as claimed.
\end{proof}

The other linear case we get is when all the entries of $T$ belong to runs of length more than 1. We will first describe exactly what we mean by that:

\begin{definition}
    Write $\T_{\tn{near-hook}}(n,r)$ for the set of near-hook tableaux of size $n$ such that the entries of $T$ can be partitioned into exactly $r$ runs of 2 or more consecutive numbers, including possibly a run that wraps around from $n$ to 2, and such that within the first column of $T$, the mod $n-1$ difference between the last number of each run and the first number of the next is always at least 3. 
\end{definition}

Note that the condition on the first column forces the top right number of $T$ to be either the last or first entry in a run, not some middle entry. Note also that the difference between the last entry of one run and the first entry of the next may be 2 if one of those entries is the top right entry of $T$, since the condition on the differences only applies to the last and first entries of the runs within the first column.

\begin{prop}\label{prop:hook_linear_closed}
    $\T_{\tn{near-hook}}(n,r)$ is closed under promotion, and the gaps within the first column between the last entry of each run and the first entry of the next run stay the same throughout the orbit.
\end{prop}

\begin{proof}
    For $n$ large, at the majority of promotion steps, all that happens is that either the entries of $T$ all decrease by 1. The interesting steps are when an entry of $T$ gets promoted into the first row. If that entry is a 2 in the top left corner of $T$, then the 2 becomes an $n$ next, which means that the run the 2 belonged to is still a run of the same length: if the 2 was the first number in the run, the run now becomes a wraparound run, if the 2 was somewhere in the middle of the run, it was already a wraparound run and it remains one, and if the 2 was the last number of the run, then it was a wraparound run but no longer is after the 2 gets promoted. In all three of those cases, all the run lengths are preserved, and the mod $n-1$ gaps between the runs are also preserved, because all entries of $T$ decrease by 1 mod $n-1$, including the 2 that becomes an $n$. Thus, the tableau stays in $\T_{\tn{near-hook}}(n,r)$ at those steps.

    The other possibility is that an entry larger than 2 gets promoted from the top right box of $T$ into the top row. If that entry is $\ell$, then the entries $2,3,\dots,\ell-1$ must all be either in the first column of $T$ or immediately above the $\ell$. For the $\ell$ to get promoted out of $T$ instead of a 2, the 2 must not be in $T$, so it must be above $\ell$, meaning $3,4,\dots,\ell-1$ are in the first column of $T$, and thus $\ell$ must be the largest element of its run. Then when the $\ell$ gets promoted, it gets replaced by an $n$ in the top right box of $T$, while all other entries of $T$ decrease by 1. The other entries in its run decrease by 1 to become $2,3,\dots,\ell-2,$ so the run is still a run of the same length, but has now become a wraparound run $n,2,3,\dots,\ell-2$. All other entries of $T$ simply decrease by 1, so the condition on the gaps between runs within the first column of $T$ still holds. Thus, the tableau stays in $\T_{\tn{near-hook}}(n,r)$.
\end{proof}

\begin{theorem}\label{thm:hook_box_linear_2}   
    for every $T[n]\in \T_{\tn{near-hook}}(n,r)$, the orbit length $\pr(T[n])$ divides $(2r-1)n - 2r.$
\end{theorem}

\begin{proof}
    Label the runs in $T$ as the first, second, $\dots,r^{\text{th}}$ run in the order they appear in the column, and let their lengths be $\ell_1,\ell_2,\dots,\ell_r$. Note that $\ell_1,\dots,\ell_r$ are not necessarily in decreasing order, and they may or may not be equal to each other, since we allow the run lengths to take any values as long as they are all at least 2, and we are ordering them based on where the runs appear in $T$, not based on their lengths. Let $g_i$ be the mod $n-1$ difference between the last entry within the first column of the $(i-1)^{\text{st}}$ run and the first entry within the first column of the $i^{\text{th}}$ run, where $i$ is taken mod $r$, so $g_1$ is the gap between the last entry of the $r^{\text{th}}$ run and the first entry of the first run.

    \begin{lemma}
        The sum of the run lengths and gaps is $$\ell_1+\dots+\ell_r+g_1+\dots+g_r=n+r.$$
    \end{lemma}

    \begin{proof}
        If we start from the beginning of the first run, adding $\ell_1-1 \pmod{n-1}$ will take us to the end of the first run. Then adding $g_2\pmod{n-1}$ will take us to the start of the second run, adding $\ell_2-1\pmod{n-1}$ will take us to the end of the second run, and so on. Adding $\sum_{i=1}^r(\ell_i-1+g_i)\pmod{n-1}$ almost takes us back to the same number from the first column that we started at, so it would seem that the sum $\sum_{i=1}^r(\ell_i-1+g_i)$ should equal $n-1$, but actually it equals $n$ because the extra number in the top right box of $T$ gets double counted in both one of the $\ell_i-1$ terms and in one of the $g_i$ terms, since the gap containing the top right number is counted as the difference between the number before the top right number in the first column and the number after the top right number in the first column, rather than between the top right number itself and either the number before or after it. Thus, we have $\sum_{i=1}^r(\ell_i-1+g_i)=n+r,$ so $\sum_{i=1}^r(\ell_i+g_i)=n+r.$
    \end{proof}
    
    Now suppose we start the orbit with the first run containing the numbers $n-\ell_1+1,\dots,n-1$ at the bottom of $T$, together with $n$ in the top right corner of $T$. For the first portion of the orbit, all the entries in $T$ will just decrease by 1 at every step, except that when each entry of $T$ decreases to 2, it slides up and wraps around to become an $n$ at the bottom of the column, while all other entries in the column slide up. The top right box of $T$ just decreases by 1 at every step of this process, until we get to a point where that box has entry $\ell_1+2,$ and the entries at the top of the column of $T$ are the other entries $3,4,\dots,\ell_1+1$ from the same run:
    \begin{center}
    \ytableausetup{boxsize=1cm}
    \begin{ytableau}
        1 &  & \cdots \\
         \vdots & *(green) n \\
        *(green) \scriptstyle{n-\ell_1+1} \\
        *(green) \vdots \\
        *(green) n-1 \\
    \end{ytableau}
    $\longrightarrow$
    \begin{ytableau}
        1 &  & \cdots \\
         \vdots & *(green) n-1 \\
        *(green) \scriptstyle{n-\ell_1} \\
        *(green) \vdots \\
        *(green) n-2 \\
    \end{ytableau}
    $\longrightarrow \dots \longrightarrow$
    \begin{ytableau}
        1 &  & \cdots \\
        *(green) 3 & *(green) \ell_1+2 \\
        *(green) \vdots \\
        *(green) \ell_1+1 \\
        \vdots \\
    \end{ytableau}
    \end{center}
    At this point, the $\ell_1+2$ gets promoted next and gets replaced by an $n$ in the top right box of $T$. Then the remaining entries of the run are $2,3,\dots,\ell_1$ in the first column, so they slide up and get promoted one by one, with a 2 from the first run wrapping around to become an $n$ at the bottom of $T$ at each step. After they have all been promoted, the first run will have wrapped around so its entries are $n-\ell_1+1,\dots,n,$ with its \emph{smallest} entry (the $n-\ell_1+1$) in the top right box of $T$, and all its larger entries at the bottom of $T$. This portion of the process is summarized below, with the entries of the run in green:
    \begin{center}
        \ytableausetup{boxsize=1cm}
        \begin{ytableau}
        1 & 2 & \cdots \\
        *(green) 3 & *(green) \ell_1+2 \\
        *(green) 4 \\
        *(green) \vdots \\
        *(green) \ell_1+1 \\
        \vdots \\
    \end{ytableau}
    $\longrightarrow$
    \begin{ytableau}
        1 & \ell_1+1 & \cdots \\
        *(green) 2 & *(green) n \\
        *(green) 3 \\
        *(green) \vdots \\
        *(green) \ell_1 \\
        \vdots \\
    \end{ytableau}
    $\longrightarrow$
    \begin{ytableau}
        1 & \ell_1 & \cdots \\
        *(green) 2 & *(green) n-1 \\
        *(green) \vdots \\
        *(green) \ell_1-1 \\
        \vdots \\
        *(green) n
    \end{ytableau}
    $\longrightarrow\dots\longrightarrow$
    \begin{ytableau}
        1 &  & \cdots \\
         & *(green) \scriptstyle{n-\ell_1+1} \\
        \vdots \\
        *(green) \scriptstyle{n-\ell_1+2} \\
        *(green) \vdots \\
        *(green) n
    \end{ytableau}
    \end{center}
    Next, all the entries belonging to other runs in $T$ will simply slide up and wrap around from being a 2 at the top of $T$ to being an $n$ at the bottom of $T$, until we get to a point where the portion of the first run that is in the first column starts just below the top left box of $T$. Then the smallest entry of the first run will slide over to the left to become the top left entry of $T$, so the first run will be entirely in the first column of $T$. The first entry of the first run slides left at the step where the last entry of the $r^{\text{th}}$ run is a 2 and then wraps around to become an $n$ in the top right corner of $T$. Since the mod $n-1$ gap between the last entry of the $r^{\text{th}}$ run within the left column and the first entry of the first run within the left column must be $g_1$, the first entry of the first run within the column is $g_1+2$ at that step, so the entry in the top right box is $g_1+1$, and then is $g_1$ when it slides left. In the figure below, entries of the first run are still highlighted in green, while entries of the $r^{\text{th}}$ run are highlighted in orange:
    \begin{center}
    \begin{ytableau}
        1 & 2 & \cdots \\
         & *(green) \scriptstyle{n-\ell_1+1} \\
        \vdots \\
        *(green) \scriptstyle{n-\ell_1+2} \\
        *(green) \vdots \\
        *(green) n
    \end{ytableau}
    $\longrightarrow$
    \begin{ytableau}
        1 & 2 & \cdots \\
         & *(green) {n-\ell_1} \\
        \vdots \\
        *(green) \scriptstyle{n-\ell_1+1} \\
        *(green) \vdots \\
        *(green) n-1
    \end{ytableau}
    $\longrightarrow \dots \longrightarrow$
    \begin{ytableau}
        1 & 2 & \cdots \\
        *(orange) 2 & *(green) g_1+1 \\
        *(green) g_1+2 \\
         \vdots \\
         *(orange) \vdots \\
         *(orange) n
    \end{ytableau}
    $\longrightarrow$
    \begin{ytableau}
        1 & 2 & \cdots \\
        *(green) g_1 & *(orange) n \\
         \vdots \\
         *(orange) \scriptstyle{n-\ell_r+1} \\
         *(orange) \vdots \\
         *(orange) n-1
    \end{ytableau}
    \end{center}
    In our process so far, the top right entry of $T$ starts as $n$, then decreases down to $\ell_1+2$, which takes $n-\ell_1-1$ steps. Then it decreases from $n$ down to $g_1+1$, taking $n-g_1$ steps. Thus, we get a total of $2n-\ell_1-g_1-1$ steps where the top right number in $T$ belongs to the first run.

    After that, we are essentially back at the starting state, with the top right entry of $T$ again being an $n$ and the rest of its run being at the bottom of $T$, except that the $n$ is now the last entry of the $r^{\text{th}}$ run instead of the first run. Thus, the above process will repeat but with the $r^{\text{th}}$ run playing the role of the first run, taking $2n-\ell_r-g_r-1$ steps. Then the same thing will happen with the top right box belonging to the $(r-1)^{\text{st}}$ run, then to the $(r-2)^{\text{nd}}$ run, and so on. After we have looped through the states where the top right box of $T$ belongs to each run, we will be back at the starting state. Thus, the total orbit length is a divisor of $$\sum_{i=1}^r (2n-\ell_i-g_i-1) = r\cdot 2n - \sum_{i-1}^t(\ell_i+g_i) -r = 2rn - (n+r)-r = (2r-1)n-2r,$$ as claimed.
\end{proof}

\begin{example}
    For the tableau $T[9]$ shown below, we have 2 runs, of lengths 2 and 3, shown in green and orange. At first the green run has its largest entry in the top right, then it has its smallest entry in the top right, then the orange run has its largest entry in the top right, then its smallest. Since $r=2$, we get an orbit length of $3\cdot9-4 = 23.$
    \begin{center}
        \ytableausetup{boxsize=0.45cm}
        \begin{tabular}{ccc ccc cccc}
        & \begin{ytableau}
            1 & 2 & 5 & 6 \\
            *(orange) 3 & *(green) 9 \\
            *(orange) 4 \\
            *(green) 7 \\
            *(green) 8
        \end{ytableau}
        & $\longrightarrow$ &
        \begin{ytableau}
            1 & 4 & 5 & 9 \\
            *(orange) 2 & *(green) 8 \\
            *(orange) 3 \\
            *(green) 6 \\
            *(green) 7
        \end{ytableau}
        & $\longrightarrow$ &
        \begin{ytableau}
            1 & 3 & 4 & 8 \\
            *(orange) 2 & *(green) 7 \\
            *(green) 5 \\
            *(green) 6 \\
            *(orange) 9
        \end{ytableau}
        & $\longrightarrow$ &
        \begin{ytableau}
            1 & 2 & 3 & 7 \\
            *(green) 4 & *(green) 6 \\
            *(green) 5 \\
            *(orange) 8 \\
            *(orange) 9
        \end{ytableau}
        & $\longrightarrow$ &
        \begin{ytableau}
            1 & 2 & 6 & 9 \\
            *(green) 3 & *(green) 5 \\
            *(green) 4 \\
            *(orange) 7 \\
            *(orange) 8
        \end{ytableau} \\ \\

        $\longrightarrow$ &
        \begin{ytableau}
            1 & 4 & 5 & 8 \\
            *(green) 2 & *(green) 9 \\
            *(green) 3 \\
            *(orange) 6 \\
            *(orange) 7
        \end{ytableau}
        & $\longrightarrow$ &
        \begin{ytableau}
            1 & 3 & 4 & 7 \\
            *(green) 2 & *(green) 8 \\
            *(orange) 5 \\
            *(orange) 6 \\
            *(green) 9
        \end{ytableau}
        & $\longrightarrow$ &
        \begin{ytableau}
            1 & 3 & 4 & 7 \\
            *(orange) 4 & *(green) 7 \\
            *(orange) 5 \\
            *(green) 8 \\
            *(green) 9
        \end{ytableau}
        & $\longrightarrow$ &
        \begin{ytableau}
            1 & 2 & 5 & 9 \\
            *(orange) 3 & *(green) 6 \\
            *(orange) 4 \\
            *(green) 7 \\
            *(green) 8
        \end{ytableau}
        & $\longrightarrow$ &
        \begin{ytableau}
            1 & 4 & 8 & 9 \\
            *(orange) 2 & *(green) 5 \\
            *(orange) 3 \\
            *(green) 6 \\
            *(green) 7
        \end{ytableau} \\ \\

        $\longrightarrow$ &
        \begin{ytableau}
            1 & 3 & 7 & 8 \\
            *(orange) 2 & *(green) 4 \\
            *(green) 5 \\
            *(green) 6 \\
            *(orange) 9
        \end{ytableau}
        & $\longrightarrow$ &
        \begin{ytableau}
            1 & 2 & 6 & 7 \\
            *(green) 3 & *(orange) 9 \\
            *(green) 4 \\
            *(green) 5 \\
            *(orange) 8
        \end{ytableau}
        & $\longrightarrow$ &
        \begin{ytableau}
            1 & 5 & 6 & 9 \\
            *(green) 2 & *(orange) 8 \\
            *(green) 3 \\
            *(green) 4 \\
            *(orange) 7
        \end{ytableau}
        & $\longrightarrow$ &
        \begin{ytableau}
            1 & 4 & 5 & 8 \\
            *(green) 2 & *(orange) 7 \\
            *(green) 3 \\
            *(orange) 6 \\
            *(green) 9
        \end{ytableau}
        & $\longrightarrow$ &
        \begin{ytableau}
            1 & 4 & 5 & 8 \\
            *(green) 2 & *(orange) 6 \\
            *(orange) 5 \\
            *(green) 8 \\
            *(green) 9
        \end{ytableau} \\ \\

        $\longrightarrow$ &
        \begin{ytableau}
            1 & 2 & 3 & 6 \\
            *(orange) 4 & *(orange) 5 \\
            *(green) 7 \\
            *(green) 8 \\
            *(green) 9
        \end{ytableau}
        & $\longrightarrow$ &
        \begin{ytableau}
            1 & 2 & 5 & 9 \\
            *(orange) 3 & *(orange) 4 \\
            *(green) 6 \\
            *(green) 7 \\
            *(green) 8
        \end{ytableau}
        & $\longrightarrow$ &
        \begin{ytableau}
            1 & 3 & 4 & 8 \\
            *(orange) 2 & *(orange) 9 \\
            *(green) 5 \\
            *(green) 6 \\
            *(green) 7
        \end{ytableau}
        & $\longrightarrow$ &
        \begin{ytableau}
            1 & 3 & 4 & 8 \\
            *(green) 4 & *(orange) 8 \\
            *(green) 5 \\
            *(green) 6 \\
            *(orange) 9
        \end{ytableau}
        & $\longrightarrow$ &
        \begin{ytableau}
            1 & 2 & 6 & 9 \\
            *(green) 3 & *(orange) 7 \\
            *(green) 4 \\
            *(green) 5 \\
            *(orange) 8
        \end{ytableau} \\ \\

        $\longrightarrow$ &
        \begin{ytableau}
            1 & 5 & 8 & 9 \\
            *(green) 2 & *(orange) 6 \\
            *(green) 3 \\
            *(green) 4 \\
            *(orange) 7
        \end{ytableau}
        & $\longrightarrow$ &
        \begin{ytableau}
            1 & 5 & 8 & 9 \\
            *(green) 2 & *(orange) 5 \\
            *(green) 3 \\
            *(orange) 6 \\
            *(green) 9
        \end{ytableau}
        & $\longrightarrow$ &
        \begin{ytableau}
            1 & 3 & 6 & 7 \\
            *(green) 2 & *(orange) 4 \\
            *(orange) 5 \\
            *(green) 8 \\
            *(green) 9
        \end{ytableau} &&&
        \end{tabular}
    \end{center}
\end{example}

The reason why the orbit length might be a divisor of $(2r-1)n-2r$ instead of necessarily equaling $(2r-1)n-2r$ is exactly the same as for the 2-row linear case in \S\ref{sec:2row_linear}. Namely, the sequence of run lengths and gaps could have $d$-fold symmetry for some $d\mid r$, where the same lengths and gaps get repeated $r/d$ times, so that the length of the orbit gets divided by $d$.

\subsection{Quadratic orbit lengths for near-hooks}\label{subsec:hook_box_quadratic}

Like in \S\ref{subsec:run_lengths}, we will need to describe how to read off the singletons and the run lengths for an arbitrary near-hook shaped tableaux that includes both runs and singletons, since sometimes the singletons may get sandwiched inside the runs. However, the process will be simpler here than in the 2-row case, because singletons getting sandwiched inside runs is essentially the only kind of overlap that can happen in the near-hook case, since all runs will essentially move together at the same rate, while all singletons will move together at a different rate. 

We can identify these ``nested" singletons by looking for mod $n-1$ differences of 2 in the first column of $T$, which we said were not allowed in the runs-only case:
\begin{enumerate}
    \item If two consecutive numbers in the first column differ by 2 mod $n-1$, we count the number \emph{before} the gap as a singleton. If we have an $i-1,i,$ and $i+2$ in the column, then the $i$ is the singleton and the $i-1$ and $i+2$ belong to the same run. If there is also an $i+1$ in the top right box, then the $i+1$ also belongs to the same run as the $i-1$ and $i+2$. Multiple singletons in a row may be nested within the same run, with gaps of 2 between them.
    
    
    \item Any other number with a gap of at least 3 to all other numbers, or in the top right box with a gap of at least 2 to all other numbers, gets counted as a singleton, and any other gap of 3 or more between adjacent first-column numbers is a break between runs.
\end{enumerate}
Write $\T_{\tn{near-hook}}(n,r,s)$ for the set of near-hook tableaux with $s$ singletons and $r$ runs of length 2 or more.

\begin{example}
    In the tableau $T[16]$ below, the gaps are large enough that the runs are just the sequences of consecutive numbers not in the first row:
    \begin{center}
        \ytableausetup{boxsize=0.6cm}
        \begin{ytableau}
            1 & 2 & 3 & 6 & 7 & 8 & 12 & 13 & 14 & 16 \\
            *(green) 4 & *(orange)15 \\
            *(green) 5 \\
            *(yellow) 9 \\
            *(yellow) 10 \\
            *(yellow) 11 \\
        \end{ytableau}
    \end{center}
    The runs are $(4,5)$ and $(9,10,11)$, highlighted in green and yellow, while 15 is a singleton, highlighted in orange. Thus, $r=2$ and $s=1$.
\end{example}

\begin{example}\label{ex:near-hook}
    In the tableau $T[15]$ below, the runs and singletons are not quite what we might expect, since some singletons are nested inside runs. The singletons are 3, 9, and 15 (highlighted in orange), while the runs are $(14,4,5)$ and $(8,11)$, with the run $(14,4,5)$ highlighted in green and the run $(8,11)$ in yellow:
    \begin{center}
        \ytableausetup{boxsize=0.6cm}
        \begin{ytableau}
        1 & 2 & 6 & 7 & 10 & 12 & 13 \\
        *(orange) 3 & *(green) 4 \\
        *(green) 5 \\
        *(yellow) 8 \\
        *(orange) 9 \\
        *(yellow) 11 \\
        *(green) 14 \\
        *(orange) 15
        \end{ytableau}
    \end{center}
    Thus, $r=3$ and $s=2$. Note that the 15 counts as a nested singleton since $n-1=14$ and thus the gap from the 15 to the 3 is $3-15=-12\equiv 2\pmod{n-1}$.
\end{example}

\begin{theorem}\label{thm:hook_box_quadratic}
    $\T_{\tn{near-hook}}(n,r,s)$ is closed under promotion, and for every $T[n]\in \T_{\tn{near-hook}}(n,r,s),$ the orbit length $\pr(T[n])$ is a divisor of the quadratic polynomial
    $$(n-2r-1)\cdot((2r+s-1)(n-2s)-2r) \ - \ 2rs(4r+2s-3).$$
\end{theorem}

\begin{proof}
    The idea is similar to the 2-row case. We will think of the singletons as moving together around one circular track, and the runs as moving together around another circular track that actually transitions between $2r+s$ different states depending on whether the top right box contains the smallest number of one of the $r$ runs, the largest number of one of the $r$ runs, or one of the $s$ singletons (or a number in the middle of a run right after the singleton). The run track will have different numbers of legal rotations depending on its current state. A singleton will generally not be allowed to have just crossed the boundary or be 1 step past the boundary at the same time the end of a run is 1 or 2 steps past the boundary, which is where the subtracted term comes from. Usually both tracks will rotate by 1 position at each promotion step, except that the run track will pause when a singleton is crossing the boundary and for the following step (with some subtleties depending on the state), and the singleton track will pause during the final step when any run crosses the boundary and for the following step. This ensures that the condition about runs and singletons not ending too close to the boundary at the same time will always hold throughout the orbit. We can construct the two tracks as follows:
    \begin{enumerate}
        \item We first draw a merged track of length $n$, labeled with the numbers $1,2,\dots,n$, with dots marking the numbers in $T$. From the procedure above, we can also mark each dot as being either a singleton or part of a run.
        \item We number the singleton track with the numbers $1,2,\dots,n-2r-1$ by following the numbering of the merged track, but skipping over the 2 numbers immediately after the last number of each run that lies in the left column (wrapping around from $n$ to 1), as well as the first number not skipped. We mark a dot on the singleton track for every singleton dot on the original track.
        \item Similarly, we number the run track with the numbers $1,2,\dots,n-2s-1$ by following the numbering of the merged track but skipping over each singleton and the number after it (again wrapping around from $n$ to 1), as well as the first number not skipped, except that if the top right box is a singleton or a run member that is 1 more than a singleton, we only skip 1 for that singleton, so we get the numbers $1,2,\dots,n-2s$ instead.
    \end{enumerate}
    

    We will say that the run track is in a ``run state" if the top right entry of $T$ is the first or last entry of a run, and is in a ``singleton state" if either the top right entry is a singleton, or it is a number partway through a run, in which case the number 1 less than it must be a singleton in the left column.

    Once we have these tracks and states, the gaps on the singleton track will simply be the differences between consecutive singletons along the track. On the run track, the gaps will be the distances between the last dot \emph{in the first column} of each run and the first dot \emph{in the first column} of the next run (so omitting the dot corresponding to the top right entry of $T$, if it is a dot).
    
    We now break the proof of Theorem \ref{thm:hook_box_quadratic} into several lemmas:

\begin{lemma}\label{lem:nearhook_num_runs}
    The number of runs, number of singletons, run lengths, and gaps between the runs stay the same at every promotion step.
\end{lemma}

\begin{proof}
    The promotion steps we need to check are the ones where there is a potential state transition, i.e. when either the top right entry of $T$ gets promoted into the top row, or the top left entry of $T$ gets promoted into the top row and the top right entry slides left to fill its position, since otherwise, all entries in the left column simply decrease by 1 mod $n-1$, possibly with one of them wrapping around from the top of the column to the bottom if a 2 becomes an $n$, and in all those cases, the mod $n-1$ gaps between all the numbers do not change, so the runs, singletons, run lengths, and gaps also do not change. 
    
    We can split these trickier promotion steps into several cases:
    \begin{itemize}
        \item \tb{Case 1:} If the top right entry is the last entry of a run, then it generally gets promoted before all the other entries of its run, getting immediately replace by an $n$ in the same box, which is now the first entry of that run, since the run has now become a wraparound run. This looks exactly the same as what we saw in the runs-only case in Theorem \ref{thm:hook_box_linear_2}. Since nothing changed within the left column of $T$, the runs, singletons, and gaps between them do not change at such a step.

        The exception is that if there are one or more singletons nested within the run, the last entry of the run will instead get promoted immediately after the last singleton. Before that promotion step, the singleton will be an $n$ and the next smallest run entry will be a 3. After the promotion step, we will get an $n-1$ at the bottom of the column, an $n$ in the top right box, and a 2 in the top left box. Thus, the $n$ will be counted as a singleton nested within the run. The gaps between this singleton and other singletons will not change, because while the singleton stayed at $n$ instead of decreasing by 1 along with the other singletons, it is now a top right nested singleton and hence it gets counted as being one position earlier on the singleton track. For the run, its first and last entries within the column just decrease by 1 at the promotion step along with the first and last entries of all other runs, so the gaps between the runs also stay the same.
        
        \item \tb{Case 2:} If the top right entry gets promoted into the top row while the run track is in a singleton state, then if that entry is equal to $\ell$ before it gets promoted, the entries $3,4,\dots,\ell-1$ must all be in the left column of $T$, and also $\ell+1$ must be in the left column of $T$, or else $\ell$ and we would not be in a singleton state. Thus, the singleton is the $\ell-1$ and is nested inside the run. After the promotion step, the $\ell$ gets replaced by an $n$ in the same box, while we now have $2,3,\dots,\ell-2,\ell$ all in the first column. 
        
        If the 3 (now a 2) was the first entry of the run, the $n$ is now a member of the run, and the $\ell-2$ is a singleton nested inside the run. Thus, the run length stays the same, its gaps to other runs stay the same (since only something not in the first column changed), and the number of singletons stays the same. The singleton just decreased by 1 from $\ell-1$ to $\ell-2$, so the gaps on the singleton track also stay the same.

        The only way for the 3 (now a 2) to not be the first entry of the run is if the $n$ is a singleton nested inside the run. Thus, after the promotion step, we have both an $n-1$ and a 2 in the run, and an $n$ in the top right box. Thus, after that step, the $n-1$ is still a singleton nested inside the run, so that nested singleton also just decreased by 1 along with all the other singletons. Thus, the number of singletons and the gaps between the singletons is still preserved in this case.

        \item \tb{Case 3:} If the top right entry of $T$ is the smallest entry of a run that slides left when the top left entry of $T$ gets promoted into the top row, then the promoted entry could have been either a singleton or the last entry of the previous run. 
        
        If it was the last entry of the previous run, then the 2 that gets promoted from the top left box of $T$ gets replaced by an $n$ in the top right box of $T$, and we get a transition that looks exactly the same as in the linear case from Theorem \ref{thm:hook_box_linear_2}. Thus, none of the run lengths, run gaps, or singleton gaps change. 
        
        If the top left entry was a singleton, then we transition to a state where the top right box contains that singleton instead, and the singleton will not be nested within the run in either case. The gaps between the singletons do not change because a singleton in the top right corner only moves ahead on the track if it is nested within a run. The gap between the runs on either side of the singleton would have decreased by 1 because the later run now has an extra entry in the first row at the beginning, except that we now only skip 1 on the run track for the top right singleton instead of the 2 spaces we skipped when that singleton was in the first column, so we have also added an extra space between the runs, and thus the gap between them remains the same.

        \item \tb{Case 4:} The final possibility is that the top right entry is a singleton that slides left when a 2 gets promoted from the top left box of $T$ into the top row of $T[n]$.

        If the top left entry that gets promoted belongs to a run, then the singleton cannot be nested within that run, since a singleton in the top right box cannot be nested within a run. Thus, the last entry of that run turns from a 2 in the top left to an $n$ in the top right, while the singleton slides left. Then the gap between that run and the run before it increases by 1 because the $n$ at the start of the run is not in the left column anymore and so does not get counted, but it also decreases by 1 because only 1 space between the runs got skipped while the singleton was in the column, but 2 spaces get skipped now that the singleton is in the left column instead. Thus, the gap between that run and the run before it stays the same. For the singleton, it just decreases by 1, so its gaps to other singletons remain the same.

        The other possibility is that the top left entry that gets promoted is another singleton. Then those two singletons both decrease by 1 mod $n-1$, so the gap between them stays the same, and they remain before or after whichever runs they were previously before or after, so the gaps from those singletons to other singletons also do not change. For the run track, it is in a singleton state at both steps, just with a different singleton in the top right, so the gap between the two runs on either side of the singletons stays the same, because both before and after the promotion step, 1 number gets skipped for one of the singletons and 2 numbers get skipped for the other, so either way, 3 total numbers get skipped on the run track for those two singletons.
    \end{itemize}
    This covers all cases, completing the proof of Lemma \ref{lem:nearhook_num_runs}
\end{proof}


\begin{lemma}\label{lem:nearhook_num_rotations}
    Given a set of singleton gaps and run gaps, the number of legal ways to choose of a rotation of the singleton track and a rotation of the run track is given by the quadratic polynomial in Theorem \ref{thm:hook_box_quadratic}.
\end{lemma}

\begin{proof}
    Since the singleton track always has length $n-2r-1$ with no state transitions, it has $n-2r-1$ rotations. For the run track, the run states look exactly the same as in Theorem \ref{thm:hook_box_linear_2} except that we replace $n$ with $n-2s$ because of the $2s$ skipped numbers, for the $s$ singletons and the $s$ spaces behind them. In Theorem \ref{thm:hook_box_linear_2}, the number of such states was $(2r-1)n-2r$, so in this case there are $(2r-1)(n-2s)-2r$ run states by the same reasoning. In the singleton states, the run track has length $n-2s$, and all $n-2s$ rotations are possible, so since there are $s$ choices for which singleton could be in (or right before) the entry in the top right box, there are $s(n-2s)$ total rotations for the run track in singleton states. Putting this together, we get $(2r+s-1)(n-2s)-2r$ total ways to choose a state and a rotation for the run track. Thus, we can choose rotations for both tracks in $(n-2r-1)((2r+s-1)(n-2s)-2r)$ ways, explaining the first part of the expression in Theorem \ref{thm:hook_box_quadratic}.

    To explain the subtracted term, we need to describe which pairs of track rotations cannot actually happen, and show that all other pairs of rotations can happen. With a  few exceptions, the illegal rotations are when the run track is at a position either 1 or 2 steps after the end of a run has crossed the boundary, while the singleton track is at a position where a singleton has just crossed the boundary or is 1 step past crossing the boundary. 
    
    To see this, note that if the singleton is nested within the run, it must finish crossing the boundary and then move an extra step before the last dot of the run can finish crossing the boundary, and the run step will stay still while the singleton is crossing and for the step after, since the singleton and the extra space behind it are nested within the run. 
    
    If the singleton is not nested within the run, then there must be at least 2 extra spaces after the end of the run and before the singleton, so the singleton cannot cross the boundary immediately after the run finishes crossing, or it would be too close to the run. Instead, the singleton track stays still for the 2 steps after a run finishes crossing the boundary, because while the next singleton on the track would seem to be getting closer to the boundary at each step since it does on the merged track, it actually does not because the 2 extra spaces after the end of the run that are skipped on the singleton track move from the start of the track to the end of the track, so actually the number of the first singleton on the singleton track does not change. 
    
    This would seem to give a subtracted term of $2r\cdot 2s\cdot(2r+s)=2rs(4r+2s)$, since there are $r$ choices for a run and 2 choices for whether it has been 1 or 2 steps since the run finished crossing the boundary, then $s$ choices for a singleton and 2 choices for whether it has just crossed the boundary or is 1 step past that, and $2r+s$ choices for which state the run track is in. However, one of those $2r+s$ states is actually inconsistent with the desired run having just been promoted. Namely, if the next number after the run that was just promoted is the start of another run, the smallest entry of the next run cannot be in the top right box of $T$, because it is now the smallest number in $T$, so it must go in the top left box of $T$. Similarly, if the next number after the just-promoted run is a singleton, then that singleton again must be in the top left box of $T$, not the top right. Either way, we can actually eliminate one of the $2r+s$ states for the top right box, so there are only $2r+s-1$ possibilities, giving a count of $4rs(2r+s-1)=2rs(4r+2s-2).$

    However, there is one other subtlety, which is that if the last entry of the just-promoted run is currently in the top right box, then we actually only need a gap of 2 between the last entry of the run and the singleton, since we just need a gap of 3 from the last entry of the run \emph{in the column} to the singleton. Thus, it is possible in this case for the singleton to get promoted at the same time the end of the run is 2 spaces past the boundary. The run track then stays still for another step after the singleton gets promoted, so the state where the run ends 2 spaces past the boundary and the singleton is 1 step past the boundary is also legal. This is 2 states we need to add back for each pair of a run and a singleton, so we need to add back $2rs$, or take away $2rs$ from the subtracted term. This gives a subtracted term of $2rs(4r+2s-2)-2rs=2rs(4r+2s-3)$, matching the formula in Theorem \ref{thm:hook_box_quadratic}.

    To check that all other rotations are possible, and that there is a 1-to-1 correspondence between pairs of track rotations and tableaux, we need to be able to reconstruct a tableau from each legal pair of track rotations. This can be done in essentially the same way as for the 2-row case. To build the merged track, we read along the pair of tracks, adding one number to the merged track at each step. When we encounter a singleton, we add the singleton and a space behind it to the merged track while adding two extra spaces to the run track at the current position, except that if the singleton goes in the top right box or comes right before the entry in the top right box (which we can tell because we know from the current state which singleton or run end goes in the top right box), then we only add one extra space to the run track. Similarly, whenever we reach the end of a run, we add two extra spaces to the singleton track at the current position, except that if the run end is in the top right box, we only add one extra space. Then the resulting merged track tells us which entries belong in $T$, and we also know which one belongs in the top right box, so we can recover the tableau $T[n]$ from the tracks.
\end{proof}

\begin{lemma}\label{lem:nearhook_equal_orbits}
    Within the subset of $\T_{\tn{near-hook}}(n,r,s)$ corresponding to a given list of singleton gaps and run gaps, all orbits have the same length, so all orbit lengths divide the quadratic polynomial from Theorem \ref{thm:hook_box_quadratic} that counts the size of that subset.
\end{lemma}

\begin{proof}
    The argument is the same as in Lemma \ref{lem:divisor_of_num_tableaux}. Namely, all orbits of tableaux corresponding to given gap sequences are the same length, because if we just shift the starting position of the singleton track by one position while keeping starting position of the run track the same, then when we compare the two orbits, the offsets between the tracks will remain the same except at a constant number of steps where a singleton or the end of a run is crossing the boundary. Thus, the orbit lengths will be the same, so the set of tableaux for a given pair of gap sequences will be partitioned into orbits of equal length, and hence all the orbit lengths will divide the total number of such tableaux, given by the quadratic polynomial from Theorem \ref{thm:hook_box_quadratic}.
\end{proof}

Combining Lemmas \ref{lem:nearhook_num_rotations} and \ref{lem:nearhook_equal_orbits} completes the proof of Theorem \ref{thm:hook_box_quadratic}.
\end{proof}

We end with an example to illustrate the portion of the promotion process where the green run crosses the boundary for the tableau from Example \ref{ex:near-hook}:

\begin{example}
    Focusing on the green run, we begin by backing up a few steps to when the green run is all at the top of $T$ and then starts being promoted. We get the promotion steps shown below, focusing just on the relevant portion of the tableau. We also show the relevant portion of the main track and the run and singleton tracks next to each tableau. The track is in a singleton state throughout, since the top right entry is always in the middle of the run. The first entry of the run gets promoted first:
    \begin{center}
        \raisebox{-\height}{\includegraphics[width=8cm]{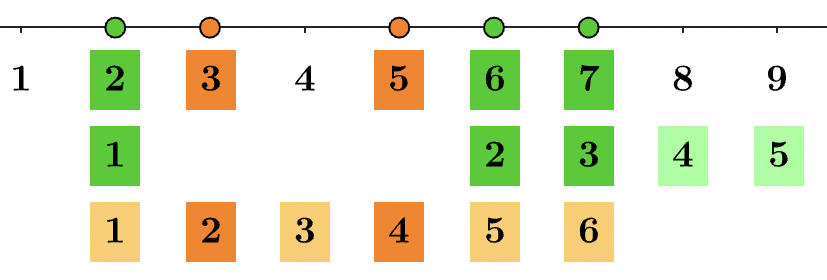}}\hspace{1cm} \begin{ytableau}
            1 & 4 & \cdots \\
            *(green) 2 & *(green) 6 \\
            *(orange) 3 \\
            *(orange) 5 \\
            *(green) 7 \\
            \vdots
        \end{ytableau}
    \end{center}
    Since $n=15$, $r=2,$ $s=3,$ and the run track is in a singleton state, the run track has length $n-2s=15-6=9,$ while the singleton track has length $n-2r-1=15-4-1=10.$ The next thing that happens is that the 2 in the run gets promoted, with both tracks rotating one position:
    \begin{center}
        \raisebox{-\height}{\includegraphics[width=8cm]{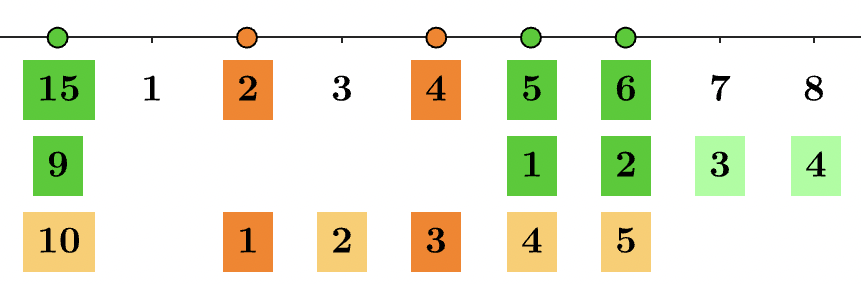}}\hspace{1cm}
        \begin{ytableau}
            1 & 3 & \cdots \\
            *(orange) 2 & *(green) 5 \\
            *(orange) 4 \\
            *(green) 6 \\
            \vdots \\
            *(green) 15
        \end{ytableau}
    \end{center}
    Next, the first singleton gets promoted, with the run track staying still:
    \begin{center}
        \raisebox{-\height}{\includegraphics[width=8cm]{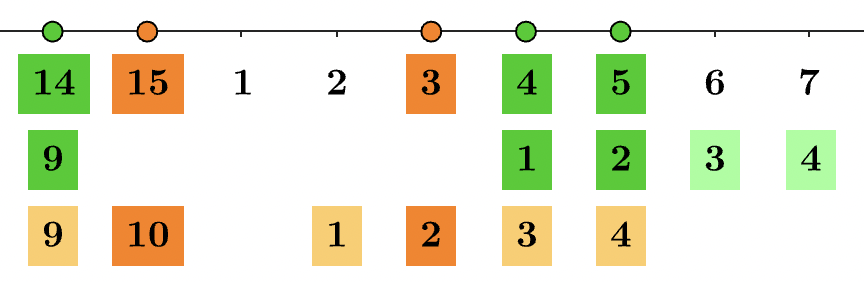}}\hspace{1cm}
        \begin{ytableau}
            1 & 2 & \cdots \\
            *(orange) 3 & *(green) 4 \\
            *(green) 5 \\
            \vdots \\
            *(green) 14 \\
            *(orange) 15
        \end{ytableau}
    \end{center}
    Then the 4 from the run track gets promoted and becomes a 15, jumping ahead of the other singleton. Both tracks rotate at this step:
    \begin{center}
        \raisebox{-\height}{\includegraphics[width=8cm]{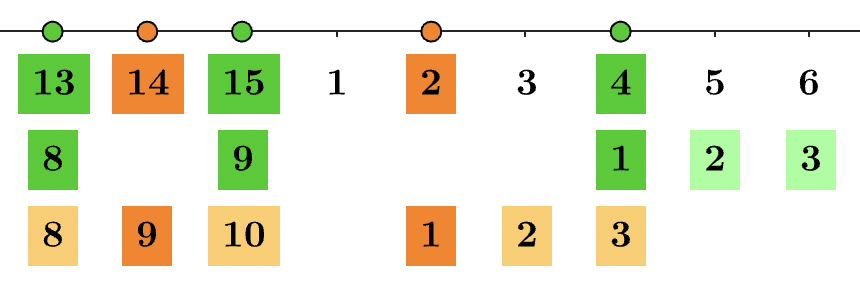}}\hspace{1cm}
        \begin{ytableau}
            1 & 3 & \cdots \\
            *(orange) 2 & *(green) 15 \\
            *(green) 4 \\
            \vdots \\
            *(green) 13 \\
            *(orange) 14
        \end{ytableau}
    \end{center}
    Then the other singleton gets promoted, with the run track staying still for that step and the following step:
    \begin{center}
        \raisebox{-\height}{\includegraphics[width=8cm]{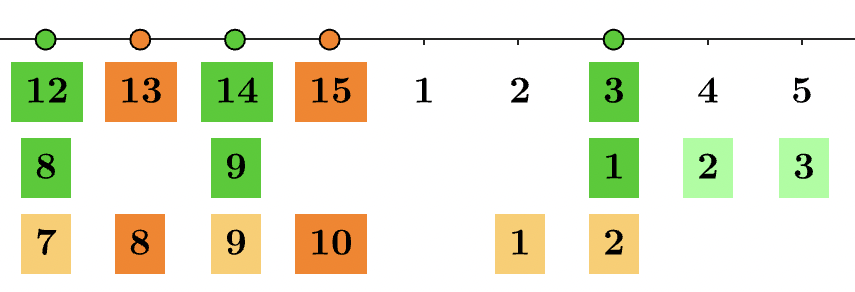}}\hspace{1cm}
        \begin{ytableau}
            1 & 2 & \cdots \\
            *(green) 3 & *(green) 14 \\
            \vdots \\
            *(green) 12 \\
            *(orange) 13 \\
            *(orange) 15
        \end{ytableau}
    \end{center}
    \begin{center}
        \raisebox{-\height}{\includegraphics[width=8cm]{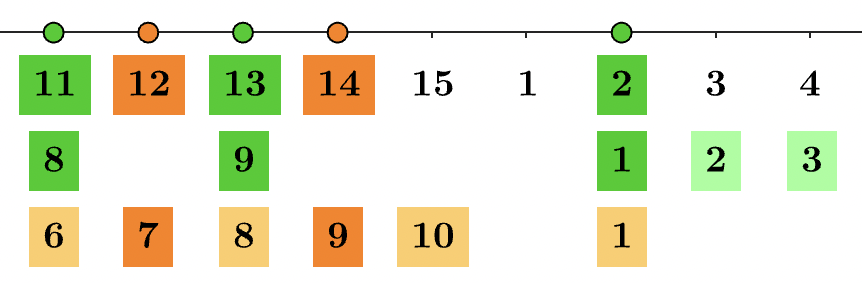}}\hspace{1cm}
        \begin{ytableau}
            1 & 3 & \cdots \\
            *(green) 2 & *(green) 13 \\
            \vdots \\
            *(green) 11 \\
            *(orange) 12 \\
            *(orange) 14
        \end{ytableau}
    \end{center}
    Then the last number of the run gets promoted, with both tracks rotating:
    \begin{center}
        \raisebox{-\height}{\includegraphics[width=8cm]{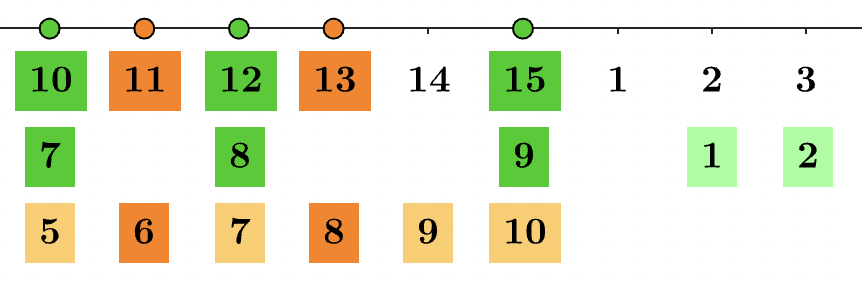}}\hspace{1cm}
        \begin{ytableau}
            1 & 2 & \cdots \\
            \vdots & *(green) 12 \\
            *(green) 10 \\
            *(orange) 11 \\
            *(orange) 13 \\
            *(green) 15
        \end{ytableau}
    \end{center}
    After that, the run track continues to rotate while the singleton track stays still for two steps:
    \begin{center}
        \raisebox{-\height}{\includegraphics[width=8cm]{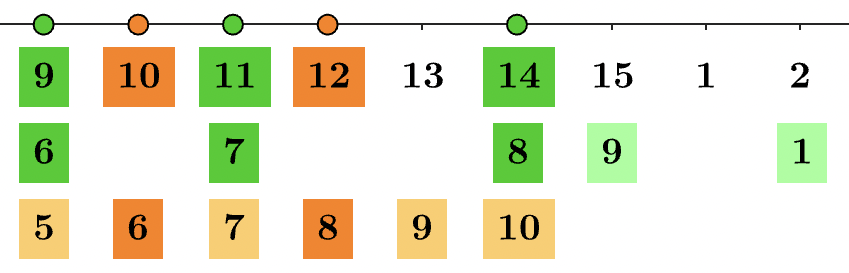}}\hspace{1cm}
        \begin{ytableau}
            1 & 2 & \cdots \\
            \vdots & *(green) 11 \\
            *(green) 9 \\
            *(orange) 10 \\
            *(orange) 12 \\
            *(green) 14
        \end{ytableau}
    \end{center}
    \begin{center}
        \raisebox{-\height}{\includegraphics[width=8cm]{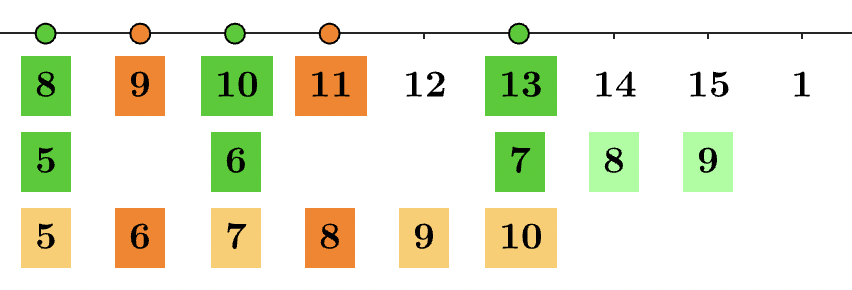}}\hspace{1cm}
        \begin{ytableau}
            1 & 3 & \cdots \\
            \vdots & *(green) 10 \\
            *(green) 8 \\
            *(orange) 9 \\
            *(orange) 11 \\
            *(green) 13
        \end{ytableau}
    \end{center}
    After that, both tracks will resume rotating as usual until the next run reaches the boundary.
\end{example}

\section*{Acknowledgments}

I'm thankful to Oliver Pechenik for getting me started on this project and for several helpful discussions about it. I'm also thankful to Jessica Striker for a helpful conversation related to this. I used Sage \cite{sage} to help formulate and test various conjectures.

\printbibliography

\end{document}